\documentclass{amsart}
\usepackage[T1, T2A]{fontenc}
\usepackage[utf8]{inputenc}
\usepackage{graphicx}
\usepackage{amssymb,latexsym}
\usepackage{subfigure}
\usepackage{tikz}

\numberwithin{equation}{section}

\def\eps{{\epsilon}}

\def\C{{\mathbb C}}
\def\D{{\mathbb D}}

\def\R{{\mathbb R}}
\def\S{{\mathbb S}}

\def\Z{{\mathbb Z}}
\def\CP{{\mathbb C\mathbb P}}

\def\ov{\overline}
\def\pl{\parallel}

\usepackage{cleveref}
\theoremstyle{plain}
\newtheorem{lemma}{Lemma}[section]
\newtheorem{proposition}[lemma]{Proposition}
\newtheorem{conjecture}[lemma]{Conjecture}

\theoremstyle{definition}

\newtheorem{remark}[lemma]{Remark}
\theoremstyle{plain}

\newtheorem{theorem}[lemma]{Theorem}

\theoremstyle{definition}
\newtheorem{definition}[lemma]{Definition}

\theoremstyle{remark}

\title[Two-parameter unfolding of a parabolic point]{Two-parameter unfolding of a parabolic point of a vector field in $\C$ fixing the origin \footnote{The author is supported by NSERC in Canada.}}

\author[C. Rousseau]{Christiane Rousseau}
\address{Christiane Rousseau, D\'epartement de
math\'ematiques et de statistique, Universit\'e de Montr\'eal, C.P. 6128,
Succursale Centre-ville, Montr\'eal (Qc), H3C 3J7, Canada.}
\email{rousseac@dms.umontreal.ca}

\usepackage{xcolor}
\usepackage{xspace}

\subjclass[2010]{37F75, 32M25, 32S65, 34M99} 

\begin{document}

\date{\today}

\begin{abstract} 
In this paper we describe the bifurcation diagram of the$2$-parameter family of vector fields  $\dot z = z(z^k+\eps_1z+\eps_0)$ over $\CP^1$ for $(\eps_1,\eps_0)\in \C^2$. There are two kinds of bifurcations: bifurcations of parabolic points and bifurcations of homoclinic loops through infinity. The latter are studied using the tool of the periodgon introduced in a particular case in \cite{CR}, and then generalized in \cite{KR}. We apply the results to the bifurcation diagram of a generic germ of 2-parameter analytic unfolding preserving the origin of the vector field $\dot z = z^{k+1} +o(z^{k+1})$ with a parabolic point at the origin.

\end{abstract}

\maketitle

\section{Introduction}

This paper is part of a larger program to explore the dynamics of polynomial vector fields on $\C$ depending on a small number of parameters, which appear as \lq\lq models\rq\rq\ of the unfolding of a parabolic point $\dot z = \frac{dz}{dt}=z^{k+1}+o(z^{k+1})$ of codimension $k$ (i.e. multiplicity $k+1$). The paper \cite{CR} studied the case $k=1$, and the paper \cite{KR}, the case $k=2$. When $k=1$, the singular points of $\dot z = z^{k+1} +\eps$ are located at the vertices of a regular polygon. For $k=2$, the family of vector fields $\dot z = z^{k+1}+\eps_1 z+\eps_0$ shows the transition between the case $\eps_1=0$ described in \cite{CR} and the case $\eps_0=0$ which has one singular point at the origin surrounded by $k$ singular points located  at the vertices of a regular polygon. When one moves from $\eps_0=0$ to $\eps_1=0$, the inner point moves outwards in one direction, and the $k$ outer singular points rotate monotonically in inverse directions on both sides of the direction followed by the inner point, so as to leave space for the inner point. This is in line with Khovanskii's fewnomial theory \cite{Kh}, which asserts that polynomials with very few monomials have roots with very equidistributed arguments (and hence real fewnomials have few real roots, regardless of their degree).

One motivation for studying these families of polynomial vector fields is that their time-one maps are good models (at least topologically) for the unfoldings of germs of diffeomorphisms with a parabolic fixed point of codimension $k$: 
\begin{equation}f(z) = z + z^{k+1}+o(z^{k+1}).\label{parabolic_k}\end{equation} A good topological model for a generic $1$-parameter unfolding is the time-one map of $\dot z= z^{k+1} +\eps_0$. While the full dynamics that can occur in perturbations of such a diffeomorphism can only be described in generic unfoldings depending on $k$ parameters (see \cite{R1}), it is not uncommon that a germ of codimension $k$ is only embedded in a generic family with less than $k$ parameters; and it is natural to study what particular dynamics occurs in such a generic family. A similar problem is that of studying the unfoldings of germs of diffeomorphisms with a parabolic fixed point of codimension $k$: \begin{equation}g(z) = \exp(2\pi i \, p/q) z + \frac1{q}z^{kq+1}+o(z^{kq+1}).\label{parabolic_periodic}\end{equation} The study is often done by considering $f= g^{\circ q}(z)$, which has the form of \eqref{parabolic_k} for some new $K=kq$. A significant difference though is the structurally stable fixed point, which can be kept at the origin by an analytic change of coordinate. In the case of \eqref{parabolic_periodic}, a generic $1$-parameter unfolding of $g$ yields a $1$-parameter unfolding of $f$, which is topologically the time-one map of $\dot z= z^{K+1} +\eps_1z$. 
Hence, it is natural to integrate the two cases in a $2$-parameter family and study the family of vector fields $\dot z= z^{K+1} +\eps_1z +\eps_0$. This is what has been done in \cite{KR}. A different problem is to study what occurs in generic $2$-parameter unfoldings $g_\eps$ of  $g$ defined in \eqref{parabolic_periodic}, through the $2$-parameter family $f_\eps=g_\eps^{\circ q}$. Here the dynamics is described topologically by the time-one map of the 2-parameter family of polynomial vector fields
 \begin{equation} \label{vector_field_q} 
 \dot Z =Q_\eps(Z)= Z(Z^{qk}+\eta_1Z^q+\eta_0),\qquad\eps=(\eta_1,\eta_0)\in \C^2.
 \end{equation}
This paper focuses on describing  the bifurcation diagram of the real dynamics of this family, which  is invariant under rotations of order $q$. 

Hence, changing to $(z,\eps_1,\eps_0)= \left(q^{\frac1k} Z^q, q^{\frac{k-1}{k}}\eta_1,q\eta_0\right)$ reduces the family  \eqref{vector_field_q} to the form
\begin{equation} \label{vector_field} 
 \dot z =P_\eps(z)= z(z^k+\eps_1z+\eps_0),\qquad\eps=(\eps_1,\eps_0)\in \C^2,
 \end{equation}
 and it suffices to study the bifurcation diagram of \eqref{vector_field}.
 
 Douady and Sentenac pioneered the study of polynomial vector fields on $\C$: they   introduced in \cite{DS} a two part invariant composed on the one hand of a combinatorial invariant in the form of a tree graph with $k+1$ vertices and, on the other hand,  of an analytic invariant given by a vector in $\mathbb H^k$. This invariant characterizes \emph{Douady-Sentenac generic} (or DS-generic) polynomial vector fields in $\C$, i.e. polynomial vector fields with simple points and no homoclinic loop through the pole at $\infty$: the polynomial vector field with a given Douady-Sentenac invariant is unique when monic and centered (the sum of the roots is zero). In the study of the vector field $\dot z = z^{k+1} + \eps$, the paper \cite{CR} introduced for each $\eps\neq0$, a new invariant, the \emph{periodgon}, or \emph{polygon of the periods}. The periodgon is a polygon with $k+1$ sides, one for each fixed point, which completely characterizes the polynomial vector field up to a rotation of order $k$, provided it is monic and centered. The periodgon was later generalized in \cite{KR} to all \emph{generic} polynomial vector fields $\dot z = dz/dt = P(z)$, where generic is understood in a different sense defined below. The periodgon is a polygon whose edges are given by oriented vectors corresponding to the periods of the different singular points of the vector field, in a proper order. It bounds a simply connected closed region in the Riemann surface of the time variable (which is a translation surface). It is defined on open sets of generic values in the parameter space for which:
 \begin{itemize} 
 \item all singular points are simple;
 \item given any singular point $z^*$ and $\delta$ such that $e^{i\delta}P'(z^*)\in i\R^+$, i.e. $z^*$ is a center for the rotated vector field $\dot z = e^{i\delta} P(z)$, then the domain of the center, called the \emph{periodic domain} of $z^*$,  is bounded by a unique homoclinic loop through infinity. 
 \end{itemize} 
These open sets are separated by surfaces in parameter space where the periodic domain of at least one singular point is bounded by several homoclinic loops through infinity. Geometrically, the periodgon is the image in $t$-space of the complement in $\CP^1$ of the union of the periodic domains of all singular points. 
 
 The bifurcations of a polynomial vector field on $\C$ can be of two types:
 \begin{itemize} \item Bifurcations of multiple singular points: these occur on the discriminant locus, an algebraic variety of real codimension 2;
 \item Bifurcations of homoclinic loops through $\infty$: these occur precisely when two vertices of the periodgon are linked by a horizontal segment lying inside it.
 \end{itemize} 
 
Hence the periodgon is a powerful tool to describe the bifurcation diagram. But it is not very easy to compute the periodgon: the difficulty is to determine the order of the edges. This order is the cyclic order around $\infty$ of the periodic domains at the singular points. The order changes when a periodic domain is bounded by more than one homoclinic loop through $\infty$: at these situations, the periodgon still exists, but it is not uniquely defined since several orders are possible.  Hence the present paper is also motivated by the need to better understand the periodgon through the study of the system \eqref{vector_field}.  

Among the questions we want to consider are the following:
\medskip

\noindent{\bf Question 1.} How many open sets in parameter space are needed to describe all generic vector fields? 
This seems to be less than the number of open sets for the Douady-Sentenac description. For instance, for the family $\dot z = z^{k+1} +\eps$, $2(k+1)$ open sets are needed for the Douady-Sentenac description, while one open set is enough with the periodgon description (see \cite{CR}). 	One open set with slits is conjectured to be sufficient with the periodgon description for the family $\dot z = z^{k+1} +\eps_1z+\eps_0$ (see \cite{KR}), while two are necessary for $\dot z= z^3+\eps_1z+\eps_0$ with the Douady-Sentenac point of view  (see \cite{R2}). For the family \eqref{vector_field} we conjecture that $k-1$ open sets are sufficient: these domains are exactly the ones to which we can reduce our study using the symmetries of the system. When $k$ is even, each open set has an additional slit. 

\medskip

\noindent{\bf Question 2.} Is the periodgon planar or does its projection on $\C$ have self-intersection? 
The periodgon of $\dot z = z^{k+1} +\eps$ is planar, while that of $\dot z = z^{k+1} +\eps_1z+\eps_0$ was conjectured to be planar. Here again we conjecture that the periodgon is planar. \medskip

The paper is organized as follows. In Section~\ref{sec:phase_portrait} we start the study of the bifurcation diagram via classical tools. In Section~\ref{sec:periodgon} we study the periodgon of \eqref{vector_field}. In Section~\ref{sec:bif_diag} we give the bifurcation diagram.  We end up with perspectives. 

\section{Study of the phase portrait of \eqref{vector_field}}\label{sec:phase_portrait} 

\subsection{Generalities on polynomial vector fields on $\C$} 
Let $\dot z = P(z)$ be a polynomial vector field of degree $k+1\geq 2$ on $\C$. Then the point at infinity is a pole of order $k-1$ with $k$ attracting and $k$ repelling separatrices. Any finite simple singular point $z_j$ is either a radial node, or a focus, or a center. In particular, it is a center if $P'(z_j)\in i\R$. Since the periodic domain of a center is always bounded by homoclinic loop(s) through infinity, 
the only bifurcations that can occur are: 
\begin{enumerate} 
\item Bifurcations of parabolic point (multiple point) when $P'(z_j)=0$. Then $P(z_j)= c(z-z_j)^{\ell+1} + o\left((z-z_j)^{\ell+1}\right)$, and $\ell$ is called the codimension of the parabolic point; 
\item Bifurcations of homoclinic loop through $\infty$, when an attracting separatrix and a repelling separatrix coallesce; 
\item And combinations of the previous types.
\end{enumerate}  

\subsection{Symmetries of the family}\label{sec:symmetries}
We consider the action of the transformation:
\begin{equation} (z,t, \eps_1, \eps_0)\mapsto (Z, T, \eta_1, \eta_0)= (Az, A^{-k}t, A^{k-1}\eps_1, A^k\eps_0)\label{linear_transf}\end{equation}
on the vector field $\tfrac{dz}{dt} = P_\eps(z)$ in \eqref{vector_field} with $t\in\R$, changing it to $\tfrac{dZ}{dT} = P_\eta(Z)$. 
We will use the particular cases: 
\begin{itemize}
\item $A=r\in \R^+$. This rescaling allows to suppose that $(\eta_1,\eta_0)\in \S^3=\{\|\eta\|=1\}$, for the \lq\lq norm\rq\rq\ introduced below in \eqref{norm}. 
\item $A= e^{\frac{2\pi i m}k}$. This gives invariance of the vector field under rotations of order $k$, modulo reparametrization. 
\item $A= e^{\frac{\pi i (2m+1)}k}$. This gives invariance under rotations of exact order $2k$, modulo reparametrization and reversing of time. 
\end{itemize}

\begin{proposition}\label{prop:symmetry} Let $\sigma$ be the  reflection with respect to the line $e^{\frac{m}{k}\pi i}\R$. 
\begin{enumerate} 
\item Let $P_\eps$ and $P_{\eps'}$ be of the form \eqref{vector_field} for $\eps=(\eps_1,\eps_0)$ and $\eps'=(\eps_1',\eps_0')$. If \begin{equation}\begin{cases} \eps_1'=e^{-\frac{2m}{k}\pi i}\ov{\eps}_1, &m\in\Z_{2k},\\
\eps_0'=\ov{\eps}_0, \end{cases}\label{cond:symmetric}\end{equation} 
 then  the vector fields  $P_\eps$ and $P_{\eps'}$ are conjugate under the change $z\mapsto \sigma(z)$.
 \item In particular, when $\eps_0\in \R$ and $\arg(\eps_1) = -\frac{\pi m}{k}$, $m\in\Z_{2k}$, then the system is symmetric with respect to the line $e^{\frac{im\pi}k}\R$, i.e. invariant under  $z\mapsto \sigma(z)$.
\end{enumerate}
\end{proposition}
\begin{proof} 
For real $t$, the reflection $\sigma:z\mapsto Z=e^{\frac{2m}{k}\pi i}\ov z$ sends the vector field \eqref{vector_field}  to 
$$\frac{dZ}{dt}=Z(Z^k + \ov\eps_1e^{-\frac{2m}{k}\pi i}Z+\ov\eps_0).$$ 
\end{proof}

\begin{proposition} Let $\sigma$ be the  reflection with respect to the line $e^{\frac{2m+1}{2k}\pi i}\R$.\begin{enumerate} 
\item Let $P_\eps$ and $P_{\eps'}$ be of the form \eqref{vector_field} for $\eps=(\eps_1,\eps_0)$ and $\eps'=(\eps_1',\eps_0')$. If \begin{equation}\begin{cases} \eps_1'=-e^{-\frac{(2m+1)}{k}\pi i}\ov{\eps}_1,  &m\in \Z_{2k},\\
\eps_0'=-\ov{\eps}_0,\end{cases}\label{cond:reversible}\end{equation}
then $P_\eps$ and $P_{\eps'}$ are conjugate under the change $(z,t)\mapsto (\sigma(z),-t)$ 
\item In particular, when $\eps_0\in i\R$ and  $\arg(\eps_1) = -\frac{\pi}2 -  \frac{2m+1}{2k}\pi$, $m\in \Z_{2k}$, then the system is reversible with respect to the line $e^{\frac{2m+1}{2k}\pi i}\R$, i.e.  invariant under $(z,t)\mapsto (\sigma(z),-t)$.
\end{enumerate}
\end{proposition}
\begin{proof} 
For real $t$, the reflection $\sigma:z\mapsto Z=e^{\frac{2m+1}{k}\pi i}\ov z$ and the time reversal $t\mapsto T=-t$ sends the vector field \eqref{vector_field}  to 
$$\frac{dZ}{dT}=Z(Z^k - \ov\eps_1e^{-\frac{2m+1}{k}\pi i}Z -\ov\eps_0).$$ 
\end{proof}

\subsection{The conic structure of the bifurcation diagram and normalizations}

The bifurcation diagram  of the real phase portrait of \eqref{vector_field} has a conic structure provided by the rescaling \eqref{linear_transf} for $A\in \R^+$. 
It is therefore sufficient to describe its intersection with a sphere $\S^3=\{\|\eps\|=1\}$, where
\begin{equation}\label{norm}
\|\eps\|=\left(\frac{|\eps_0|}{k-1}\right)^{\frac{1}{k}}+\left(\frac{|\eps_1|}{k}\right)^{\frac{1}{k-1}}.
\end{equation}
If $\eps\neq 0$, we can scale $z$ so that $\|\eps\|=1$ in \eqref{vector_field}. Then it is natural to write $|\eps_1|=k(1-s)^{k-1}$ and $|\eps_0|=(k-1)s^k$, with $s\in[0,1]$. The sphere $\S^3=\{\|\eps\|=1\}$ can be parameterized by three real coordinates: one radial coordinate $s\in [0,1]$, and two angular coordinates, namely the arguments of $\eps_0$ and $\eps_1$. 
But both arguments act on the position of the singular points. Hence we rather choose one angular parameter $\theta$ that will control the relative position of singular points, and a second angular parameter $\alpha$ that will be a rotational parameter, namely we write the system in the form:
\begin{equation}\label{3_par}
\dot z = z\left(z^k -k(1-s)^{k-1}e^{-i (k-1) \alpha} z+(k-1)s^ke^{i(\theta-k\alpha)}\right).
\end{equation} 
This corresponds to 
\begin{equation} {\eps_0}=(k-1)s^ke^{i(\theta-k\alpha)},\quad {\eps_1}=-k(1-s)^{k-1}e^{-i (k-1)\alpha},\quad \|\eps\|=1,\label{angular_coordinates}\end{equation}
 with $\theta, \alpha\in[0,2\pi]$,
For all parameter values the system has a singular point at $z_0=0$. The two extreme values $s=0$ and $s=1$ correspond to the $1$-parameter vector fields $\dot z = z(z^k+\eps_1 z)$ and
$\dot z= z(z^k +\eps_0)$. In the latter there are $k$ singular points at the vertices of a regular $k$-gon and one singular point at the origin, while in the former there are $k-1$ singular points at the vertices of a regular $(k-1)$-gon and a double singular point at the origin. Moving $s$ from $0$ to $1$ is the transition from one to the other. The singular point $z_0$ is double when $\eps_0=0$, corresponding to $s=0$ because it has merged with an extra singular point $z_1$. The parameter $s$ controls the migration of $z_1$ outwards when moving from $\eps_0=0$  to $\eps_1=0$. When $s$ increases, the movement of the outer singular points $z_2, \dots, z_k$ is very smooth so as to create the exact needed space for the inner singular point $z_1$ moving outwards.
The parameter $\theta$ determines the direction in which $z_1$ moves outwards. The parameter $\alpha$ is a rotation parameter. It plays no role in the relative position of the singular points (these rotate as a rigid solid). On the other hand,  it is responsible for the monotonic movements of the separatrices of the pole at infinity producing the bifurcations of homoclinic loops.

Using  $s$ and $\theta$ as polar coordinates, we will describe the dynamics over the parameter disk $|se^{i\theta}|\leq 1$.

\subsection{Geometry of the parameter space}\label{sec:Geometry}
The parameter space is the 3-sphere $\S^3=\{\|\eps\|=1\}$, which is a quotient of $[0,1]\times (\S^1)^2$, on which we use coordinates $(s,\theta,\alpha)$ defined in \eqref{angular_coordinates}. The quotient consists in identifying 
\begin{align}\begin{split}
(s,\theta,\alpha)&\sim(s,\theta+2\pi,\alpha)\sim(s,\theta,\alpha+2\pi)\sim(s,\theta+\tfrac{2\pi}{k-1},\alpha+\tfrac{2\pi}{k-1}),\\
(0,\theta,\alpha)&\sim(0,0,\alpha)\sim(0,0,\alpha+\tfrac{2\pi}{k-1}),\\ 
(1,\theta,\alpha)&\sim(1, 0,\alpha-\tfrac{\theta}{k})\sim(1, 0,\alpha-\tfrac{\theta}{k}+\tfrac{2\pi}{k}),\end{split}\label{rel:quotient}
\end{align}
for all $s,\theta,\alpha$.
We naturally find a generalization of the  Hopf fibration of $\S^3$ over $\S^2$ given by the projection $(s,\theta,\alpha)\mapsto(s,\theta\mod \frac{2\pi}{k} )$, with $\S^2$ being the quotient of $[0,1]\times \S^1$ by identifying $(s,\theta)\sim (s,\theta+\frac{2\pi}{k})$ for all $s,\theta$, and $(0,\theta)\sim (0,0)$, $(1,\theta)\sim (1, 0)$ for all $\theta$.  Here $s\in (0,1)$ parametrizes a family of tori in $\S^3$, where each torus is filled by a family of $(k,k-1)$-torus knots, 
each knot corresponding to constant $(s,\theta)=(s_0,\theta_0)$  and being parametrized by $\alpha$. 
For $s=0$, the torus degenerates to a circle parametrized by $\alpha$ and covered $k-1$ times, and for $s=1$, the torus degenerates to a circle parameterized by $\alpha$ and covered $k$ times.

The only bifurcations are homoclinic connections of two separatrices of $\infty$, and the two bifurcations of parabolic point. 
The former, of real codimension $1$, is studied through the periodgon. Several bifurcations can occur simultaneously, yielding higher order bifurcations. The boundaries of the surfaces of homoclinic connections in parameter space occur along the higher order bifurcations, including the parabolic point bifurcations.

\subsection{Bifurcation of parabolic points}

Parabolic points are important because they organize the dynamics and the bifurcations of homoclinic loops. A parabolic point of codimension $\ell$ has $2\ell$ \emph{sepal zones}, i.e.  connected regions  filled by trajectories having their $\alpha$-limit and $\omega$-limit at the parabolic point.
These sepal zones go to $\infty$ and coincide with $2\ell$ saddle sectors of $\infty$. Hence, the corresponding sectors are called the \emph{sepal sectors} of $\infty$ (see Figure~\ref{elliptic_sectors}). The sepal zones are bounded by separatrices. Each boundary of a sepal zone is the limit of a homoclinic loop through $\infty$ that circles around a unique singular point.   Note that there are always an even number of non sepal sectors of infinity between two sepal sectors  (see Figure~\ref{parabolic_cod1_origin}).

\begin{figure}\begin{center}
\includegraphics[width=5cm]{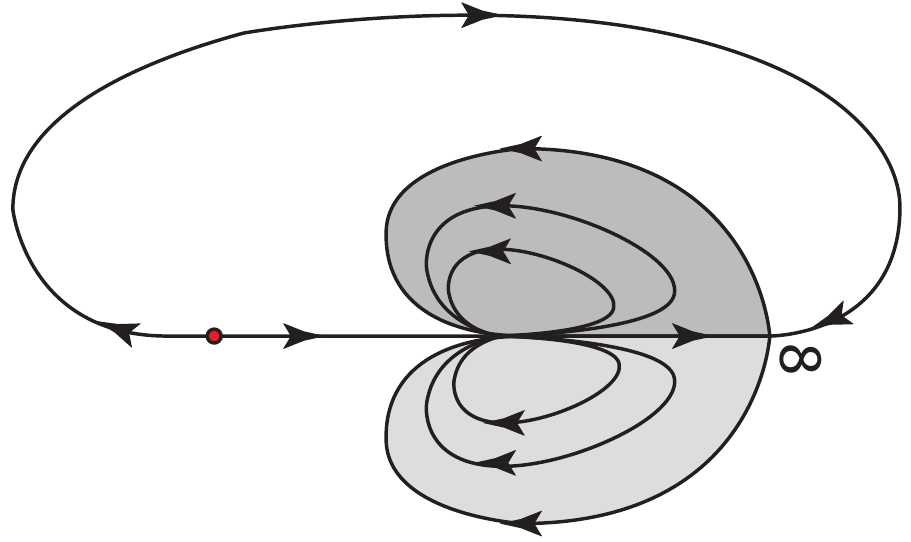}\caption{The two sepal zones of the parabolic point of $\dot z = z^3+z^2$.}\label{elliptic_sectors}\end{center}\end{figure}

\begin{proposition} \begin{enumerate}
\item There are two kinds of bifurcations of parabolic point of codimension 1. 
\begin{itemize} 
\item Bifurcation of parabolic point at the origin when $\eps_0=0$ (see Figure~\ref{parabolic_cod1_origin}). 
The two sepal zones at the parabolic point correspond to two almost opposite sepal sectors at $\infty$, namely if $m_1$ and $m_2$ are the number of adjacent non sepal sectors between the two sepal sectors, then $|m_1-m_2|\leq 2$. 
\item The other type occurs when the discriminant of $z^k+\eps_1z+\eps_0$ vanishes and $\eps_0,\eps_1\neq0$. The discriminant is given by 
\begin{equation} 
\Delta(\eps_1,\eps_0)
=(-1)^{\lfloor\frac{k}2\rfloor}(k-1)^{k-1}k^k\left[\big(\tfrac{\eps_0}{k-1}\big)^{k-1}-\big(-\tfrac{\eps_1}{k}\big)^k\right].\label{Delta}
\end{equation}
The intersection of $\Delta=0$ with the sphere $\S^3$ is a $(k,k-1)$ torus knot. 
Using \eqref{angular_coordinates}, $\Delta$ vanishes if and only if $s=\frac12$ and $\theta=\frac{2\pi j}{k-1}$,
in which case $z=\frac12e^{\frac{2\pi j}{k-1}}$ is a generic parabolic point (double root) of \eqref{vector_field}  (see Figure~\ref{parabolic_cod1}). 
The two sepal zones at the parabolic point correspond to two adjacent sepal sectors at $\infty$. 
\end{itemize}
\item A bifurcation of parabolic point of codimension $k$ occurs when $\eps_0=\eps_1=0$, in which case the origin is a singular point of multiplicity $k+1$. \end{enumerate}\end{proposition}

\begin{figure}\begin{center} 
\subfigure[$k=4$]{\includegraphics[width=3.4cm]{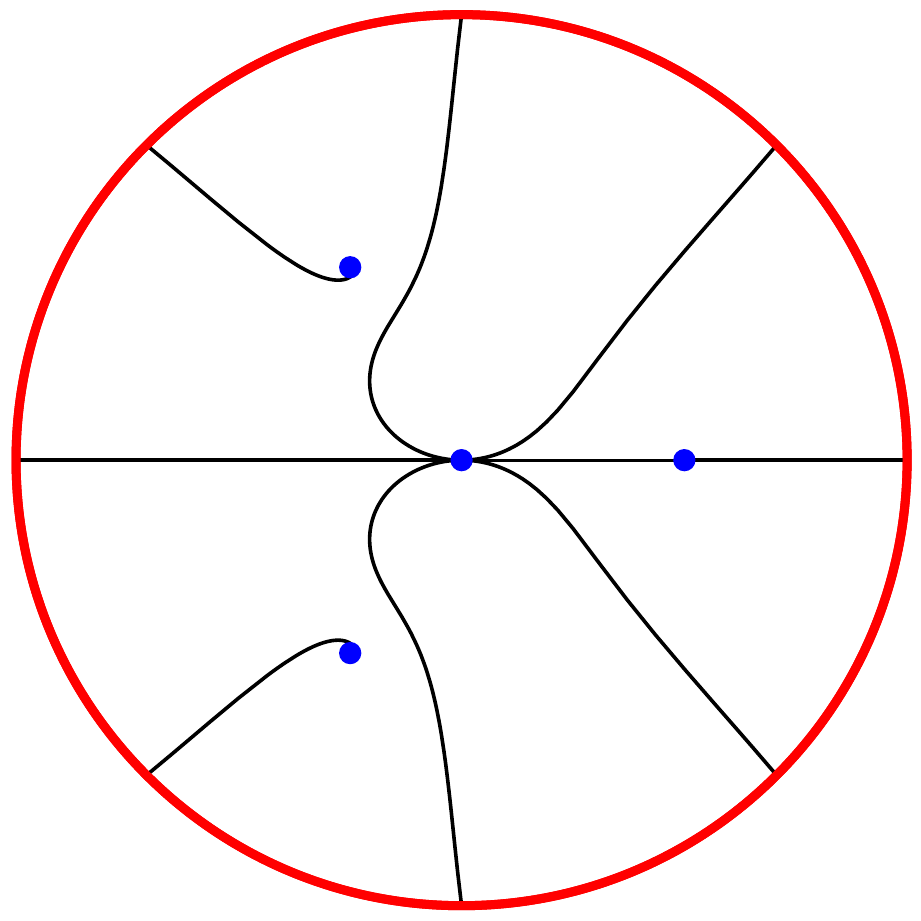}}\quad\subfigure[$k=5$]{\includegraphics[width=3.4cm]{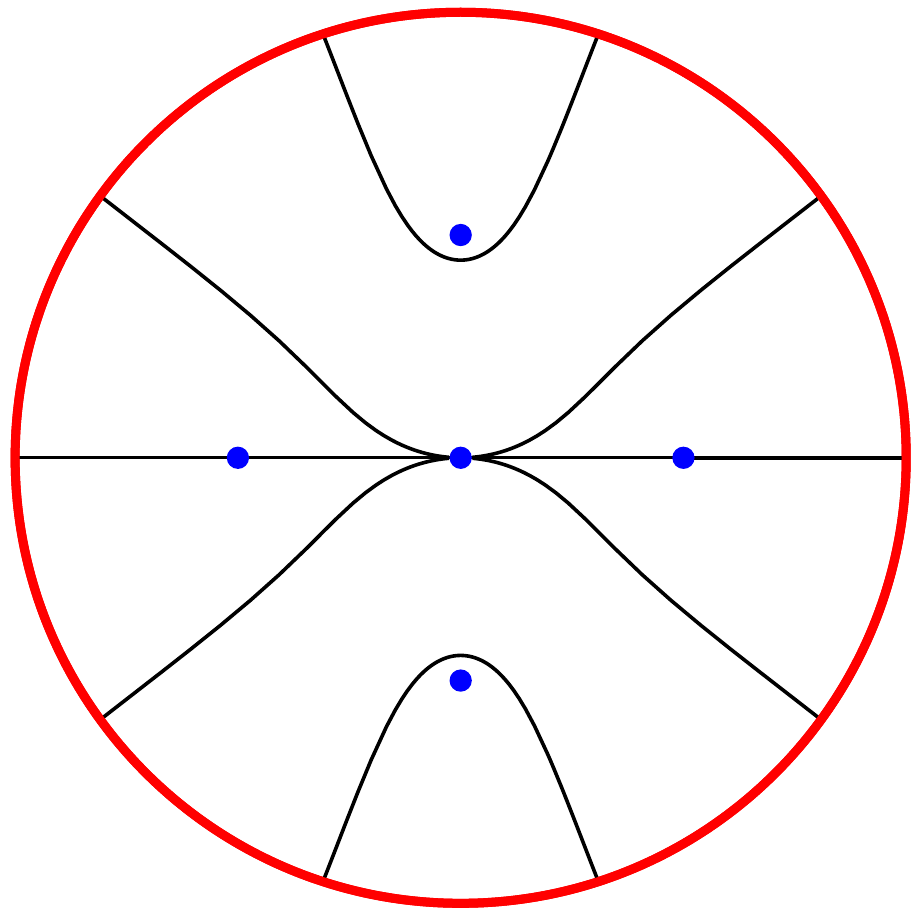}}\quad\subfigure[$k=6$]{\includegraphics[width=3.4cm]{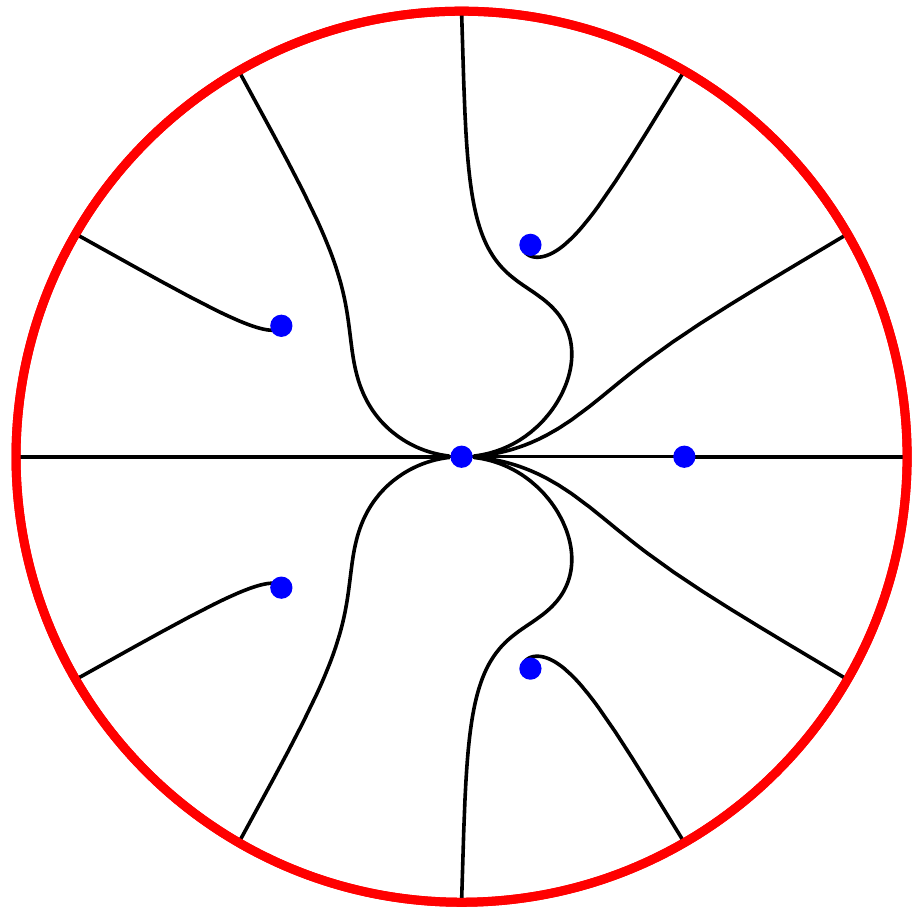}} \subfigure[$k=7$]{\includegraphics[width=3.4cm]{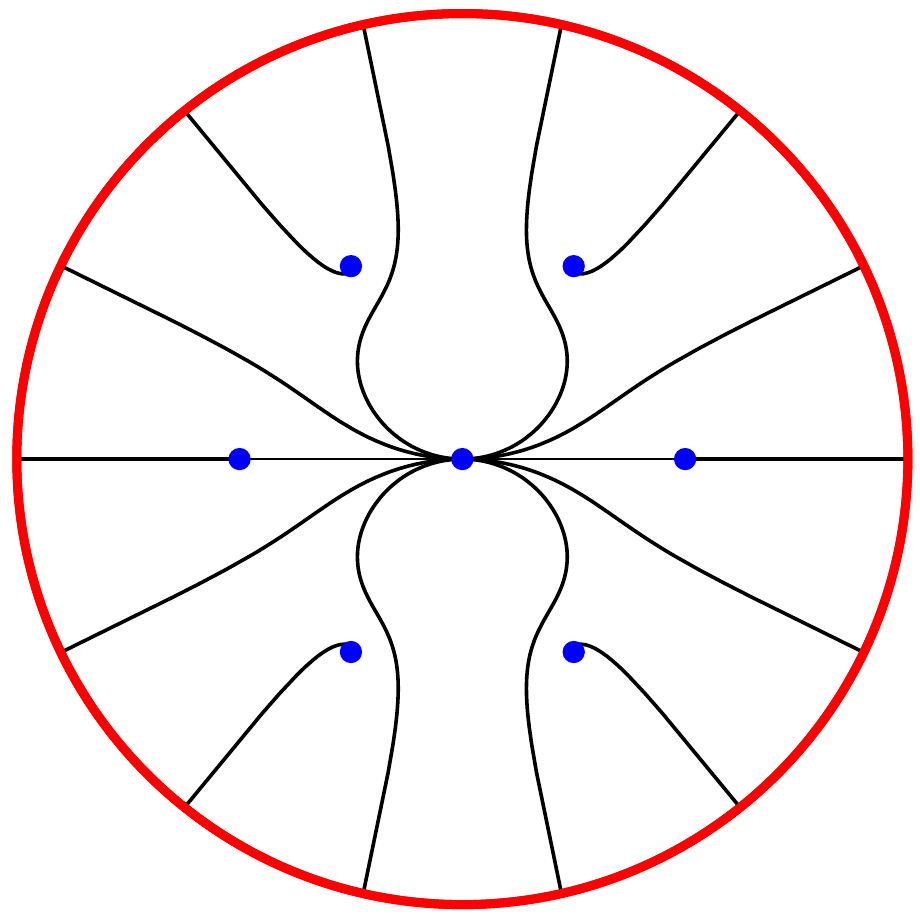}} \quad\subfigure[$k=8$]{\includegraphics[width=3.4cm]{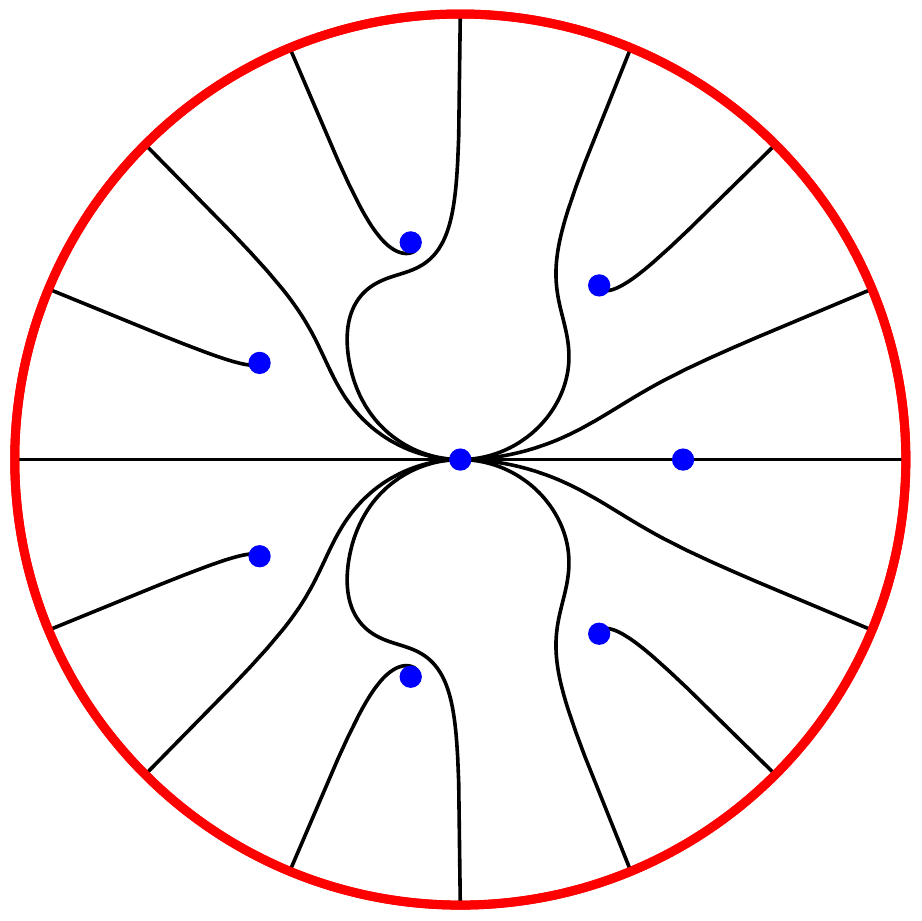}}\quad\subfigure[$k=9$]{\includegraphics[width=3.4cm]{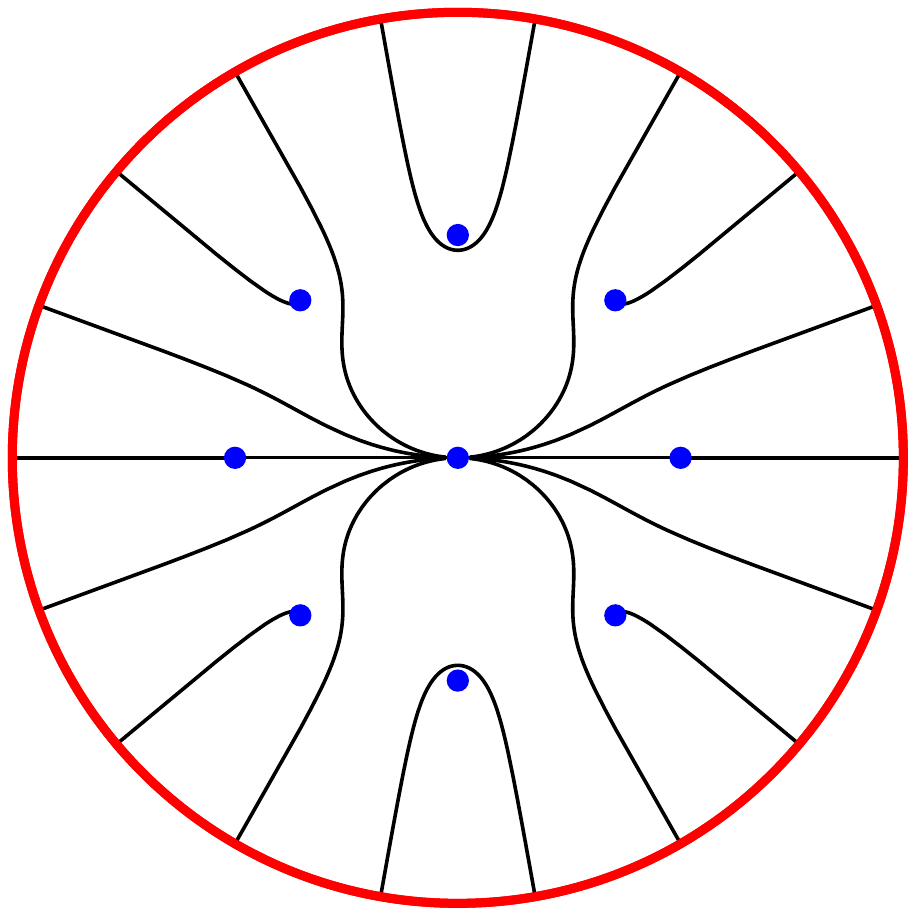}} \subfigure[$k=10$]{\includegraphics[width=3.4cm]{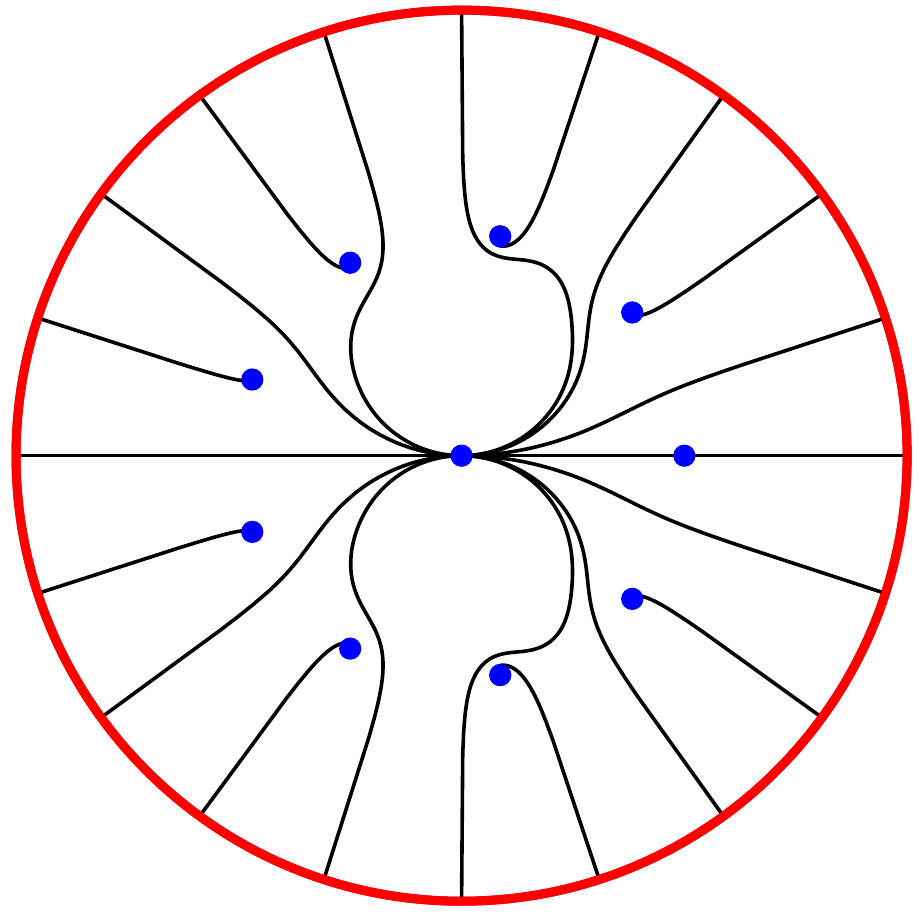}}\quad\subfigure[$k=11$]{\includegraphics[width=3.4cm]{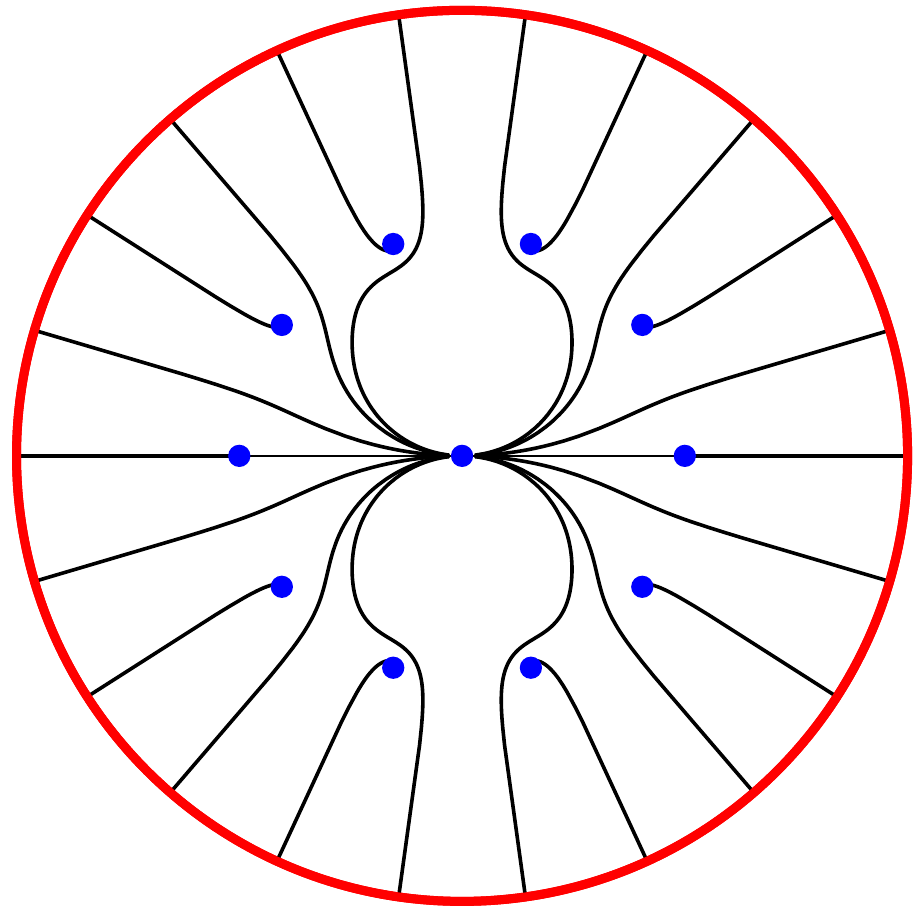}}\quad\subfigure[$k=12$]{\includegraphics[width=3.4cm]{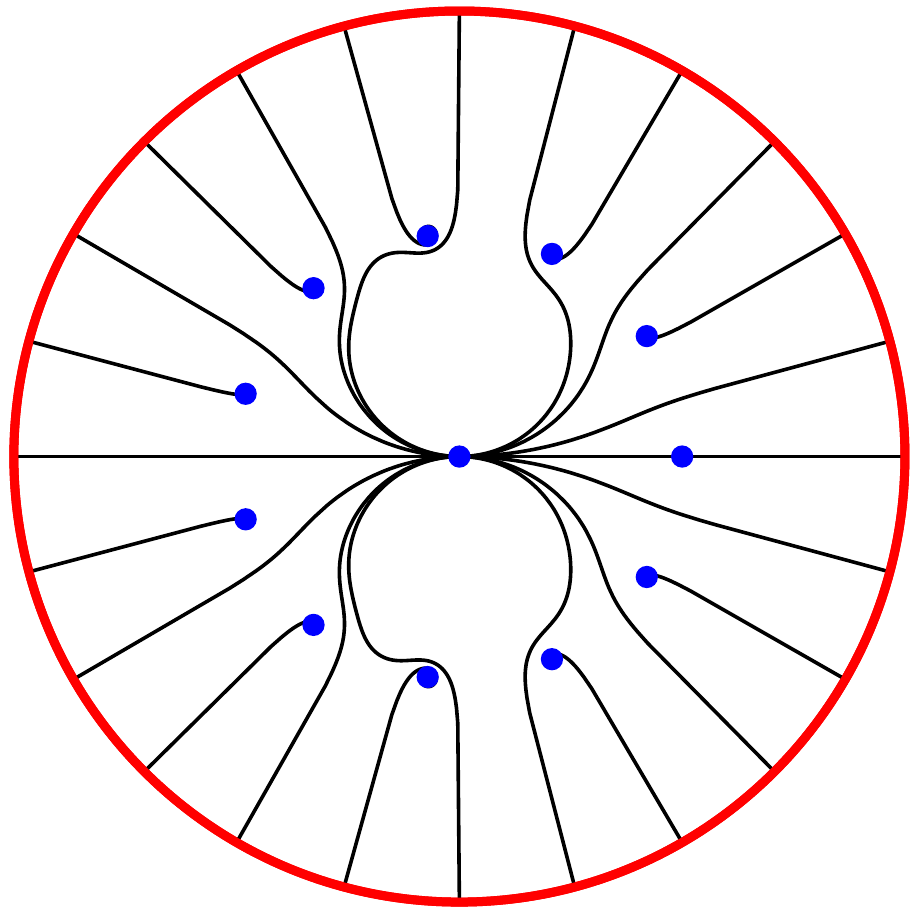}}
\caption{The phase portrait of $\dot z = z^{k+1}-z^2$ with a parabolic point of codimension $1$ at the origin.}\label{parabolic_cod1_origin}
\end{center}\end{figure}


\begin{figure}\begin{center} 
\subfigure[$k=4$]{\includegraphics[width=3.4cm]{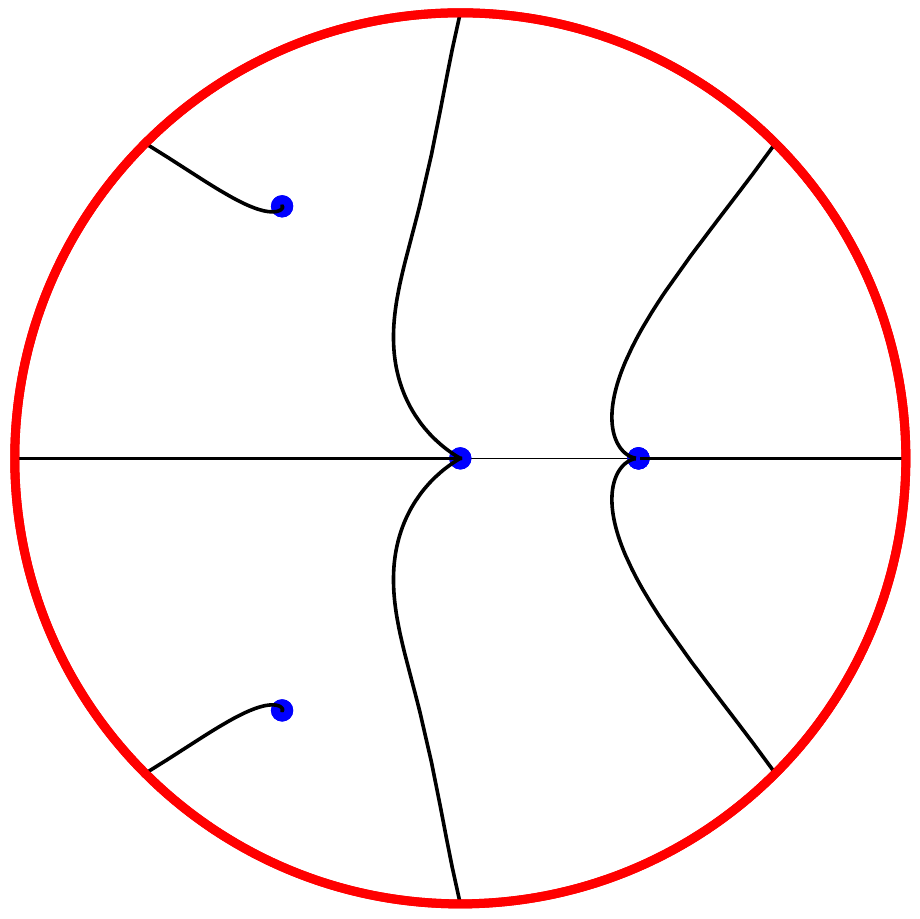}}\quad\subfigure[$k=7$]{\includegraphics[width=3.4cm]{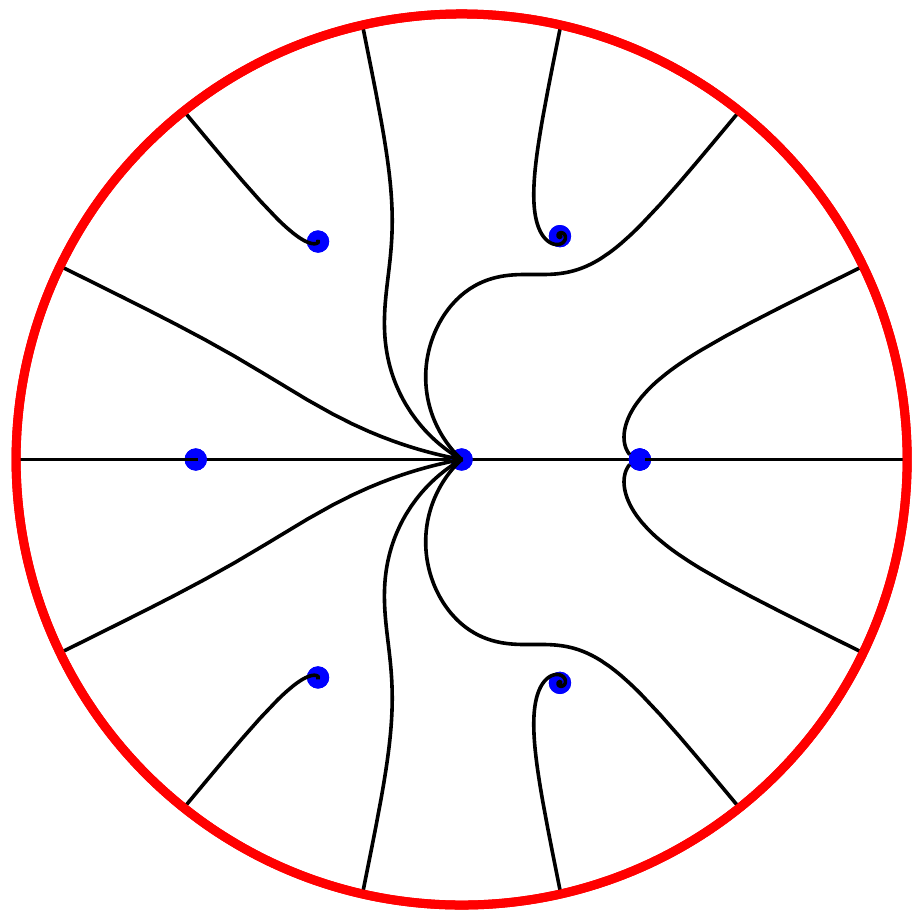}}\quad\subfigure[$k=9$]{\includegraphics[width=3.4cm]{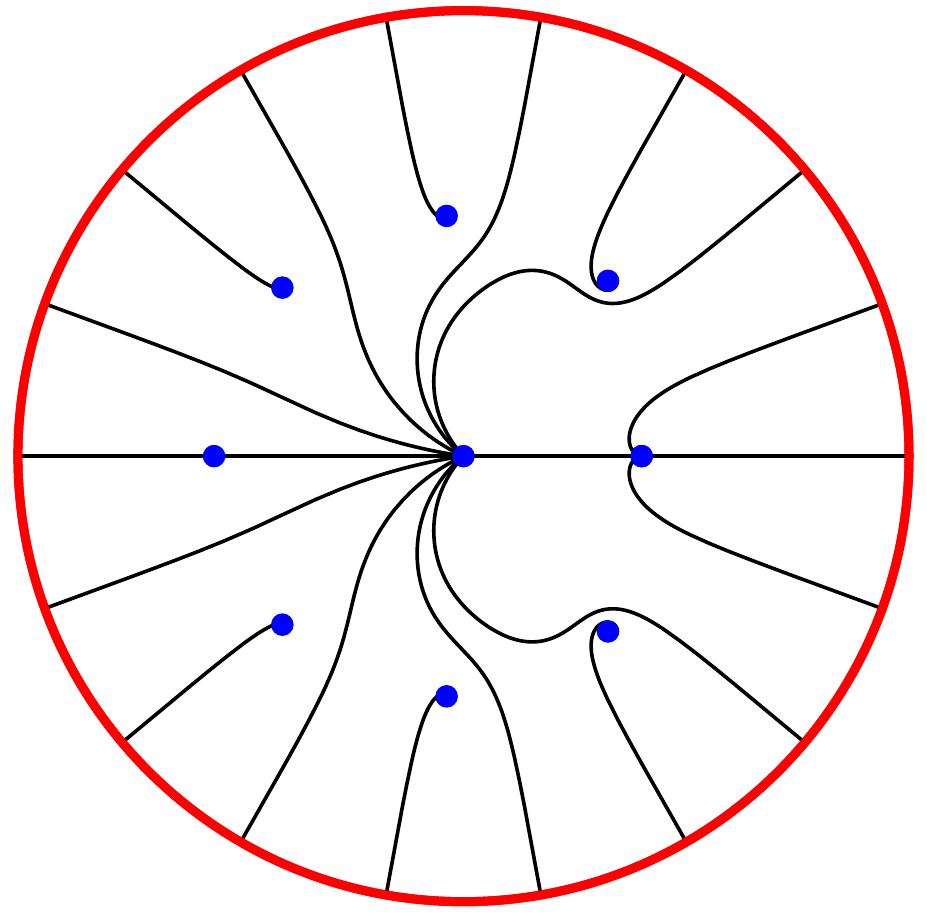}} 
\caption{The phase portrait of $\dot z = z^{k+1}-kz^2+(k-1)z$ with a parabolic point of codimension $1$ at $z=1$.}\label{parabolic_cod1}
\end{center}\end{figure}


\section{The periodgon of \eqref{vector_field}}\label{sec:periodgon}

\subsection{The definition of the periodgon of a polynomial vector field}

Let $
\dot z = \frac{dz}{dt} =P(z)$ be a polynomial vector field of degree $k+1$ with singular points $z_0,\dots, z_k$ and 
$$\nu_j=2\pi i \;{\rm Res}_{z_j} \frac1{P(z)}$$ be the period of $z_j$: this period corresponds to the \lq\lq travel time\rq\rq\ $\int_{\gamma_j} dt$ along a small loop $\gamma_j$ surrounding only $z_j$. If $z_j$ is simple, then $\nu_j= \frac{2\pi i}{P'(z_j)}$. Note that any simple equilibrium point $z_j$ is a center of the rotated vector field 
\begin{equation}\dot z = e^{i\arg \nu_j}P(z). \label{rotated_vf}\end{equation}

\begin{definition} 
Let $
\dot z = \frac{dz}{dt} =P(z)$ be a polynomial vector field and $z_j$ be a simple singular point. The \emph{periodic domain} of $z_j$  is the basin of the center (periodic zone) at $z_j$ of \eqref{rotated_vf}. The boundary of the periodic domain of $z_j$ consists in one or several homoclinic loops through the pole at infinity of \eqref{rotated_vf}, which are called the \emph{homoclinic loop(s)} of $z_j$.  \end{definition}

\begin{lemma} \cite{KR} When all singular points of a vector field $
\dot z = \frac{dz}{dt} =P(z)$ are simple, then their periodic domains are disjoint. If, moreover, they have only one homoclinic loop, then these homoclinic loops are disjoint. If some points have multiple homoclinic loops, then some homoclinic loops agree up to orientation. 
\end{lemma} 
This lemma allows defining the periodgon in the generic case where all singular points have exactly one homoclinic loop. 
The periodgon lies in the time-space $t$, a translation surface, where 
$$t= \int\frac{dz}{P(z)}.$$ The variable $t$ is well defined in the translation surface up to a tanslation as long as the variable $z$ is restricted to a simply connected domain of $\C$ containing no singular points.

\begin{definition} Let $
\dot z = \frac{dz}{dt} =P(z)$ be a polynomial vector field with simple singular points, each having exactly one homoclinic loop. Then the \emph{periodgon}  of the vector field is the image in $t$-space of the complement of the union of the periodic domains bounded by the homoclinic loops (see Figure~\ref{periodgon_def}). 
This definition has a limit in the nongeneric case where some singular point has more than one homoclinic loop (see Figures~\ref{periodgon_def_lim1} and \ref{periodgon_def_lim2}). But the order of the sides in the nongeneric case is not uniquely defined, this coming from the fact that different generic situations have the same limit. \end{definition}

\begin{figure}\begin{center}
\subfigure[Preimage of periodgon]{\includegraphics[width=4.5cm]{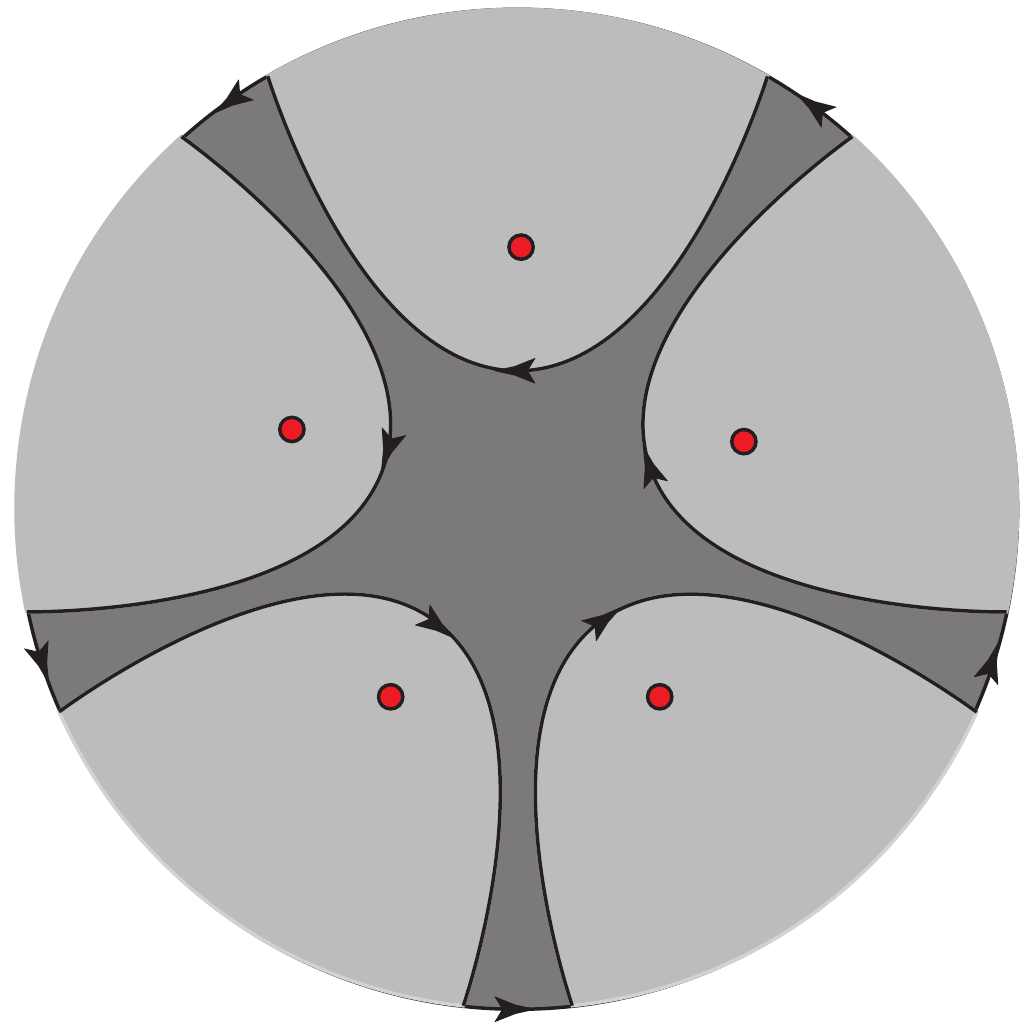}}\qquad\qquad\subfigure[Periodgon]{\includegraphics[width=3.8cm]{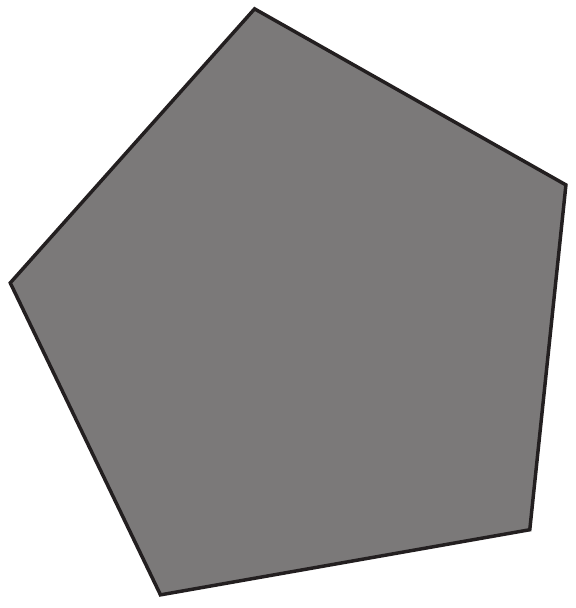}}\caption{The periodgon in $t$-space and its preimage in $z$-space.}\label{periodgon_def}\end{center}\end{figure}

\begin{figure}\begin{center}
\subfigure[Preimage]{\includegraphics[width=4cm]{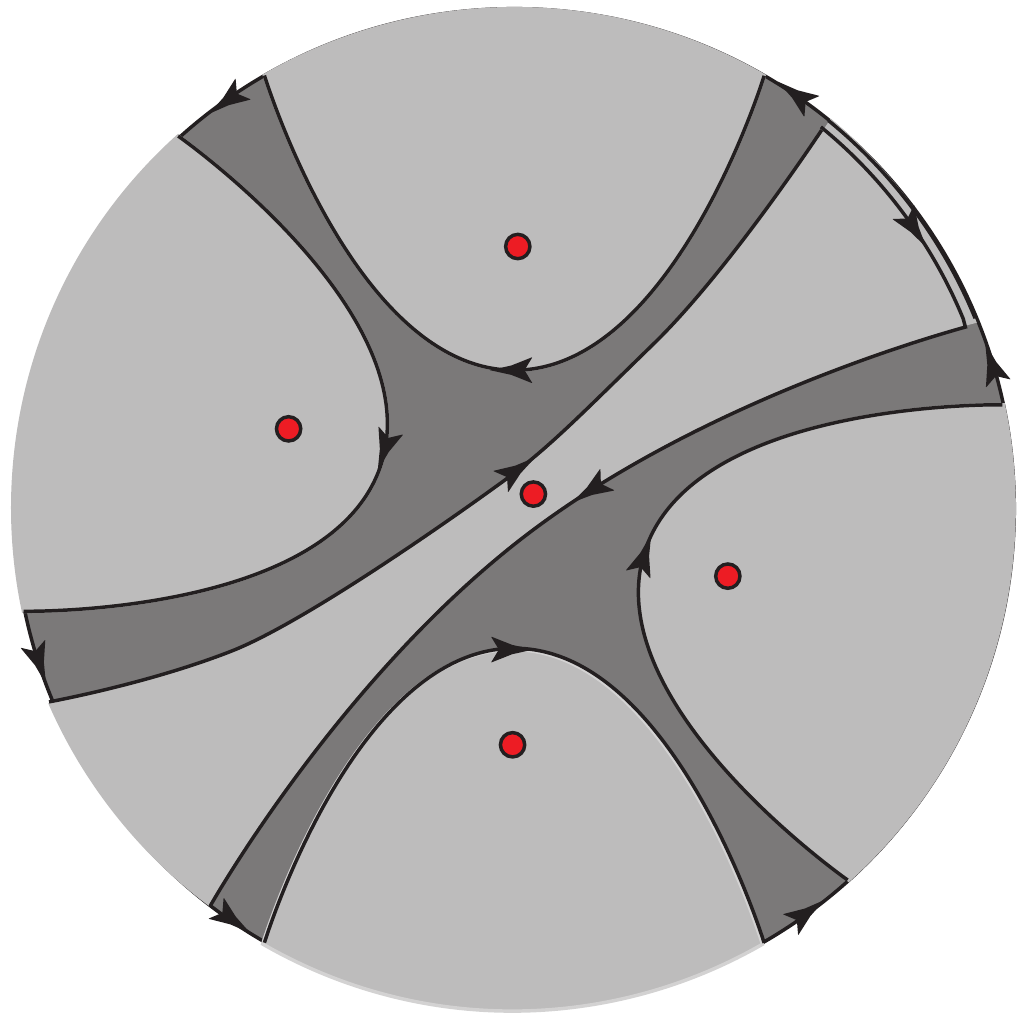}}\qquad\subfigure[Preimage]{\includegraphics[width=4cm]{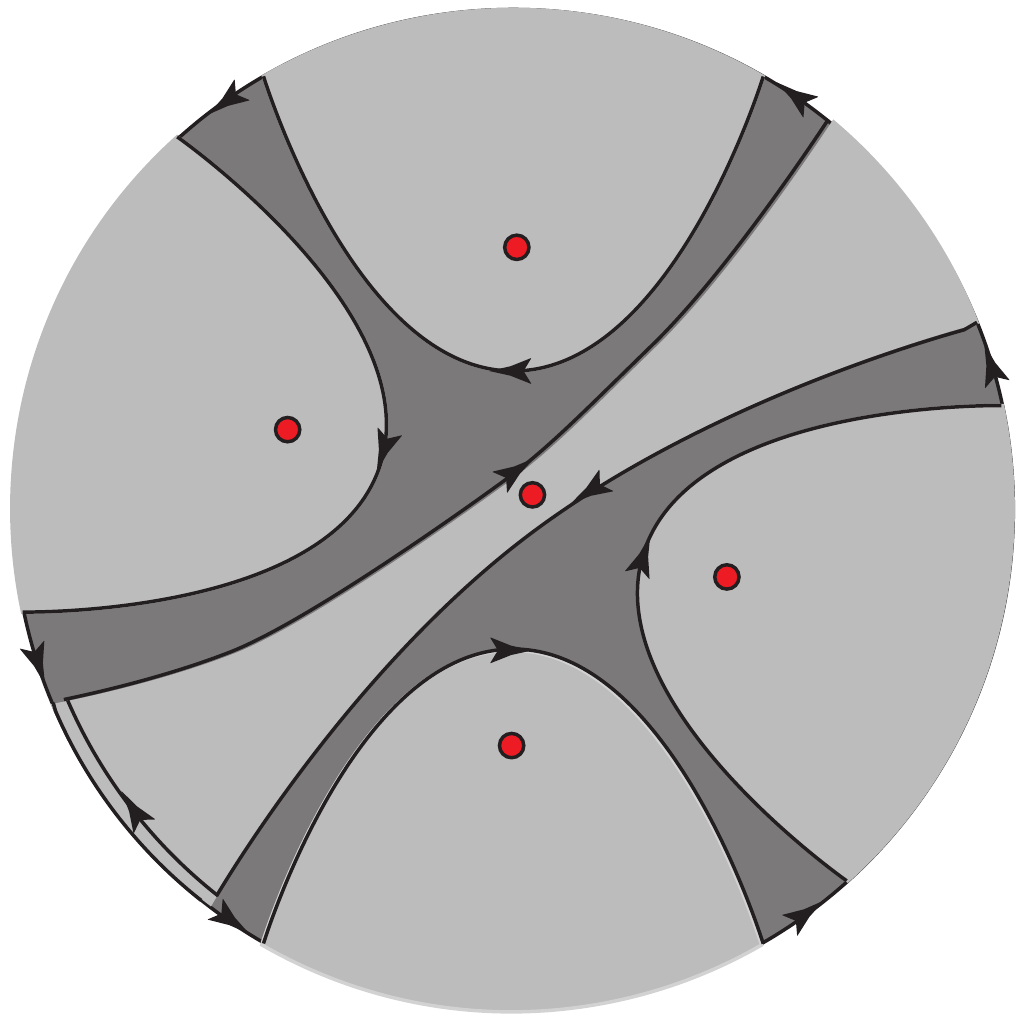}}\\
\subfigure[First periodgon]{\includegraphics[width=4cm]{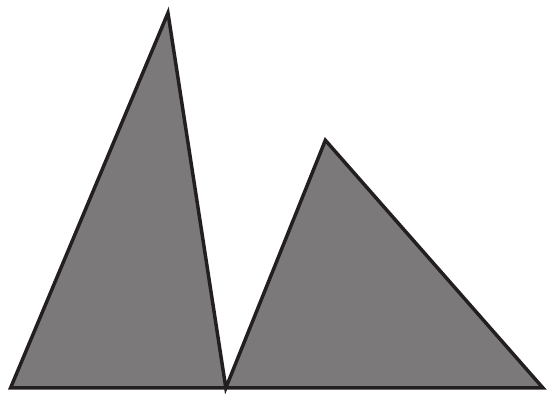}}\qquad\subfigure[Second periodgon]{\includegraphics[width=4cm]{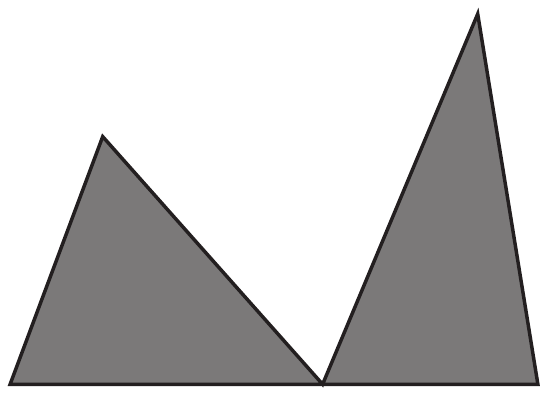}}\\
\subfigure[First unfolding]{\includegraphics[width=4cm]{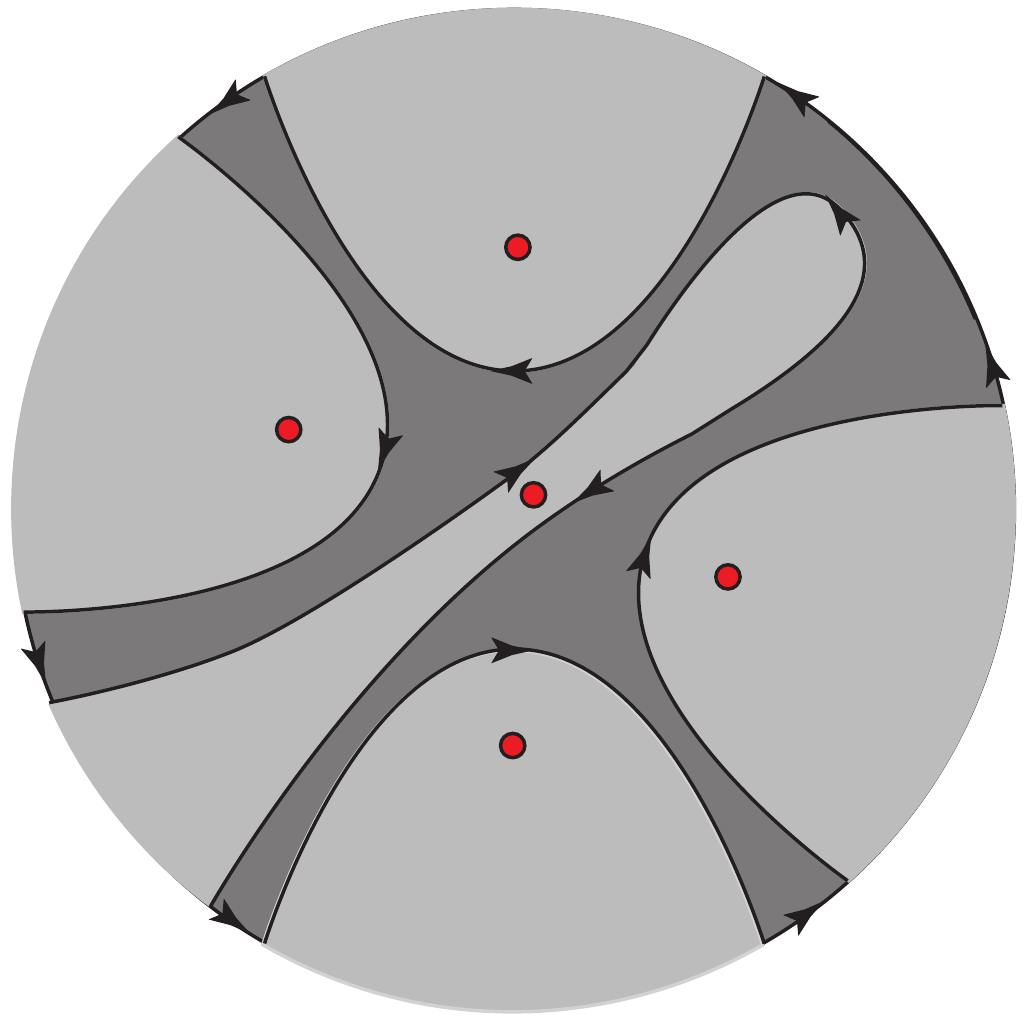}}\qquad \subfigure[Second unfolding]{\includegraphics[width=4cm]{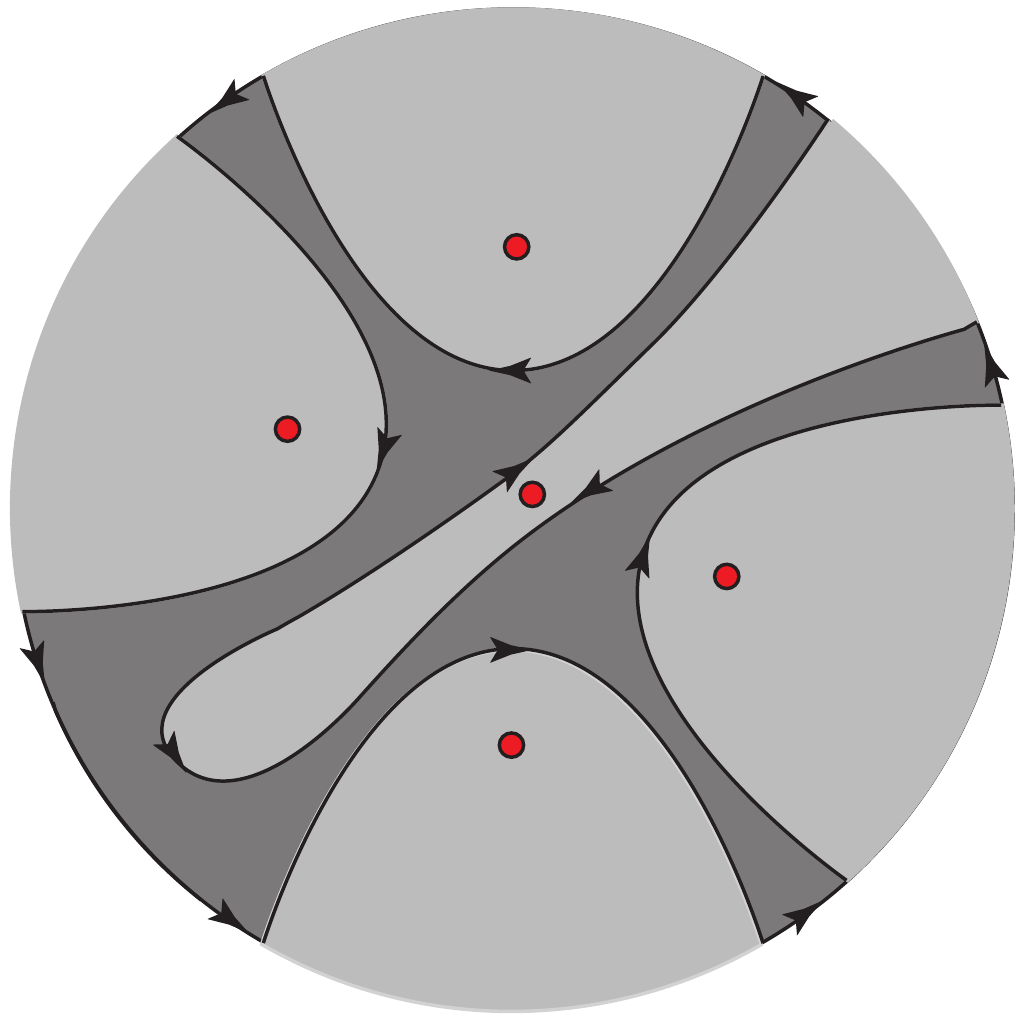}}\\
\subfigure[First unfolded periodgon]{\includegraphics[width=4cm]{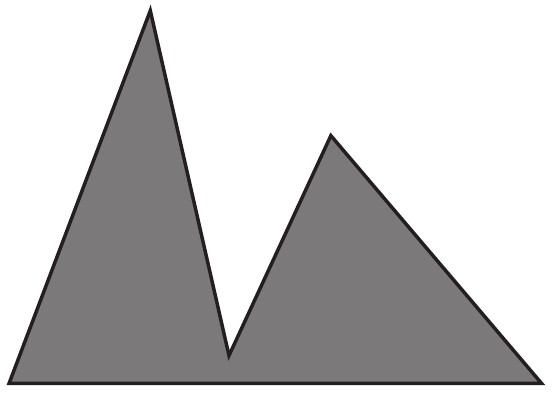}}\qquad\subfigure[Second unfolded periodgon]{\includegraphics[width=4cm]{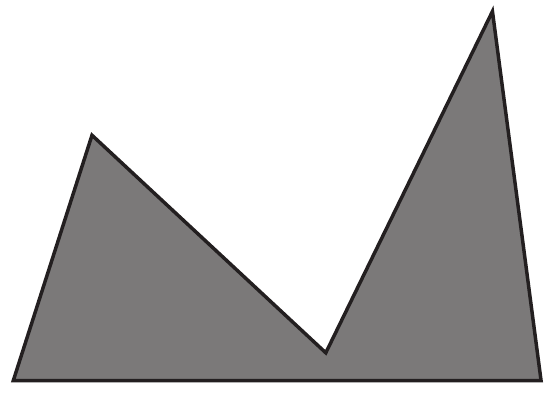}}\\
\caption{The same preimage ((a) and (b)) in $z$-space corresponds to two different nongeneric periodgons in (c) and (d). The respective unfoldings of the preimages appear in (e) and (f) and their corresponding unfolded periodgons in (g) and (h).}\label{periodgon_def_lim1}\end{center}\end{figure}

\begin{figure}\begin{center}
\subfigure[]{\includegraphics[width=3.5cm]{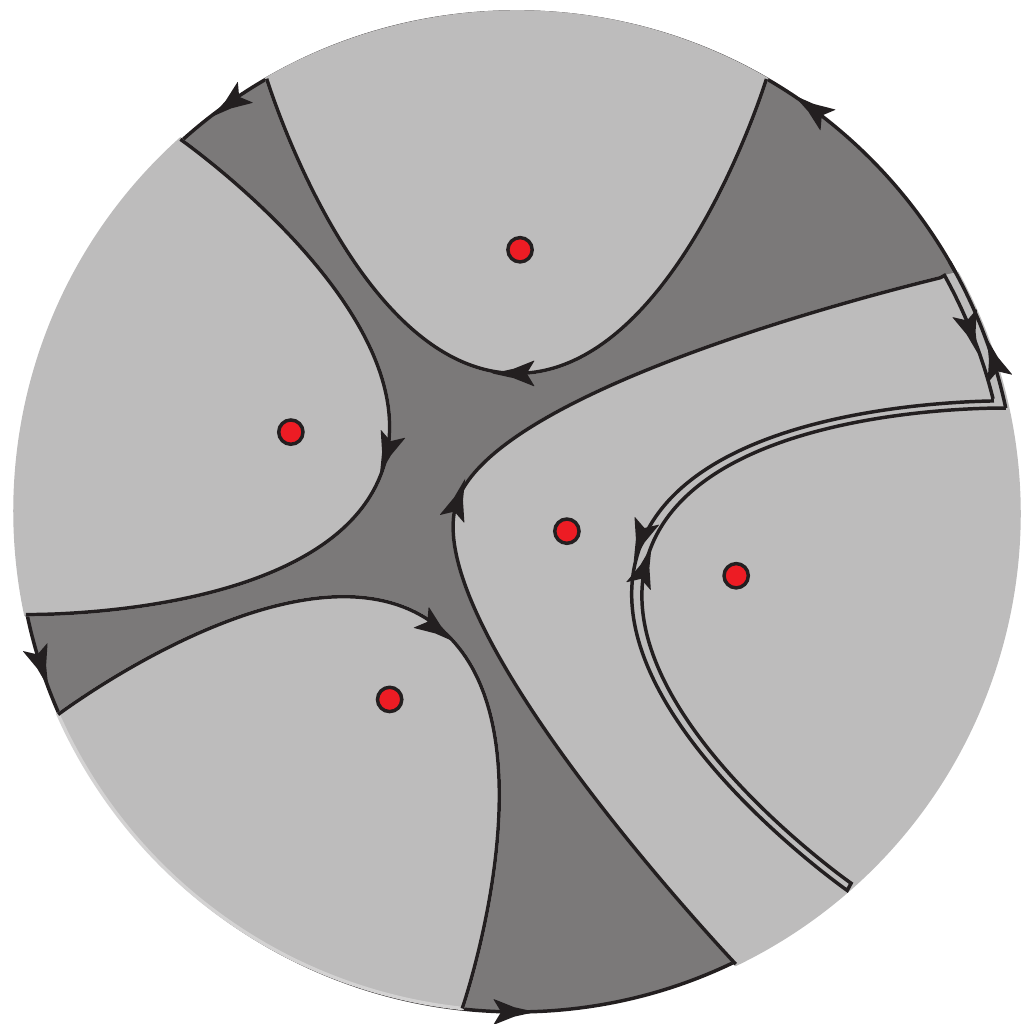}}\qquad\qquad\subfigure[]{\includegraphics[width=3.5cm]{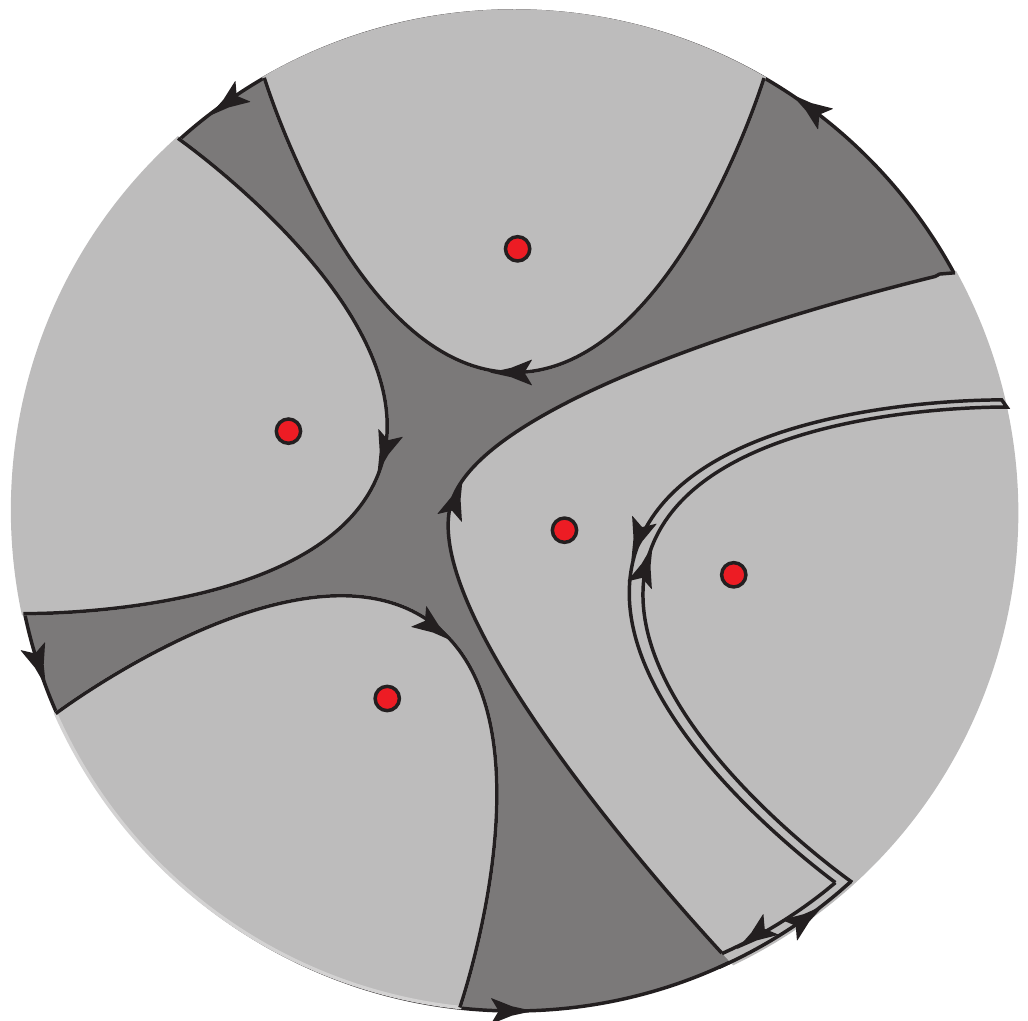}}\\
\subfigure[]{\includegraphics[width=3.5cm]{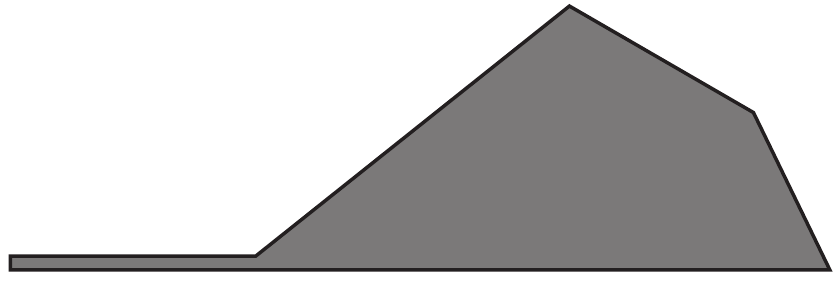}}\qquad\qquad\subfigure[]{\includegraphics[width=3.5cm]{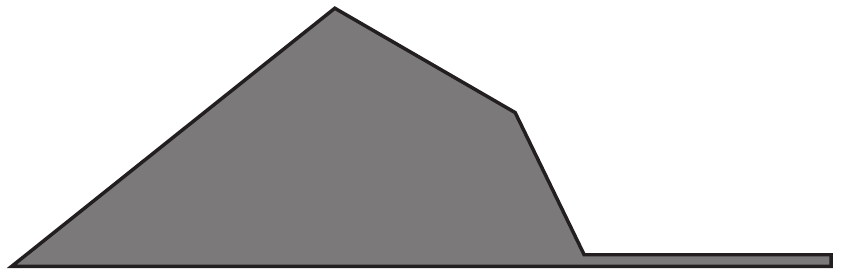}}\caption{In (a) and (b), another preimage in $z$-space corresponding to two different nongeneric periodgons in (c) and (d). The unfoldings are not drawn. }\label{periodgon_def_lim2}\end{center}\end{figure}

\subsection{The periodgon in the degenerate case}\label{sec:degenerate}

It is possible to generalize the notion of periodgon in the degenerate case when some singular point is parabolic (see \cite{KR}). For that, we need to define the parabolic domain of a parabolic point.

\begin{definition}\label{def:sepal} \begin{enumerate} 
\item A \emph{sepal zone} of a parabolic point is a connected component of the complement of the union of the separatrices, which is filled by trajectories having their $\alpha$-limit and $\omega$-limit at the parabolic point.
\item The \emph{parabolic domain} of a parabolic point $z_0$ of a polynomial vector field $\dot z = P(z)$ is the union of all sepal zones of $z_0$ in all rotated vector fields $e^{i\beta} P(z)$ (see Figures~\ref{fig:parab_domain1} and \ref{fig:parab_domain2}). 
\item The boundaries of the periodic and parabolic domain of the different singular points intersect only at infinity and the periodgon is the image in $t$-space of the complement of the union of the periodic and parabolic domains. \end{enumerate}\end{definition}

\begin{figure}\begin{center}
\subfigure[$\alpha=0$]{\includegraphics[width=3.5cm]{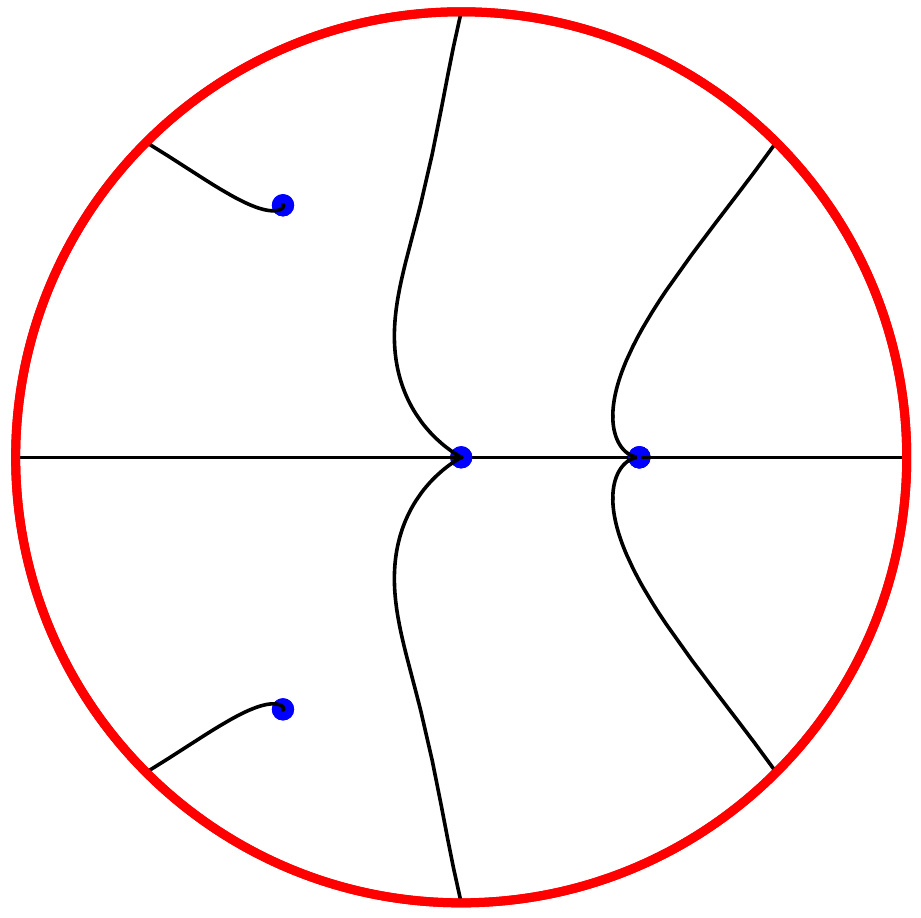}}\qquad\subfigure[$\alpha\in (0,\frac{\pi}2)$]{\includegraphics[width=3.5cm]{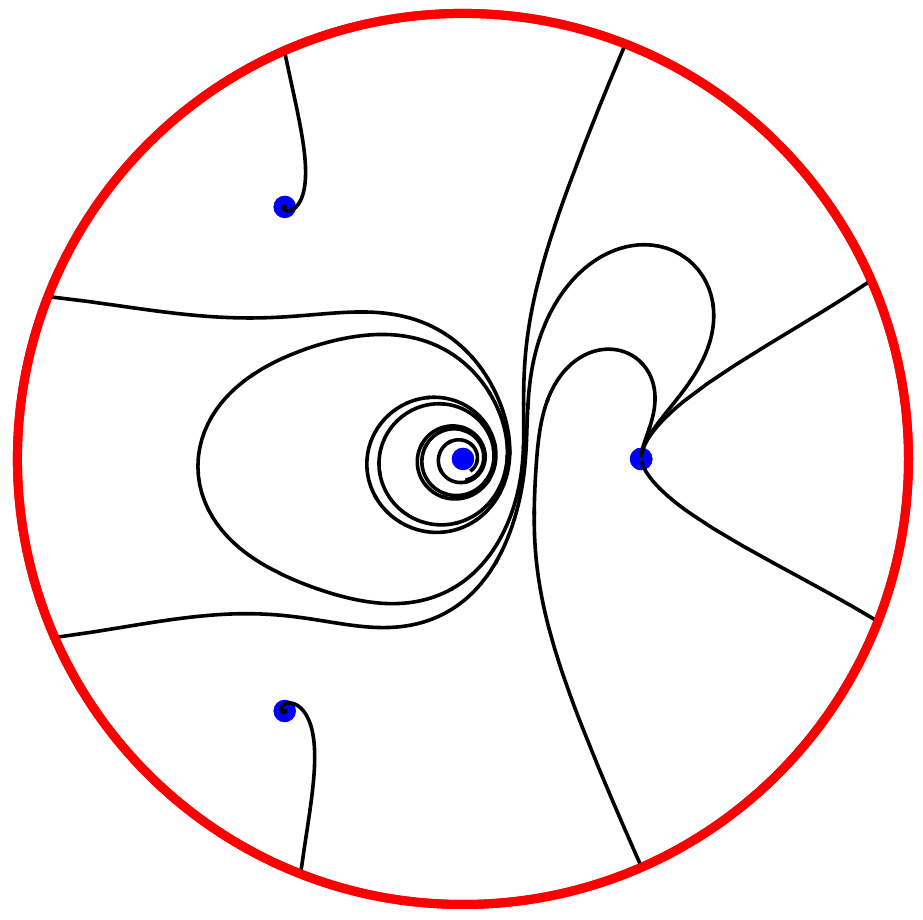}}\qquad\subfigure[$\alpha=\frac{\pi}2$]{\includegraphics[width=3.5cm]{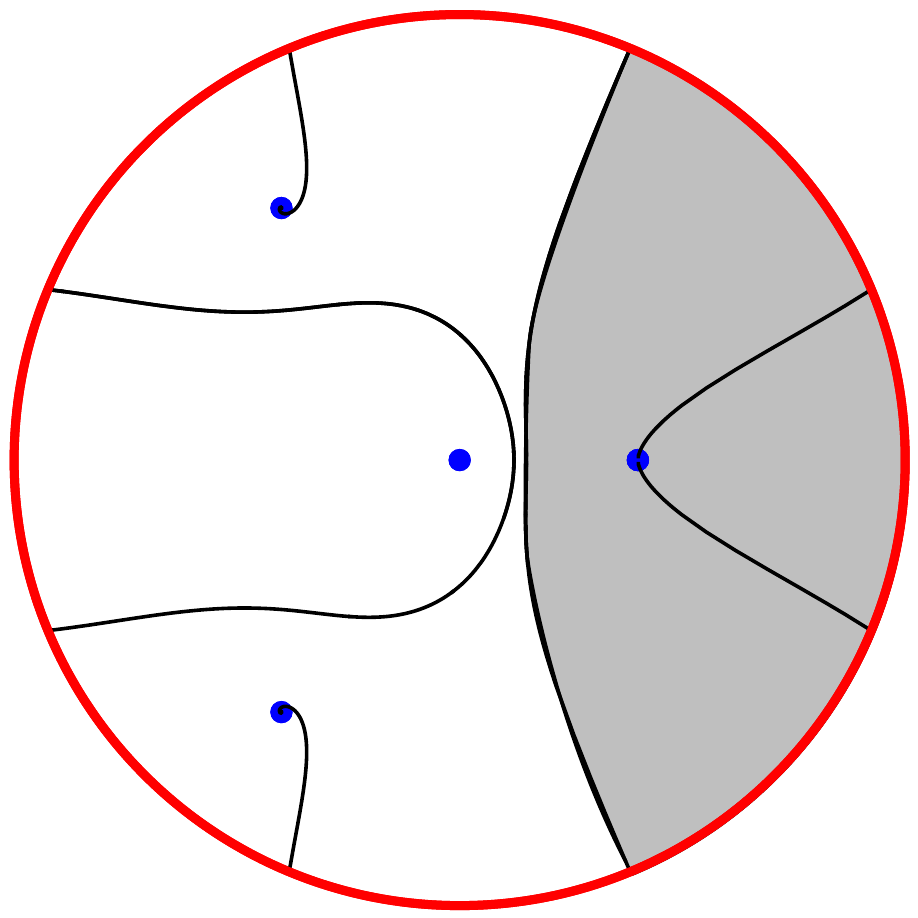}}
\caption{The phase portrait of $\dot z= e^{i\alpha}(z^5-4z^2+3z)$ with a parabolic point at $z=1$ and the parabolic domain in (c). }\label{fig:parab_domain1}\end{center}\end{figure}

\begin{figure}\begin{center}
\subfigure[Parabolic domain of $0$]{\includegraphics[width=4cm]{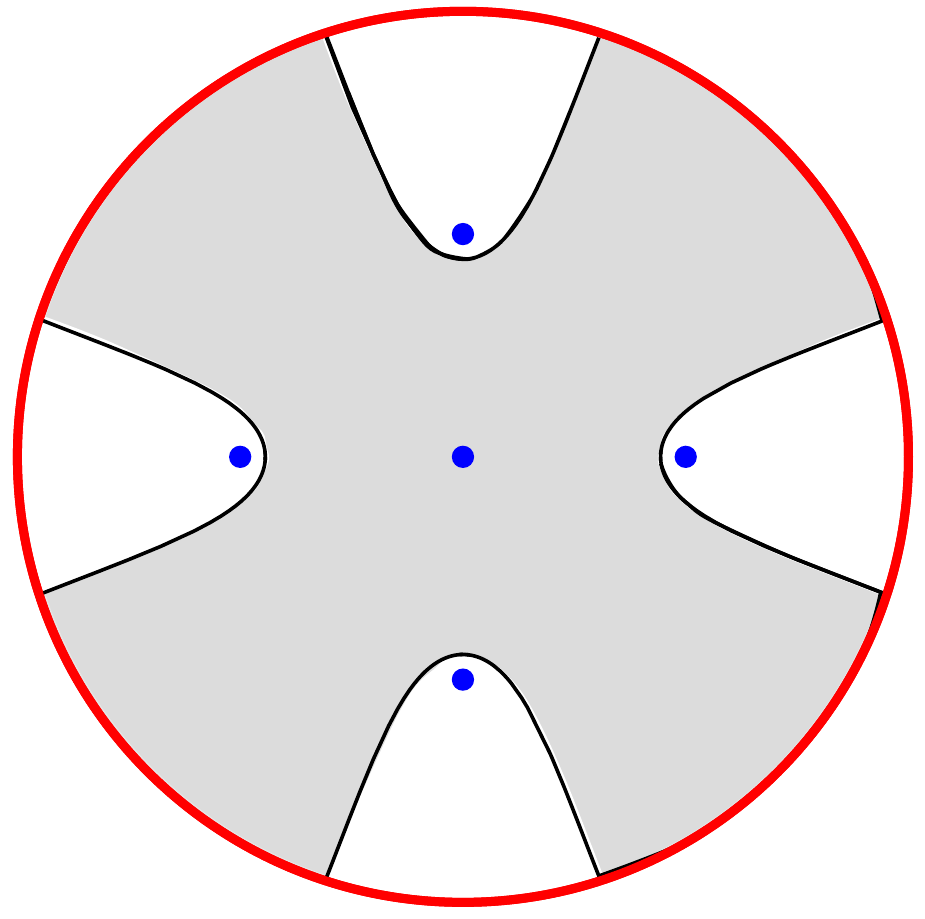}}\qquad\qquad \subfigure[Periodgon of $\dot z= z^6-z^2$]{\includegraphics[width=4cm]{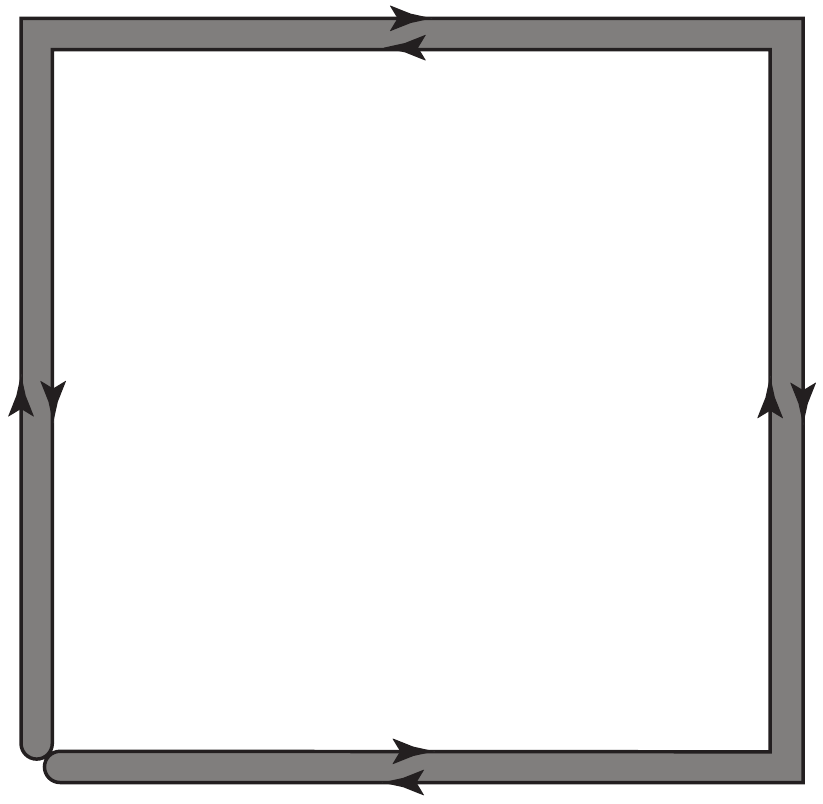}}\caption{The parabolic domain of the origin in  $\dot z= z^6-z^2$. Since the parabolic domain is the complement of the union of the periodic domains of the four singular points, then  the periodgon has empty interior in this case. }\label{fig:parab_domain2}\end{center}\end{figure}

Note that the periodgon has no limit when approaching a parabolic point. Indeed, two sides of the periodgon become infinite. Moreover, their argument makes a nearly full turn (resp. a nearly half-turn full) around the origin when considering  an unfolding of the form  $\dot z = z^2 -\eps z+O(z^3)$ (resp. $\dot z = z^2-\eps +O(z^3)$) (see~Figure~\ref{bif_parabolic_s=0}).

\begin{figure}\begin{center}
\includegraphics[width=6cm]{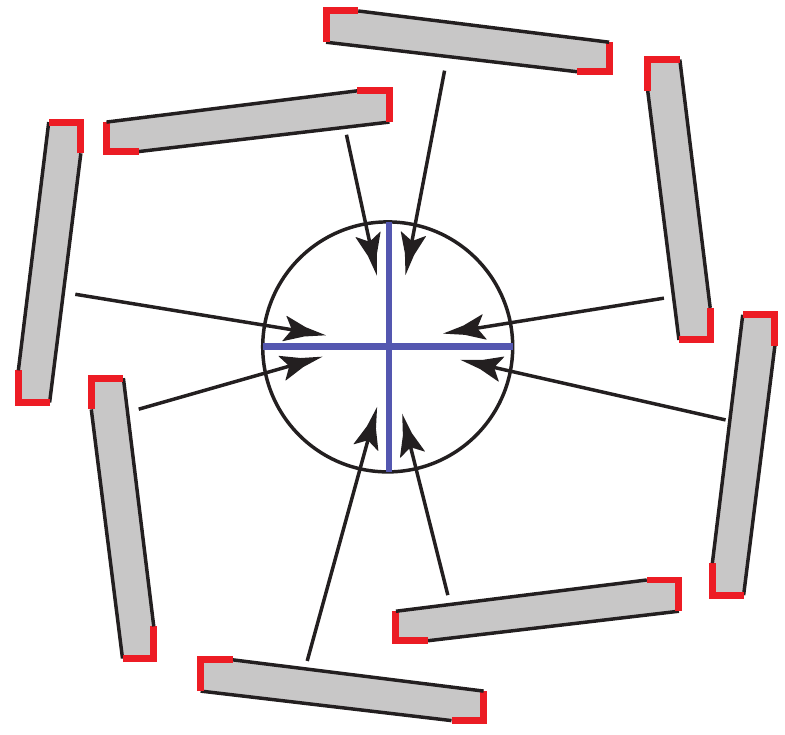}\caption{The bifurcation diagram of the periodgon of $\dot z= z^6-z^2 +se^{i\theta}z$ for $s$ small. The center cycle is the parameter space $se^{i\theta}$ for $s$ small and $\theta\in [0,2\pi]$. The lengths of the black sides corresponding to the periods of $z_0$ and $z_1$ tend to infinity when $s\to0$. The bifurcations occur for $\theta= \frac{m\pi}2$.}\label{bif_parabolic_s=0}\end{center}\end{figure}

\subsection{The rotational property with respect to $\alpha$}\label{sec:rot} The change of coordinate $z\mapsto Z=e^{i\alpha}z$ brings  \eqref{3_par} to the form 
$$\dot Z = e^{-ik\alpha} Z\left(Z^k-k(1-s)^{k-1}Z+(k-1)s^ke^{i\theta}\right).$$ The periodgon of \eqref{3_par} is that of $\dot Z = Z\left(Z^k-k(1-s)^{k-1}Z+(k-1)s^ke^{i\theta}\right)$ rotated by $e^{-ik\alpha}$. Hence, it suffices to study the shape of the periodgon for $\alpha=0$. 

\subsection{The eigenvalues of \eqref{vector_field}}
Because of all the symmetries described in Section~\ref{sec:symmetries} we limit ourselves to $\theta\in\left(0,\frac{\pi}{k-1}\right)$ and, as discussed in Section~\ref{sec:rot}, to $\alpha=0$. We are interested in understanding the shape of the periodgon.  Let us call $z_0=0$, and let $z_1,\dots, z_k$ be the other singular points. They are numbered by increasing argument starting with $z_1$, where $\arg(z_1)\in \left(\theta,\frac{\pi-\theta}{k}\right)$.  Let $\lambda_j$ be the eigenvalue of $z_j$. Then,
\begin{equation}\begin{cases} \lambda_0= (k-1)s^ke^{i\theta},\\
\lambda_j= k(k-1)\left((1-s)^{k-1}z_j-s^ke^{i\theta}\right), & j=1, \dots, k.\end{cases}\label{eigenvalues}\end{equation}

\begin{lemma}\label{lemma_crossing} For $s\neq0,1$, the eigenvalue $\lambda_0$ can only be collinear with one of the $\lambda_j$ if  $\theta= \frac{\pi \ell}{k-1}$ for some integer $\ell$. \end{lemma}
\begin{proof} If $\lambda_j$ is collinear with $\lambda_0$ and $s\neq0,1$, then \eqref{eigenvalues} implies that $\arg z_j= \theta +m_1\pi$ for some integer $m_1$. This in turn implies that $\arg z_j^k= \theta+ m_2\pi$, for some integer $m_2$, hence the result. \end{proof}

\begin{lemma}\label{sign_re_eig} For all $s>0$ and $\theta\in[0,\frac{\pi}{k-1}]$,  then ${\rm Re} (\lambda_0)>0$, and ${\rm Re}(\lambda_j)<0$ if ${\rm Re}(z_j)<0$ and $j>0$. For $s$ close to $1$, then ${\rm Re} (\lambda_j)<0$ for all $j>0$. Moreover, if ${\rm Re}(z_j)>0 $ for all $s$, then ${\rm Re}(\lambda_j)$ changes sign when $s$ decreases from $1$ to $0$. \end{lemma}
\begin{proof} This follows from \eqref{eigenvalues}. \end{proof}

\subsection{Preliminaries on the singular points}

Since for $\alpha=0$ the singular points apart from $z_0=0$ are the same as those of $\dot z = z^k-k(1-s)^{k-1}z +(k-1)s^ke^{i\theta}$ we can apply the results of \cite{KR}.

\begin{proposition}\label{KR} \cite{KR} We consider \eqref{3_par} with  $\theta \in [0,\frac{\pi}{k-1}]$ and $s\in[0,1]$. Let $z_0=0$, and $z_1, \dots, z_k$ be the other singular points. 

\begin{enumerate}
\item The singular points  $z_1(s,\theta,0),\ldots, z_{k}(s,\theta,0)$ have distinct arguments for all $s\in (0,1]$, unless $\theta=0$, in which case the two roots $z_1(s,0,0)$ and $z_k(s,0,0)$ both have zero argument for $s\leq\frac{1}{2}$. Then it makes sense ordering them by increasing value of argument. (When $s=0$, then $z_2, \dots z_k$ have distinct arguments.)
\item If $z_j(s,\theta,0)\notin e^{i\theta}\R$, then the absolute value of $\arg(e^{-i\theta}z_j(s,\theta,0))\in(-\pi,\pi)$ increases monotonically with $s$.
\item This implies that 
$$z_1(0,\theta,0)=0,\quad\text{and}\quad z_j(0,\theta,0)=k^{\frac{1}{k-1}}e^{\frac{2\pi i (j-1)}{k-1}},\ j=2,\ldots,k,$$
and the roots are caught for all $s\in[0,1]$ in the following disjoint sectors:
\begin{align*}
 \arg z_1(s,\theta,0)&\in[\theta,\tfrac{\theta+\pi}{k}],\\
 \arg z_j(s,\theta,0)&\in\left[\tfrac{2\pi(j-1)}{k-1},\tfrac{\theta+(2j-1)\pi}{k},\right],\quad\text{for } 2\leq j\leq\tfrac{k+1}{2},\\
 \arg z_j(s,\theta,0)&\in\left[\tfrac{\theta+(2j-1)\pi}{k},\tfrac{2\pi(j-1)}{k-1}\right],\quad\text{for } \tfrac{k}{2}+1\leq j\leq k.
\end{align*}
\item For $\theta=0$, the two roots $z_1(s,0,0)$ and $z_k(s,0,0)$ are real for $s\in [0,\frac{1}{2}]$ and merge for $s=\frac{1}{2}$. For $s>\frac{1}{2}$, they  split apart in the imaginary direction. 
\end{enumerate} 
\end{proposition} 

\begin{lemma}\label{lemma:roots_opposed} We consider \eqref{3_par} with  $\alpha=0$, $\theta \in [0,\frac{\pi}{k-1}]$ and $s\in[0,1]$. Let $z_0=0$, and $z_1, \dots, z_k$ be the other singular points. 
Then two distinct nonzero roots $z_j$ and $z_\ell$, $j,\ell>0$,  cannot point in opposite directions unless $\theta=0$ or $\theta=\frac{\pi}{k-1}$.\end{lemma}
\begin{proof} Suppose that $z_j= r_je^{i\phi}$ and $z_\ell = -r_\ell e^{i\phi}$. Then 
$\left(r_j^k-(-1)^kr_\ell^k\right)e^{i(k-1)\phi}-k(1-s)^{k-1}(r_j+r_\ell)=0$, from which it follows that $e^{i(k-1)\phi}= \pm1$. Substituting into $z^k-k(1-s)^{k-1}z+(k-1)s^ke^{i\theta}=0$ yields the result.  \end{proof}

\subsection{Shape of the periodgon} 
To understand the shape of the periodgon we need to understand the boundaries of the periodic domains of the singular points. We conjecture that the boundaries of the periodic domains of the $z_j$ for $j>1$ always consist of a single homoclinic loop: this is supported by numerical simulations. It is only the boundaries of the periodic domains of $z_0$ and $z_1$ which undergo several bifurcations when the parameters vary. This is what we study now. We start with the situations $s=0$ and $s=1$, and then vary $s\in (0,1)$. All these are steps to prove the following theorem

\begin{theorem}\label{thm:periodgon}\begin{enumerate} 
\item 
The boundary of the periodic domain of $z_0$ has more than one homoclinic loop for 
$$\begin{cases} 
\theta= \frac{2j\pi}{k-1}, &\quad k \: {\rm odd}, \: {\rm and}\: s\in (0,1),\\
\theta= \frac{(2j-1)\pi}{k-1}, &\quad k \: {\rm even}, \: {\rm and}\: s\in (0,1).\end{cases} $$

\noindent The boundary of the periodic domain of $z_1$ has more than one homoclinic loop for 
$$
\theta=\frac{2j\pi}{k-1}, \quad {\rm all}\:\: k, \: {\rm and}\: s\in\left(0,\frac12\right).$$
\item There are no other bifurcations of the periodgon
\begin{itemize}
\item In the neighborhood of $s=0$;
\item In the neighborhood of $s=1$;
\item In the neighborhood of $\theta=\frac{2j\pi}{k-1}$, $j=0, \dots, k-2$;
\item For even $k$, in the neighborhood of $\theta= \frac{(2j-1)\pi}{k-1}$, $j=0, \dots, k-2$. 
\end{itemize}
\end{enumerate}
\end{theorem}  

This leads to the conjecture.

\begin{conjecture}\label{conj:periodgon}
 The only bifurcations of the periodgon occur
\begin{itemize} 
\item along the rays $\theta= \frac{2j\pi}{k-1}$ for $k$ odd;
\item along the rays $\theta= \frac{(2j-1)\pi}{k-1}$ and the half-rays $\theta= \frac{2j\pi}{k-1}$, $s\in\left(0,\frac12\right)$ for $k$ even.
\end{itemize} 
Moreover, only the singular points $z_0$ and $z_1$ can have more than one homoclinic loop. \end{conjecture}

\subsubsection{Proof of Theorem~\ref{thm:periodgon}}

We discuss here (1) and (2) simultaneously. Indeed, (2) follows from  transversality properties for the situations studied in (1). 

\medskip\noindent{\bf The case $s=1$.} Using the change $z\mapsto Z= e^{-i\frac{\theta}k}z$ it suffices to study the boundary of the periodic domain of $z_0$ from the system $\dot Z = iZ(Z^k+k)$, for which all points are simultaneously centers. Moreover, the system is symmetric under $z\mapsto e^{i\frac{\pi}k} z$, and hence so is the boundary of the periodic domain of $z_0$. Then it consists of $k$ homoclinic loops as in Figure~\ref{boundary_0} (a) and (d). The periodgon in that case is degenerate to a long line segment corresponding to the vector $\nu_0$ and $k$ small segments corresponding to vectors $\nu_j=-\frac{\nu_0}{k}$, $j>0$.
\begin{figure}\begin{center}
\subfigure[$s=1$, $\theta=0$]{\includegraphics[width=3cm]{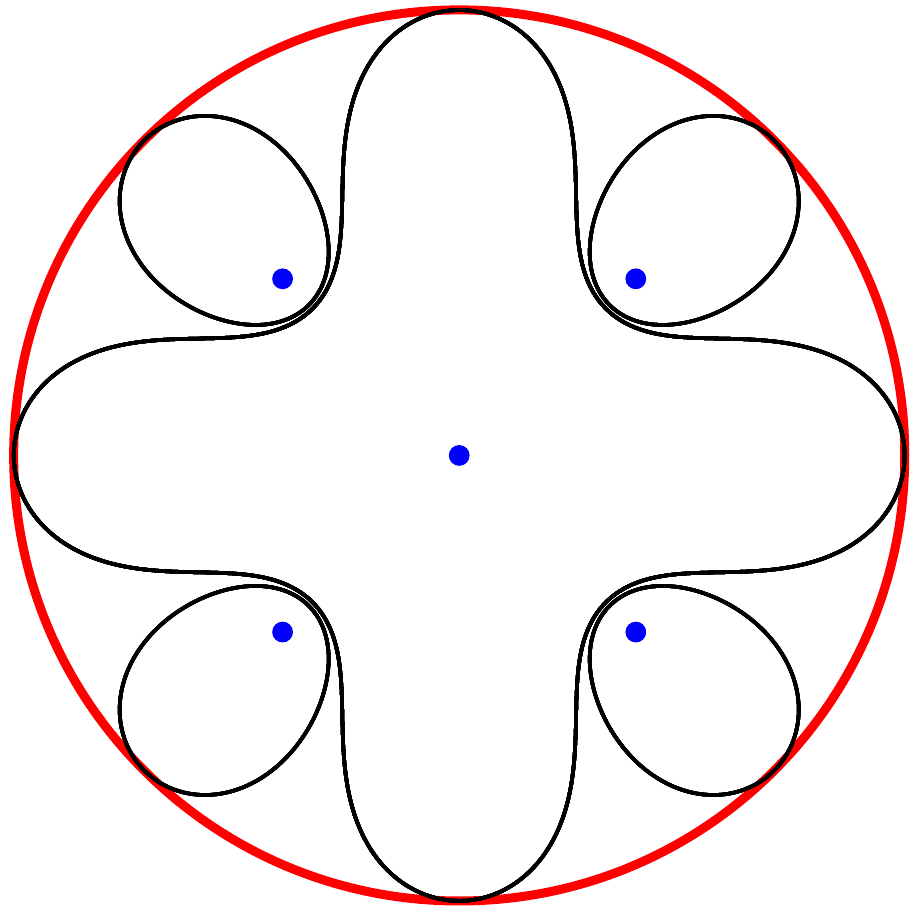}}\qquad\subfigure[$s\approx1$,  $\theta\in(0,\frac{\pi}{k-1})$]{\includegraphics[width=3cm]{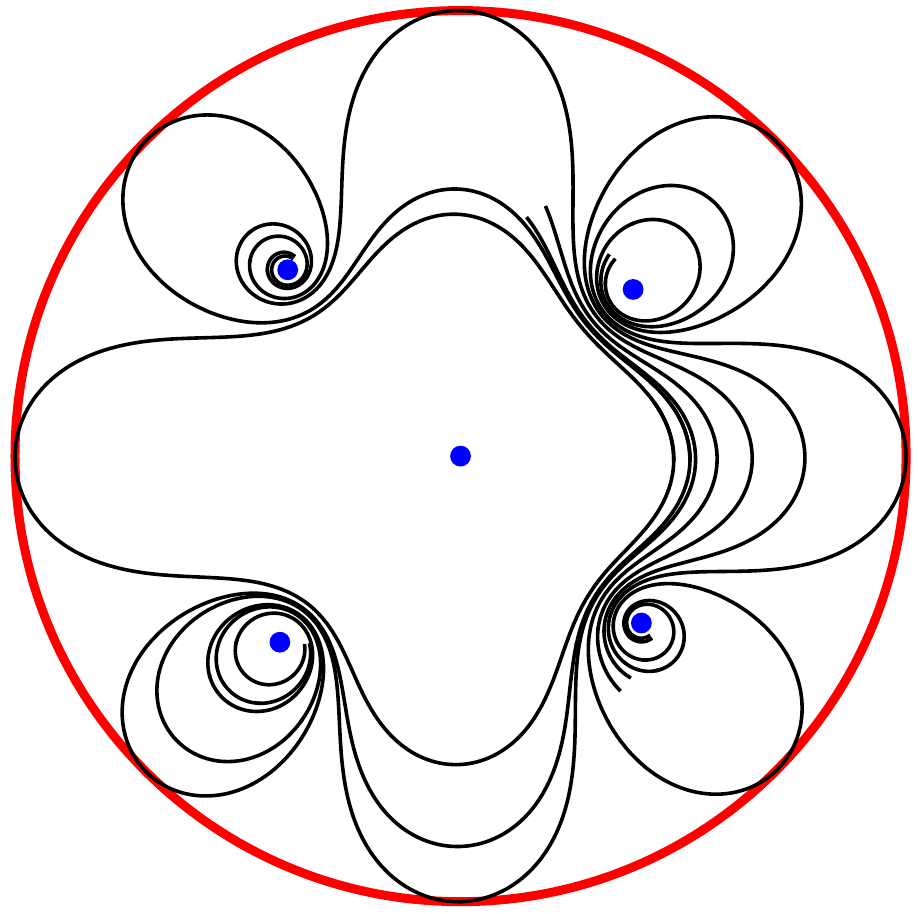}}\qquad\subfigure[$s\approx1$,  $\theta=\frac{\pi}{k-1}$]{\includegraphics[width=3cm]{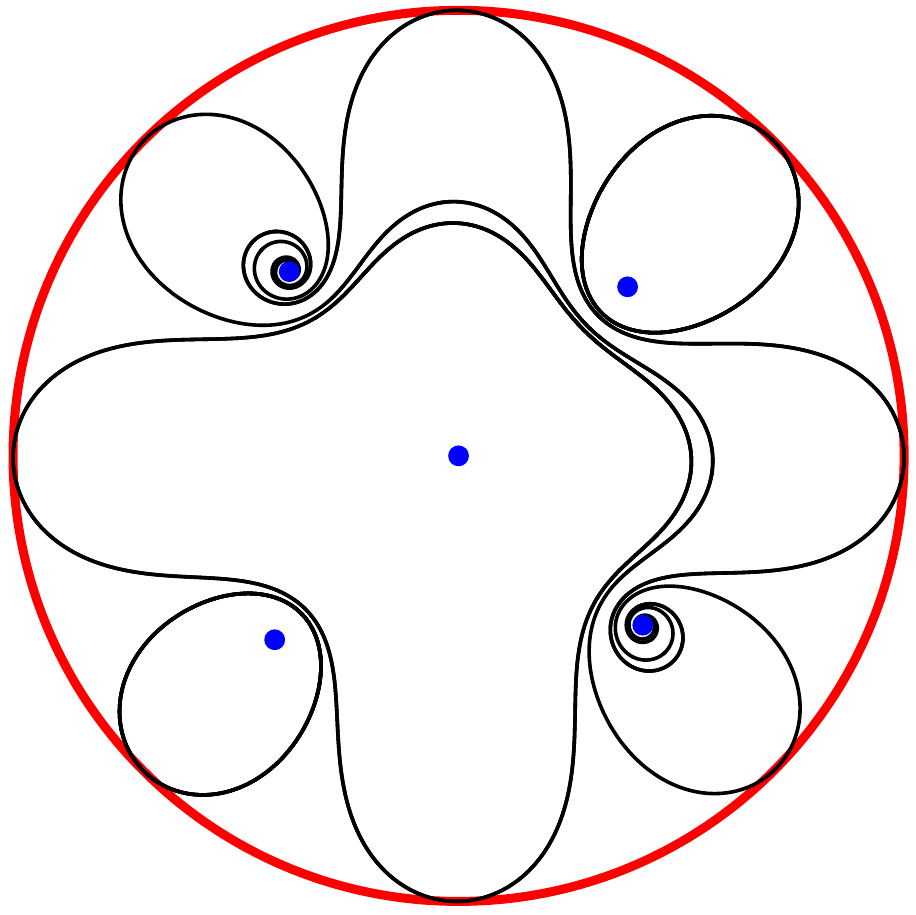}}\\
\subfigure[$s=1$, $\theta=0$]{\includegraphics[width=3cm]{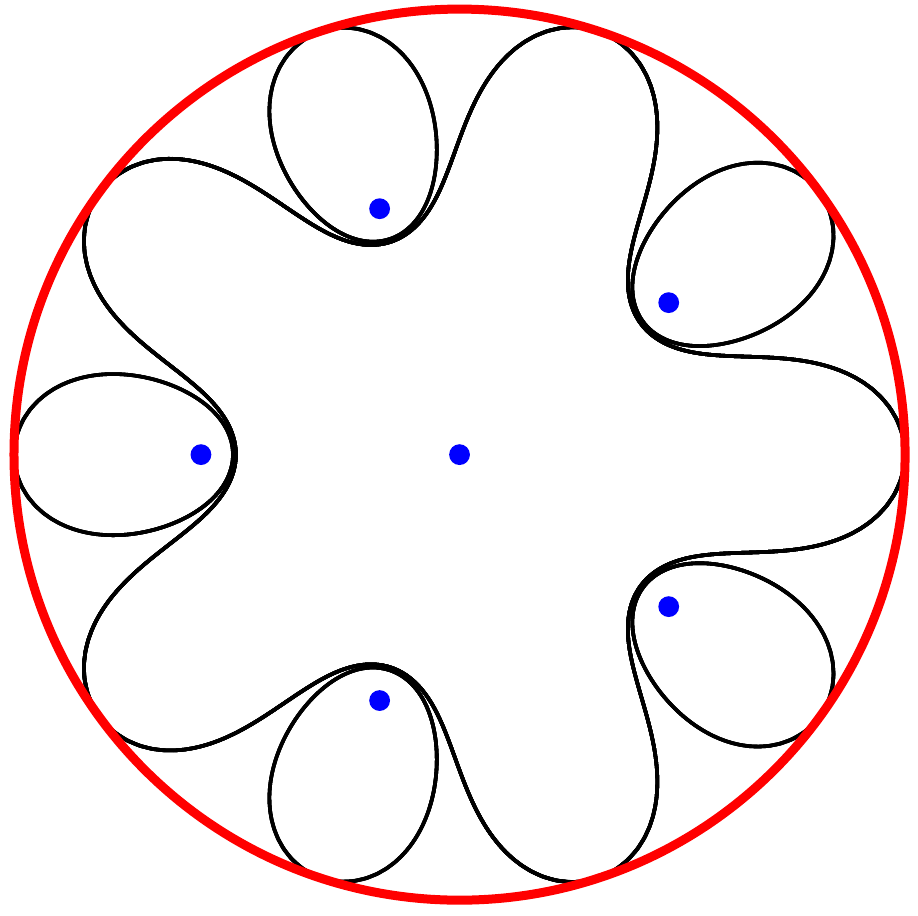}}\qquad\subfigure[$s\approx1$, $\theta=0$]{\includegraphics[width=3cm]{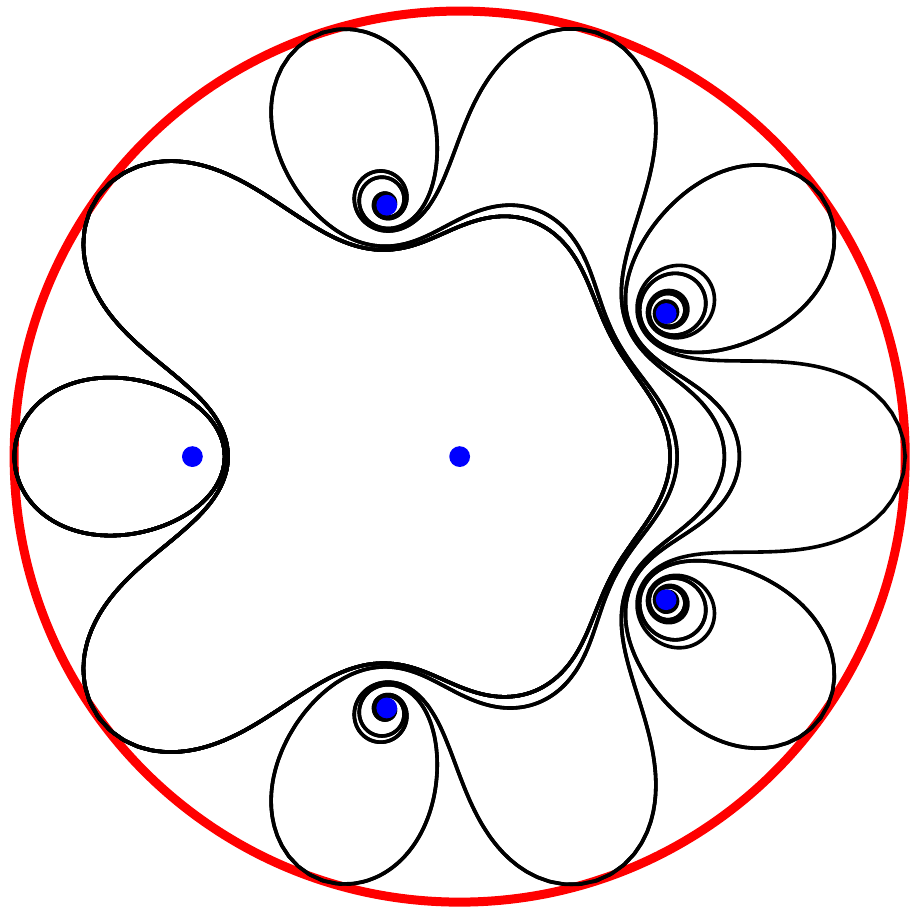}}\qquad\subfigure[$s\approx1$, $\theta=\frac{\pi}{k-1}$]{\includegraphics[width=3cm]{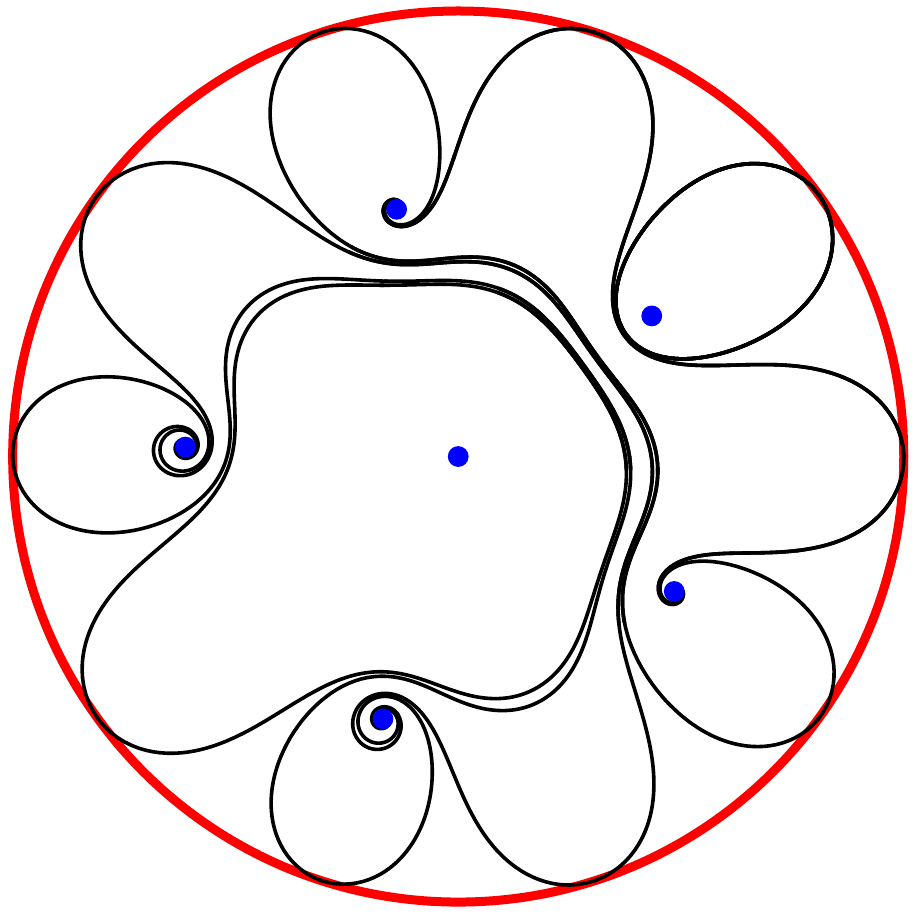}}\caption{The boundary of the periodic domain (homoclinic loop(s)) of $z_0$ for $s\approx 1$, $\alpha= \frac{\theta}{k}$ and $k=4$ (resp. $k=5$) on the upper (resp. lower) row. The figures are obtained by integrating on a disk the vector field $\dot z =e^{i\delta}P_\eps(z)$ so that $e^{i\delta} P_\eps'(z_0)\in i\R$. We see multiple homoclinic loops around  $z_0$ for $\theta=0$, $k$ odd and $\theta=\frac{\pi}{k-1}$, $k$ even.}\label{boundary_0}\end{center}\end{figure}

\medskip\noindent{\bf The case $s=0$}. In that case $z_0=z_1$ is a parabolic point. In the case of one parabolic point the periodgon has been defined in Section~\ref{sec:degenerate}, but it is not the limit when $s\to 0$ of the periodgon for $s\neq0$. However, to understand the periodgon for $s$ small we need to understand the periodic domains at $s=0$ of the nonzero singular points, and which separatrices of $\infty$ land at $z_0$. 
Looking at the system $\dot z = z^2(z^{k-1}-k)$, then $\lambda_j=k(k-1)z_j$ for $j\geq2$, from which it follows that each ${\rm Re}(\lambda_j)$ has the sign of ${\rm Re} (z_j)$. This allows drawing the phase portrait (see Figure~\ref{parabolic_cod1_origin}). Indeed, the sepal zones of the parabolic point separate the singular points into two groups: the attactive ones on one side, and the repelling ones on the other side. There are also two centers when $k\equiv 1\:(\text{mod}\:4)$: the basin of each center is surrounded by a sepal zone. 

The nonzero singular points are of the form $z_j=\exp\left(\frac{2\pi (j-1)}{k-1}i\right)$, $j=2, \dots, k$, with eigenvalues given in \eqref{eigenvalues}. Now, consider the system $\dot z = iP_\eps(z)$, which is reversible with respect to the real axis. There are one (resp. two) additional singular points on the real axis for $k$ even (resp. $k$ odd), which are centers. The singular points in the upper (resp. lower) half-plane are attracting (resp. repelling) for $s=0$, and that will remain the case for $s$ small. This yields the phase portraits in Figure~\ref{parabolic_rotated} for the vector field $\dot z =iP_\eps(z)$, since all attracting (resp. repelling) singular points are necessarily linked to the repelling (resp. attracting) sector of the parabolic point at $s=0$. 

\begin{figure}\begin{center}
\subfigure[$k=4$]{\includegraphics[width=3.5cm]{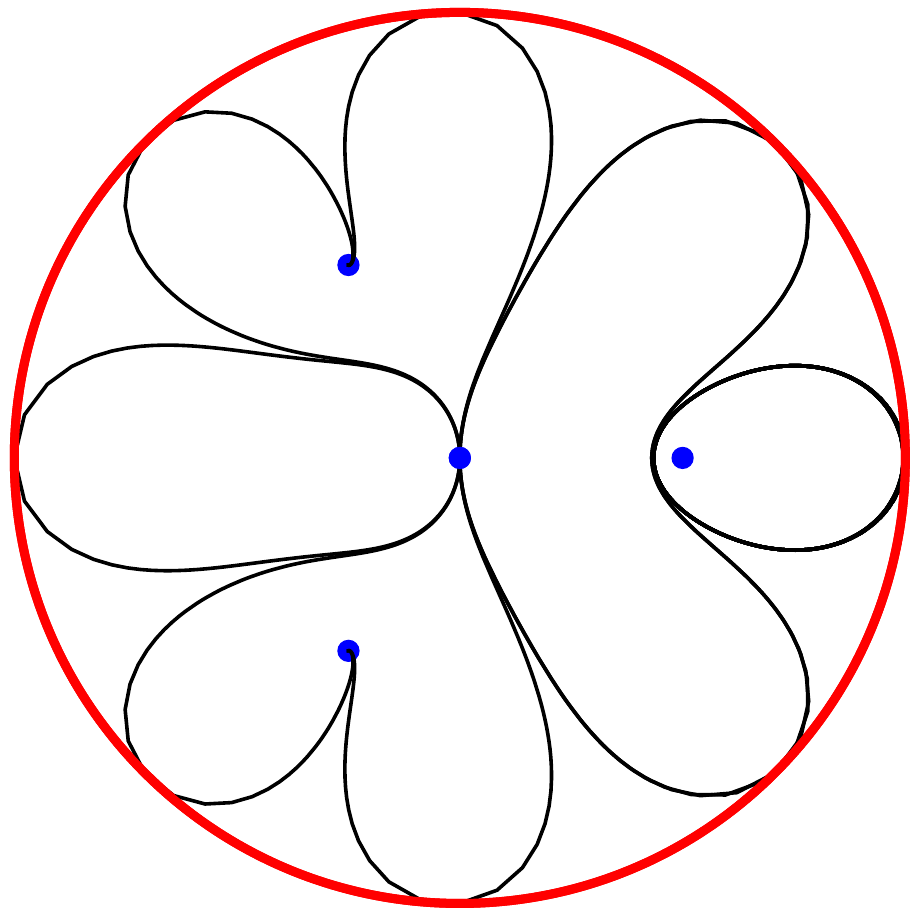}}\qquad\qquad\subfigure[$k=5$]{\includegraphics[width=3.5cm]{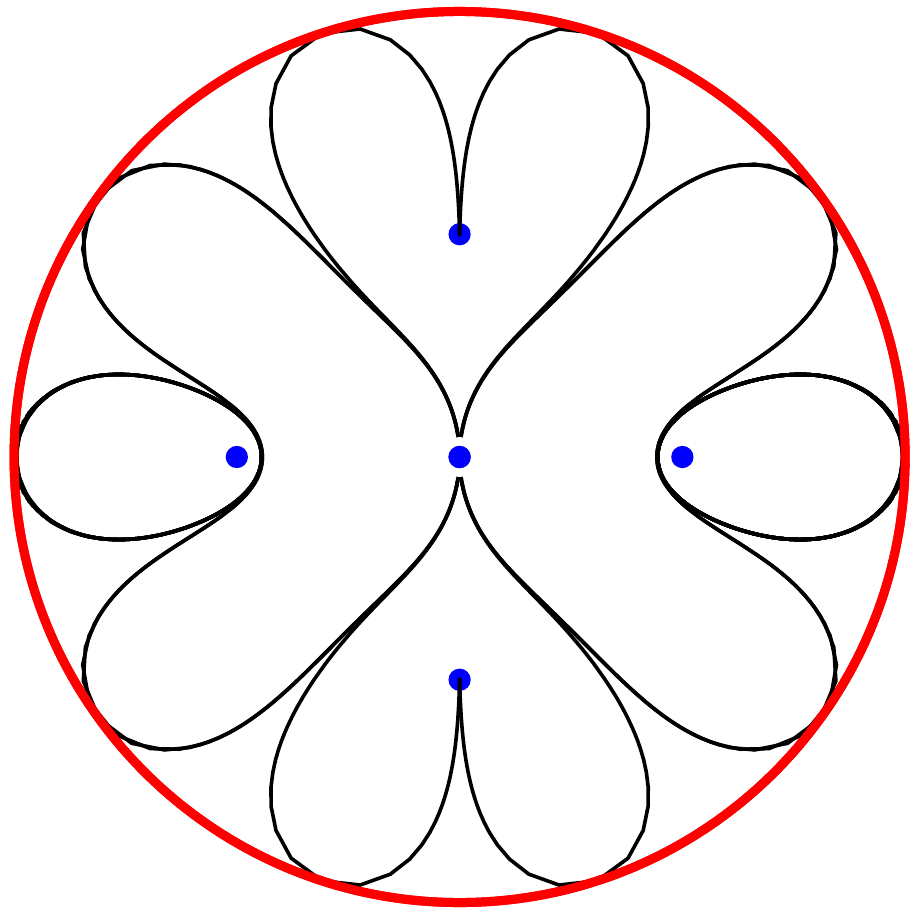}}\caption{The phase portrait of $\dot z = i(z^{k+1}- kz^2)$.}\label{parabolic_rotated} \end{center}\end{figure}

\medskip\noindent{\bf The case $\theta=\frac{2j\pi}{k-1}$.} (See Figures~\ref{boundary_0}(e) and \ref{boundary_01}(a) and (d).) It suffices to consider the case $\theta=0$, where the system is symmetric with respect to the real axis. For $s\in (0,\frac12)$ and $\theta=0$, the system has $3$ (resp. $4$) singular points on the real axis for $k$ even (resp. odd), and they are simultaneously
 centers of the system $\dot z = iP_\eps(z)$. Then $0=z_0<z_1<z_k$,  from which it follows that the periodic domain of $z_1$ is bounded by two homoclinic loops. When $s=\frac12$, the points $z_1$ and $z_k$ merge in a parabolic point. Then $z_1$ (resp. $z_k$) moves in the upper (resp. lower) half-plane for $s>\frac12$. Hence, the periodic domain of $z_1$ is between those of $z_2$ and $z_k$ and this is also valid for small $\theta$. Similarly for odd $k$ we have $z_{\frac{k+1}2}<z_0=0$, yielding that $z_0$ has two homoclinic loops. For small positive $\theta$, then the periodic domain of $z_0$ is between that of $z_{\lfloor\frac{k+1}2\rfloor}$ and that of $z_{\lceil\frac{k}2\rceil+1}$.  From \eqref{eigenvalues}, the singular point $z_j$, $j\geq2$, of  system $\dot z = iP_\eps(z)$ is attracting (resp. repelling) if ${\rm Im}(z_j)>0$ (resp. ${\rm Im}(z_j)<0$). Hence, from the symmetry, the singular points not on the real axis are linked two by two and the passage for these links is to the right of $z_0$.

 \medskip\noindent{\bf The case $\theta=\frac{(2j-1)\pi}{k-1}$.} (See Figures~\ref{boundary_0}(c) and (f) and \ref{boundary_01}(c) and (f).) It suffices to consider the case $j=1$. The change $Z=e^{-i\frac{\pi}{k-1}} z$ transforms the system into $\dot Z= e^{-i\frac{\pi}{k-1}} Z(-Z^k -k(1-s)^{k-1}Z+(k-1)s^k)$, which means that the periodic domains of the singular points along the axis $Z=0$ can be seen from the system $$\dot Z= \omega(Z) =i Z(-Z^k -k(1-s)^{k-1}Z+(k-1)s^k),$$ again a reversible system with respect to the real axis. For  $k$ even there is always a singular point $Z_{\frac{k}2+1}$ on $\R^-$, whose homoclinic loop is part of the boundary of the periodic domain of $z_0=Z_0=0$. Hence, there is always a bifurcation of the periodgon along the line $\theta=\frac{(2j-1)\pi}{k-1}$ for $k$ even. The eigenvalue at the singular points $Z_j$ of $\dot Z=\omega(Z)$ is again of the form $i\lambda_j$ for $\lambda_j$ defined in \eqref{eigenvalues} (with $z_j$ replaced by $Z_j$). Hence $Z_j$ is attracting (resp. repelling) if ${\rm Im}(Z_j)>0$ (resp. ${\rm Im}(Z_j)<0$). As in the case for $\theta=0$, each singular point of the upper half-plane is linked to one in the lower half-plane  and the passage is between $Z_0$ and $Z_1$. Hence (modulo the conjecture that the $Z_j$, $j\geq2$ always have a unique homoclinic loop,) there is no bifurcation of the periodgon for increasing $s\in (0,1)$ and the order of the sides is the same as the one for $s$ small, which will be discussed below.

\medskip\noindent{\bf The case $s$ close to $1$.} We start by computing the eigenvalues for $\alpha=\frac{\theta}{k}$.
Let $\lambda_j$ be the eigenvalue of $z_j$. Then,
\begin{equation}\begin{cases} \lambda_0= (k-1)s^k,\\
\lambda_j= k(k-1)\left((1-s)^{k-1}e^{-i\frac{(k-1)\theta}k}z_j-s^k\right), & j=1, \dots, k.\end{cases}\label{eigenvalues2}\end{equation} To compute the boundary of the periodic domain of $z_0$ we multiply the system and its eigenvalues by $i$. 
Then ${\rm Re}(i\lambda_j)>0$ if and only if ${\rm Im}(\lambda_j)<0$, and ${\rm Im}(\lambda_j)$ has the sign of ${\rm Im}(e^{-i\frac{(k-1)\theta}k}z_j)$. For $s=1$, then $z_j= (k-1)^{\frac1k}e^{i\frac{(2j-1)\pi}k}$. Hence 
$${\rm Im}(\lambda_j) \begin{cases} >0, & 1\leq j\leq \frac{k}2,\\
<0, &\frac{k+1}2<j \leq k,\\
<0 & j=\frac{k+1}2, \theta>0, \: $k$ \:\text{odd},\\
=0, &j=\frac{k+1}2, \theta=0, \:$k$ \:\text{odd}.\end{cases}$$
This gives the direction in which the homoclinic loops that surround all singular points $z_j$, $j>0$, are broken in $\dot z= iP_\eps(z)$ for $s$ close to $1$. Indeed for $s=1$ all singular points are centers. When $s\neq1$, the point $z_j$ becomes an  attracting (resp. repelling) focus and ${\rm Im}(\lambda_j) >0$ (resp. negative), and stays a center when ${\rm Im}(\lambda_j) =0$. Since the vector field can have no limit cycle, this gives the direction in which the homoclinic loop is broken. Of course, since everything is continuous, it suffices to see how it is broken for $\theta=0$, which has been studied above.
The case $\alpha=0$ comes by applying the change of variable $z\mapsto Z=e^{i\alpha}z$ as described in Section~\ref{sec:rot}. The factor $e^{-ik\alpha}$ has no influence on the shape of the period domains. Hence the periodic domains for $\alpha=0$ are obtained by rotating those for $\alpha=\frac{\theta}{k}$ of an angle $\alpha=\frac{\theta}{k}\in\left[0,\frac{\pi}{k(k-1)}\right]$.

\medskip\noindent{\bf The case $s$ close to $0$.} 
Here we consider the case $\alpha=0$.

For $s=0$, let us first compute the periodic domain of $z_j= k^{\frac1{k-1}} \exp\left(\frac{2\pi ji}{k-1}\right)$, $j\geq2$. It is the domain of the center at $z_j$ for $\dot z = i\exp\left(-\frac{2\pi ji}{k-1}\right)(z^{k+1}-z^2)$.  Letting $Z=  \exp\left(-\frac{2\pi ji}{k-1}\right)z$ changes the system to $\dot Z = iZ^2(Z^{k-1}-1)$, whose phase portrait appears in Figure~\ref{parabolic_rotated}. Because of the symmetry and the fact that $Z_j\in \R^+$, this shows that its periodic domain is bounded by a unique homoclinic loop. 
\begin{figure}
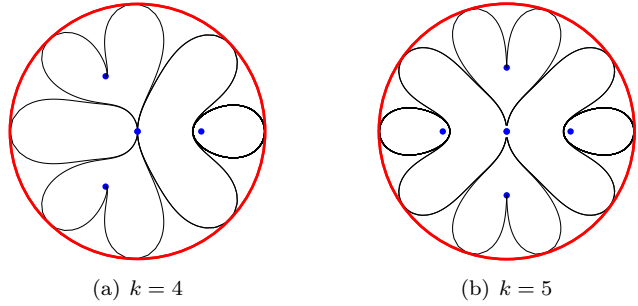
\begin{center}
\subfigure[$k=4$]{\includegraphics[width=3.5cm]{parabolic_rotated_5}}\qquad\qquad\subfigure[$k=5$]{\includegraphics[width=3.5cm]{parabolic_rotated_6}}\caption{The phase portrait of $\dot z = i(z^{k+1}- kz^2)$.}\label{parabolic_rotated} \end{center}\end{figure}

For $s\approx 0$, the boundary of the periodic domain (homoclinic loop(s)) of $z_0$ for different values of $\theta\in [0,\frac{\pi}{k-1}]$ and that of $z_1$ for $\theta\in\{0,\frac{\pi}{k-1}\}$ appear in Figure~\ref{boundary_01}. This comes from the knowledge at $s=0$, at $\theta=0$  and $\theta= \frac{\pi}{k-1}$, as studied above.

 \begin{figure}\begin{center}
\subfigure[$\theta=0$, $k=4$]{\includegraphics[width=3.2cm]{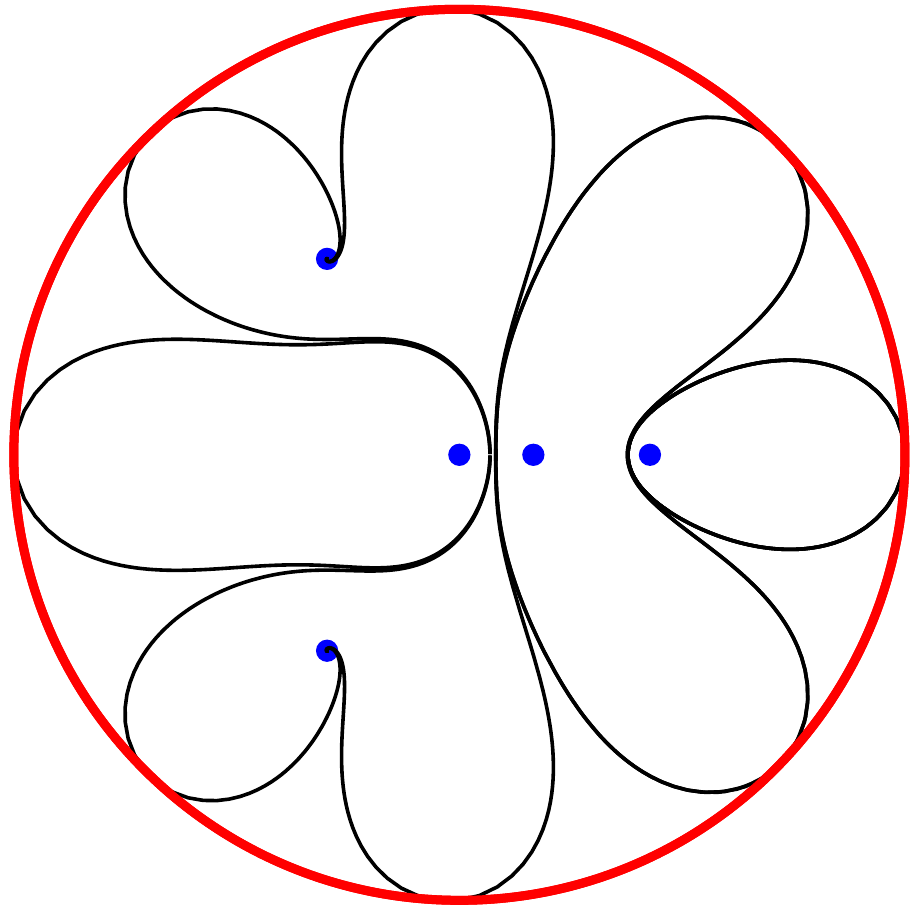}}
\qquad\subfigure[$\theta\in(0,\frac{\pi}{k-1})$, $k=4$]{\includegraphics[width=3.2cm]{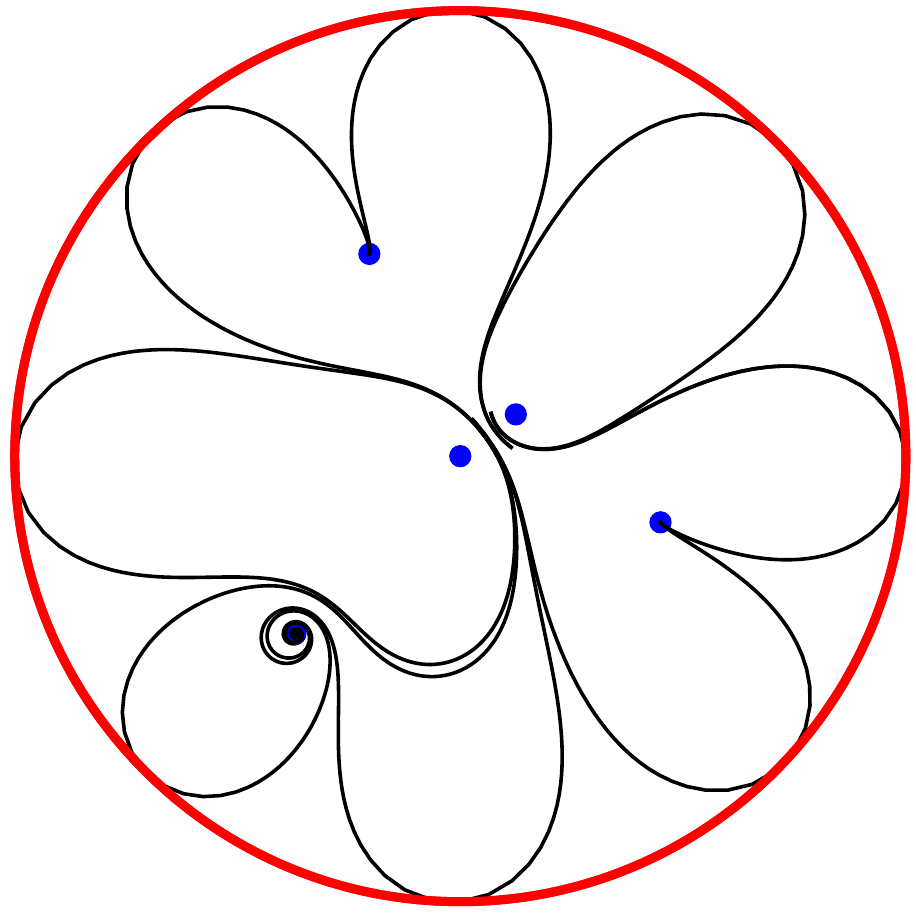}}\qquad
\subfigure[$\theta=\frac{\pi}{k-1}$, $k=4$][$\theta=\frac{\pi}{k-1}$, $k=4$]{\includegraphics[width=3.2cm]{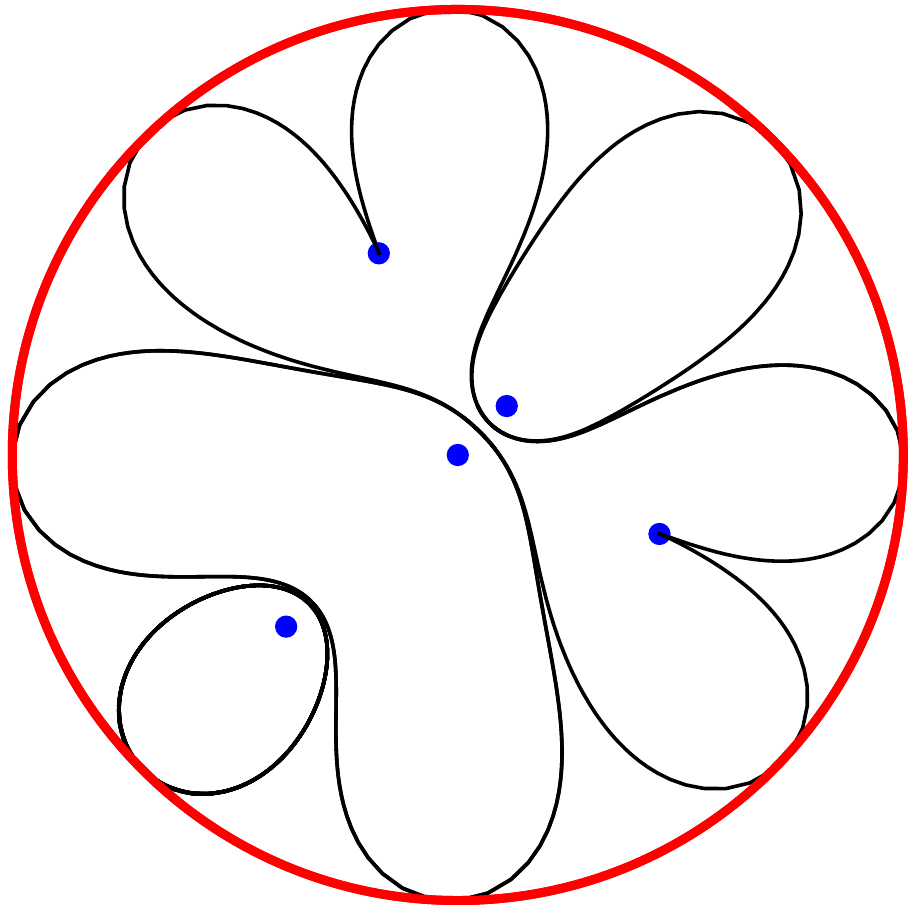}}\\
\subfigure[$\theta=0$, $k=5$]{\includegraphics[width=3.2cm]{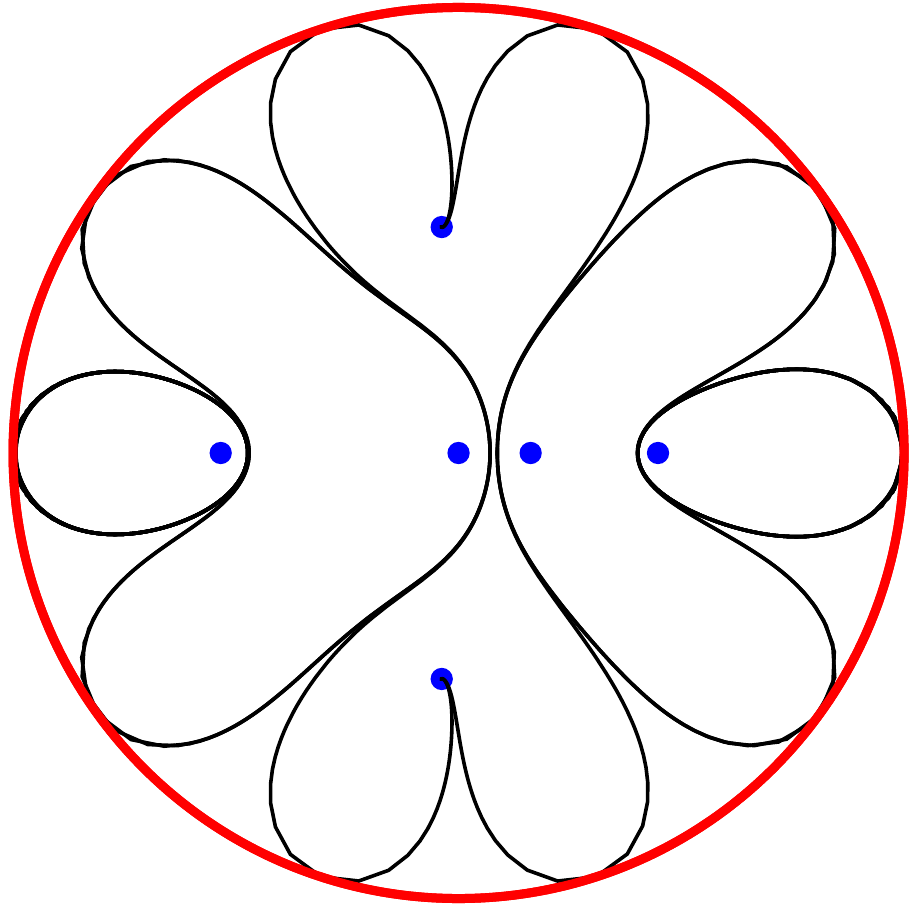}}\qquad\subfigure[$\theta\in(0,\frac{\pi}{k-1})$, $k=5$]{\includegraphics[width=3.2cm]{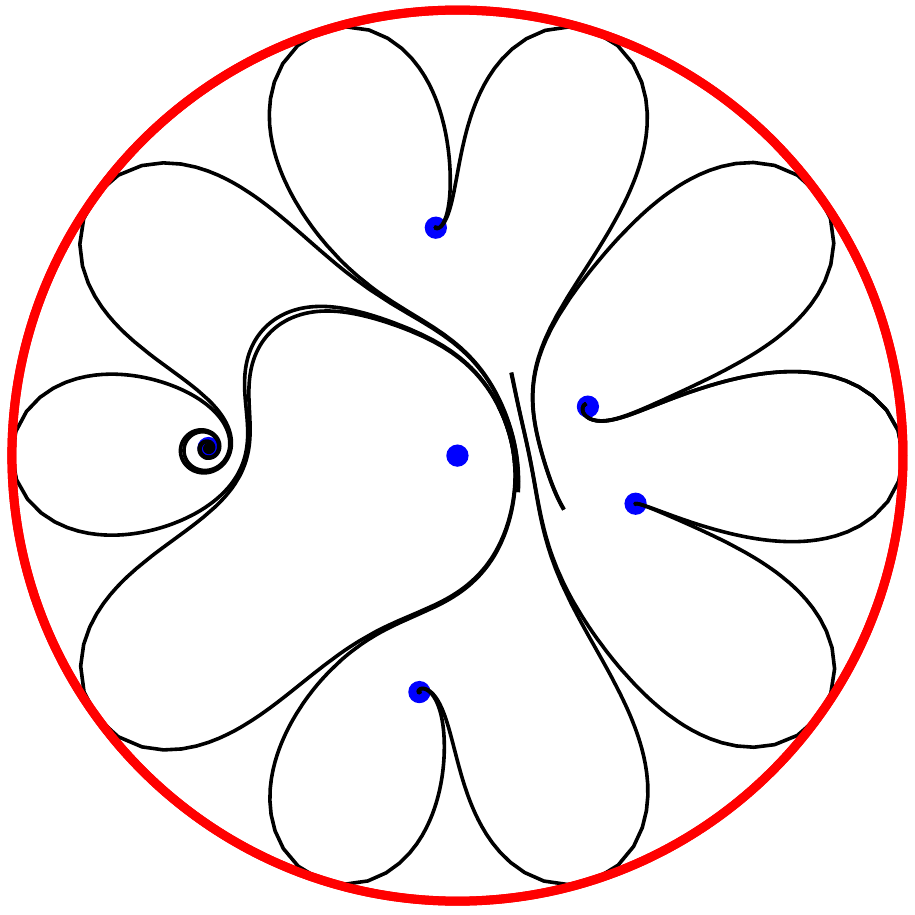}}\qquad\subfigure[$\theta=\frac{\pi}{k-1}$, $k=5$][$\theta=\frac{\pi}{k-1}$, $k=5$]{\includegraphics[width=3.2cm]{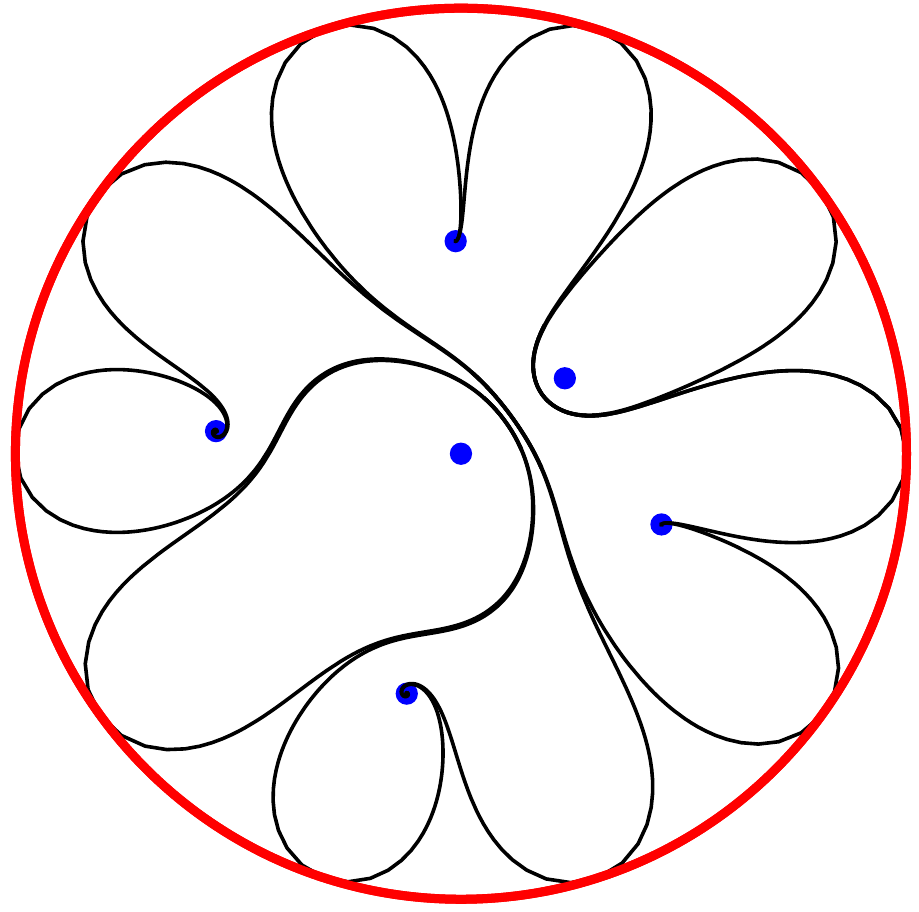}} \caption{For $s\approx 0$, the boundary of the periodic domain (homoclinic loop(s)) of $z_0$ for different values of $\theta\in [0,\frac{\pi}{k-1}]$ and that of $z_1$ for $\theta\in\{0,\frac{\pi}{k-1}\}$. The figures are obtained by integrating on a disk $\dot z =e^{i\delta}P_\eps(z)$ so that $e^{i\delta} P_\eps'(z_0)\in i\R$. We see multiple homoclinic loops around  $z_1$ for $\theta=0$, and around $z_0$ for $\theta=0$, $k$ odd and $\theta=\frac{\pi}{k-1}$, $k$ even.} \label{boundary_01}\end{center}\end{figure}

For  $\theta\in \left(0,\frac{\pi}{k-1}\right)$, the symmetry existing for $\theta=0$ is broken, and all periodic domains have one homoclinic loop. To find the periodic domain of $z_0$ we must now consider the vector field $\dot z = e^{i\left(\frac{\pi}2-\theta\right)}P_\eps(z)$. The points that were previously on the real axis are $z_k= k^{\frac1{k-1}}(1-s) -k^{-1}s^ke^{i\theta}+o(s^k)$ and, for $k$ odd, $z_{\frac{k+1}2}= -k^{\frac1{k-1}}(1-s) -k^{-1}s^ke^{i\theta}+o(s^k)$, yielding that ${\rm Re}(e^{i\left(\frac{\pi}2-\theta\right)}\lambda_k)>0$ (resp. ${\rm Re}(e^{i\left(\frac{\pi}2-\theta\right)}\lambda_{\frac{k+1}2})<0$ for $k$ odd). Then  $z_k$ (resp. $z_{\frac{k+1}2}$ for $k$ odd) is a repelling (resp. attracting) point of $\dot z = e^{i\left(\frac{\pi}2-\theta\right)}P_\eps(z)$. 

 Hence, the periodic domain of $z_0$ is between that of $z_{\lfloor\frac{k+1}2\rfloor}$ and that of $z_{\lceil\frac{k}2\rceil+1}$, while that $z_1$ is between $z_k$ and $z_2$.

\medskip\noindent{\bf The case $s\approx\frac12$ and $\theta\approx 0$.} There is a parabolic point with $z_1=z_k$ for $(s,\theta)=(\frac12, 0)$.  What happens here is very close to the situation studied in \cite{KR}, where the periodic domain of $z_1$ is bounded by two homoclinic loops for real $s<\frac12$ and by one homoclinic loop elsewhere. 

\medskip\noindent This ends the proof of Theorem~\ref{thm:periodgon}.\hfill $\Box$

\subsubsection{Discussion of Conjecture~\ref{conj:periodgon}}

The conjecture is motivated by the fact that the proposed bifurcation diagram for the periodgon is the simplest that connects the known bifurcations. Moreover, the numerical simulations show no bifurcations of the periodgon outside these bifurcation loci. See for instance Figures~\ref{fig:periodgon_K5} and \ref{fig:periodgon_K10}.

\begin{figure}\begin{center} 
\subfigure[]{\includegraphics[height=4cm]{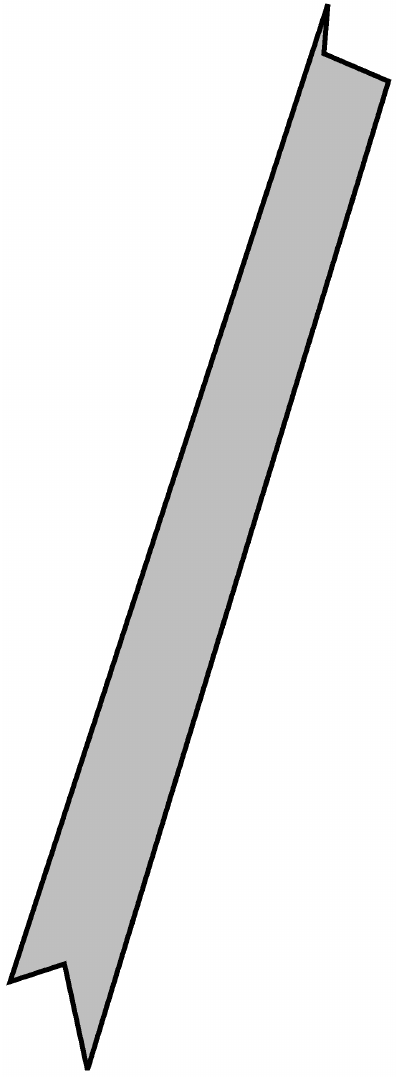}}\qquad\subfigure[]{\includegraphics[height=4cm]{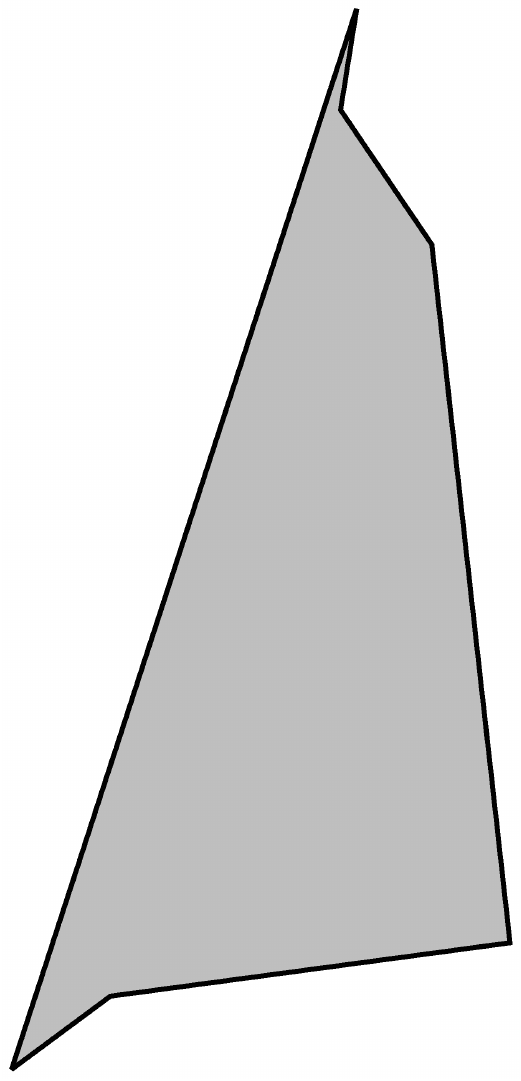}}\qquad\subfigure[]{\includegraphics[height=4cm]{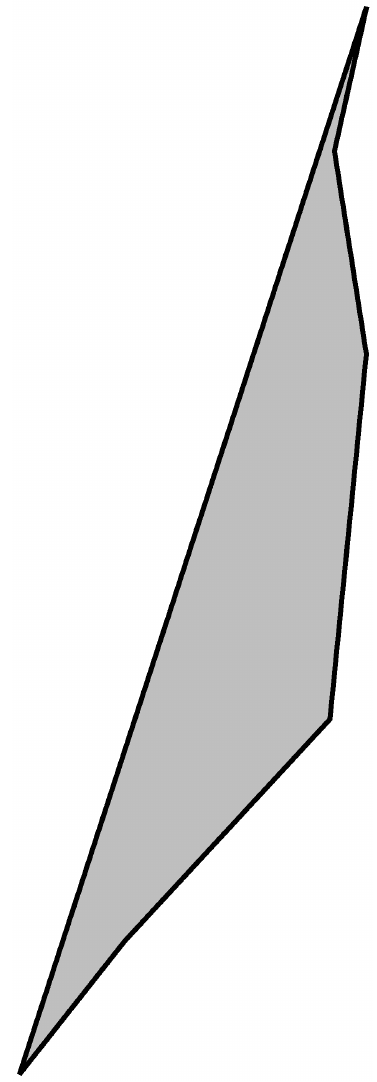}}\qquad\subfigure[]{\includegraphics[height=4cm]{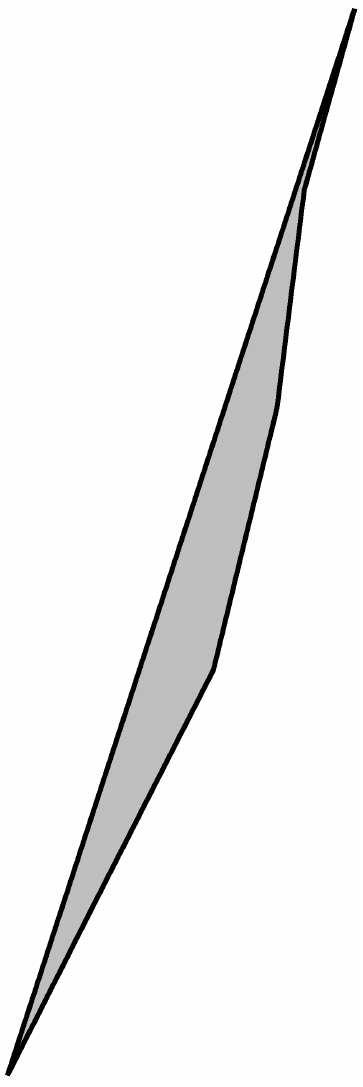}}\qquad\subfigure[$s=1$]{\includegraphics[height=4cm]{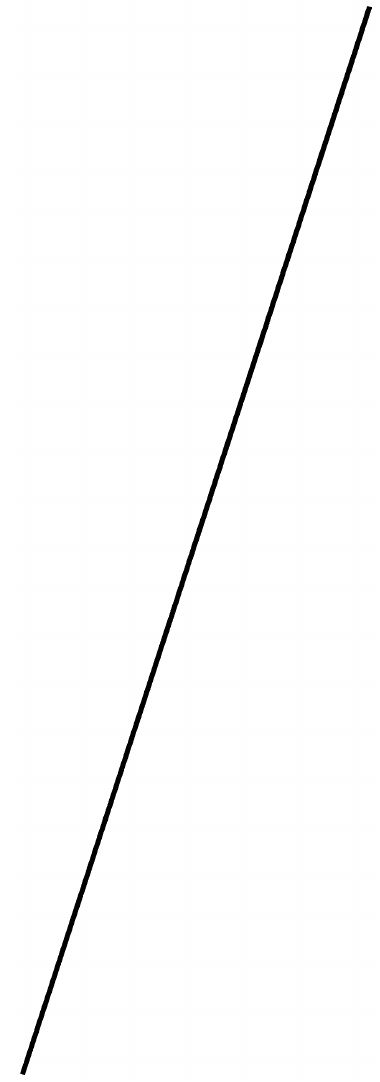}}\caption{The periodgon for $k=5$, $\theta=\pi/10$ and increasing nonzero values of $s$.}\label{fig:periodgon_K5}\end{center}\end{figure}

\begin{figure}\begin{center} 
\subfigure[]{\includegraphics[height=4cm]{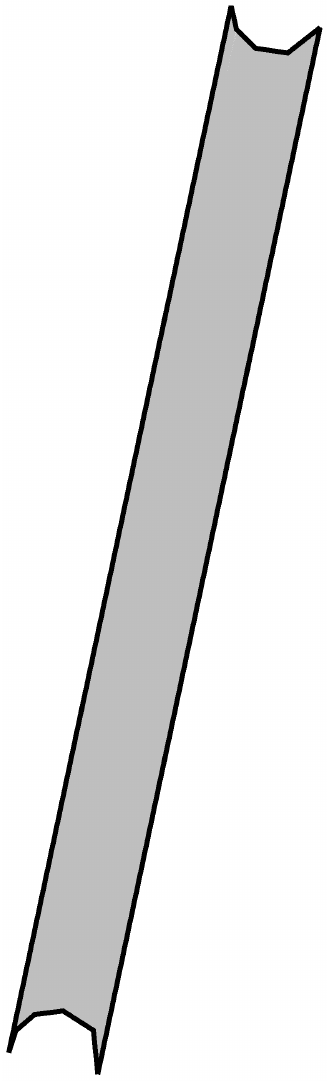}}\qquad\subfigure[]{\includegraphics[height=4cm]{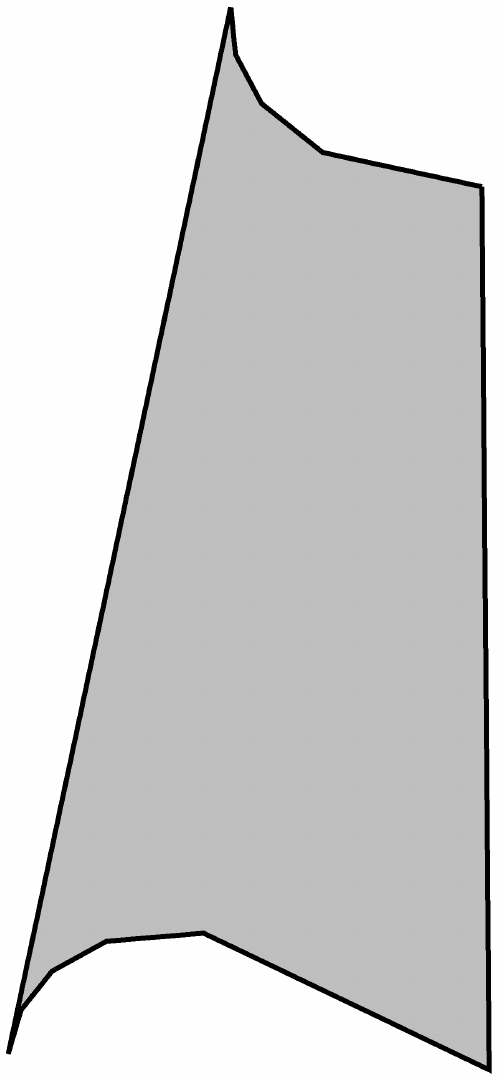}}\qquad\subfigure[]{\includegraphics[height=4cm]{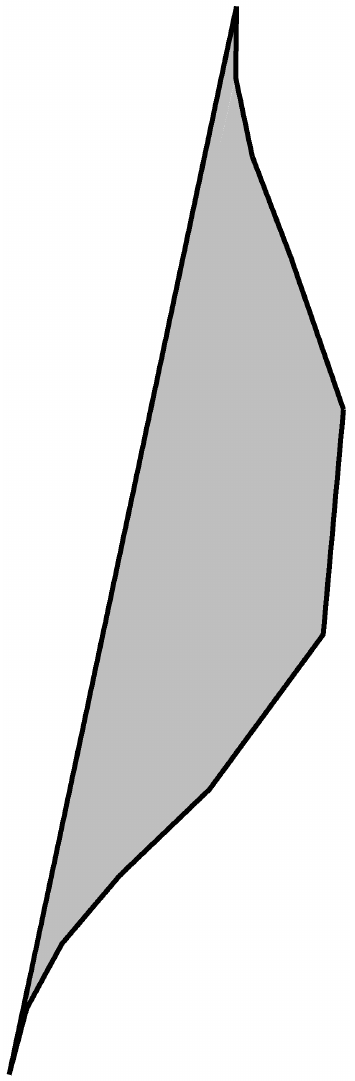}}\qquad\subfigure[]{\includegraphics[height=4cm]{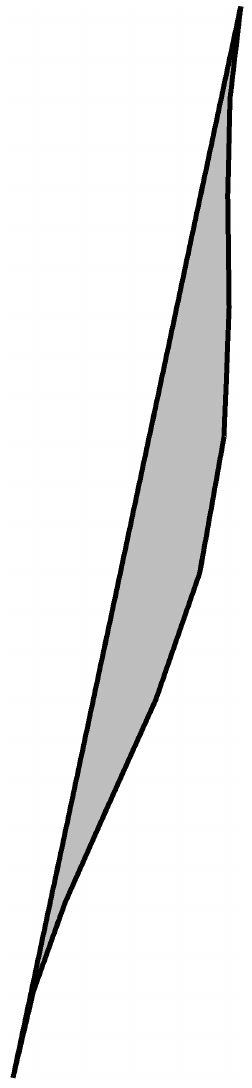}}\caption{The periodgon for $k=10$, $\theta=\pi/15$ and increasing  nonzero values of $s$.}\label{fig:periodgon_K10}\end{center}\end{figure}

\section{The bifurcation diagram of \eqref{vector_field}}\label{sec:bif_diag} 

\begin{theorem} The bifurcation diagram of \eqref{vector_field} consists of
\begin{enumerate} 
\item Real codimension 1 bifurcations of homoclinic loops. These occur when exactly two vertices of the periodgon can be joined by a horizontal segment inside the closed periodgon.  \item Real codimension 2 bifurcations of parabolic points for $\eps_0=0$ and $\Delta=0$, where $\Delta$ is given in \eqref{Delta}.
\item The only bifurcations of higher order are intersections of bifurcations of the two previous types. 
\item All bifurcations occur on ruled hypersurfaces, surfaces or curves  invariant under $(\eps_1, \eps_0)\mapsto (r^{k-1}\eps_1, r^k\eps_0)$ for $r>0$. 
\item On $\S^3$, the boundaries of the codimension 1  bifurcation surfaces are either higher order homoclinic loops or pieces of the curves $\eps_0=0$ and/or $\Delta=0$.
\end{enumerate} \end{theorem}

\begin{remark}
Under Conjecture~\ref{conj:periodgon} the periodgon is planar. In the general case, it is a simply connected region of a translation surface. A  homoclinic loop occurs precisely when  two vertices of the periodgon can be joined by a horizontal segment inside that translation surface.\end{remark}

\section{The bifurcation diagram of a generic $2$-parameter unfolding of a parabolic point of codimension $k$ preserving the origin}\label{sec:germ_vf} 

In this section we consider a two-parameter unfolding preserving the origin of a germ of analytic vector field with a parabolic point at the origin, which we can, without loss of generality, suppose to be of the form  $\dot z = z^{k+1}+ Az^{2k+1} + o(z^{2k+1})$. Using the Weierstrass preparation theorem and a scaling in $z$, such an unfolding has the form $\dot z = zQ_\eta(z)(1+g_\eta(z))$, with $Q_\eta(z) = z^k+\sum_{j=0}^{k-1} b_j(\eta)z^j$ depending on $\eta= (\eta_1,\eta_0)\in(\C^2,0)$, and  $g_\eta(z)=g(\eta,z)=O(z)$. The unfolding is generic if 
\begin{equation} \left| \frac{\partial(b_1,b_0)}{\partial( \eta_1,\eta_0)}\right|\neq0.\end{equation}
Dividing $g(z)$ by $zQ_\eta(z)$ allows writing the vector field as
\begin{equation} \dot z = zQ_\eta(z)(1+ R_\eta(z) +zQ_\eta(z)h_\eta(z))\label{prepared}\end{equation} where $R_\eta(z)= \sum_{j=1}^k c_j(\eta) z^j$ is such that $c_j(0)=0$ for $j<k$. We change parameter to $\eps=(\eps_1,\eps_0)= (b_1(\eta),b_0(\eta))$ and we still note the polynomials by $Q_\eps$ and $R_\eps$. 

The vector field in the form \eqref{prepared} is called \emph{prepared}. In particular the eigenvalues at the singular points $z_0=0$ and $z_j$, $j=1, \dots, k$, depend only of the polynomials $Q_\eps$ and $R_\eps$. More precisely, 
\begin{equation}\begin{cases}\lambda_0= \eps_0\\
\lambda_j= z_jQ_\eps'(z_j)(1+R_\eps(z_j)).\end{cases}\end{equation} Note that $1+R_\eps(z_j)$ is close to $1$ for $\eps$ small. 

We want to study the phase portrait of \eqref{prepared} for $z$ in some small fixed disk $\D_r$ and all $\eps$ in a small polydisk $\D_\rho= \{\pl \eps\pl<\rho\}$. By taking a smaller $r$ it is always possible to suppose that the vector field is defined on $\partial \D_r$. The radius $r$ is chosen sufficiently small so that the vector field has the same behavior as $\dot z= z^{k+1}$ near $\partial \D_r$ (see Figure~\ref{fleur}). Similarly, $\rho$ is chosen sufficiently small so that the $k+1$ singular points bifurcating from the origin remain far from  $\partial \D_r$.

The bifurcation diagram has a conic like structure. This is seen by writing the parameters as
\begin{equation}
\eps= \left(-k\zeta^{k-1}(1-s)^{k-1} e^{-i(k-1)\alpha}, (k-1)\zeta^ks^ke^{i(\theta-k\alpha)}\right)\label{par_eps}\end{equation}
with $s\in[0,1]$, $\theta\in[0,2\pi]$, $\alpha\in[0,2\pi]$ and $\zeta\in [0,\rho)$. Then a rescaling $(z,t)\mapsto(\frac{z}{\zeta}, \zeta^kt)$ brings the system to the form
\begin{equation} 
\frac{dZ}{dt}=U(Z) +O(\zeta),\label{poly_perturb}\end{equation}
where $$U(Z)= Z\left(Z^k-k(1-s)^{k-1}e^{-i(k-1)\alpha}+(k-1)s^ke^{i(\theta-k\alpha)}\right),$$
i.e to a small pertubation of the system \eqref{3_par} studied above. However, the scaling transforms the domain $\D_r$ into $\D_{r/\zeta}$ which tends to $\C$ when $\zeta\to 0$.

\begin{figure}\begin{center}\includegraphics[width=4.5cm]{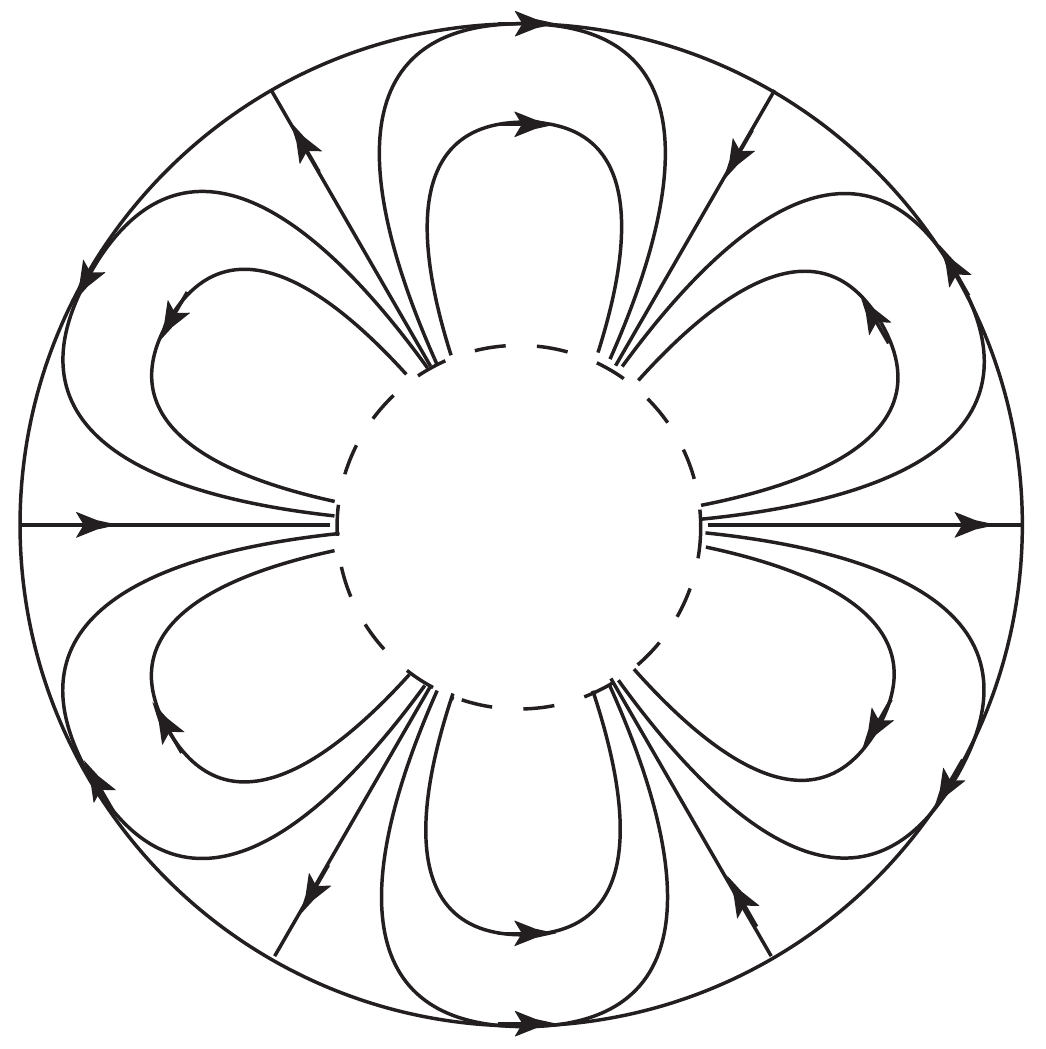}\caption{The phase portrait of \eqref{prepared} close to $\partial \D_r$ for $r$ small and $\rho\ll r$.}\label{fleur} \end{center}\end{figure}

\begin{proposition}\label{prop:parabolic} Bifurcations of parabolic points occur along 
\begin{enumerate} 
\item $\eps_0 = 0$\emph{:} the bifurcation has codimension 1 if $\eps_1\neq0$, and codimension $k$ otherwise;
\item $\Delta(\eps)=0$, where $\Delta(\eps)$ is the discriminant of $Q_\eps(z)$. Note that \begin{equation} 
\Delta(\eps)
=(-1)^{\lfloor\frac{k}2\rfloor}(k-1)^{k-1}k^k\left[\big(\tfrac{\eps_0}{k-1}\big)^{k-1}-\big(-\tfrac{\eps_1}{k}\big)^k+ o\left(\pl \eps\pl^{k(k-1)}\right)\right].\label{Delta_unfold}
\end{equation}
The bifurcation is of codimension $1$ as soon as $\eps\neq0$. \end{enumerate} \end{proposition}

It is possible to generalize the definition of periodgon to this case. 

\begin{definition}\cite{KR} Let $\dot z = \omega(z)$ be a holomorphic vector field in $\D_r$ with all singular points simple. 
\begin{enumerate} 
\item The \emph{periodic domain in $\D_r$} of a singular point $z_j$ is the union of the periodic trajectories surrounding $z_j$ for the rotated vector field $\dot z = e^{i\arg\nu_j} \omega(z)$, where $\nu_j= \frac{2\pi i}{\lambda_j}$. Its boundary is tangent to $\D_r$ (see Figure~\ref{Gen_periodic_domains} for some periodic domains in $\D_r$). The periodic domains of different points are disjoint. 
\item Let $t=\int \frac{dz}{\omega(z)}$. The \emph{ generalized periodgon} is the image in $t$-space of the complement in $\D_r$ of the union of the periodic domains in $\D_r$ (see Figure~\ref{gen_periodgon}). \end{enumerate}\end{definition}

\begin{figure}\begin{center}\subfigure[]{\includegraphics[width=3.6cm]{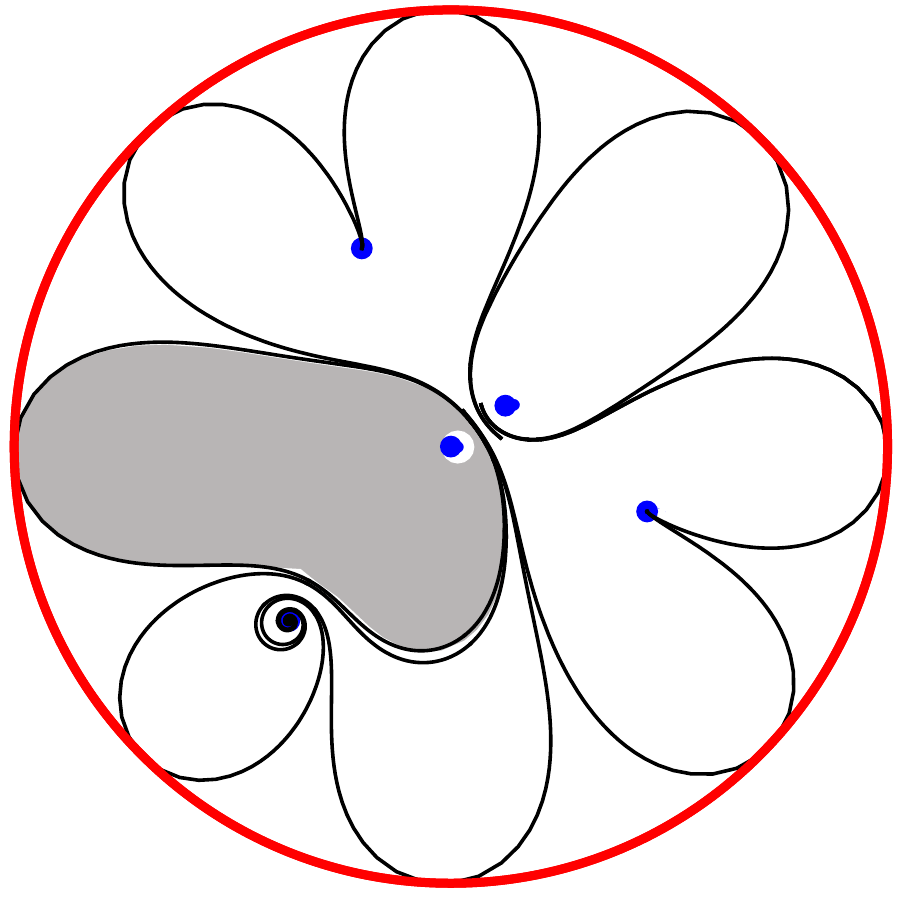}}\qquad\subfigure[]{\includegraphics[width=3.6cm]{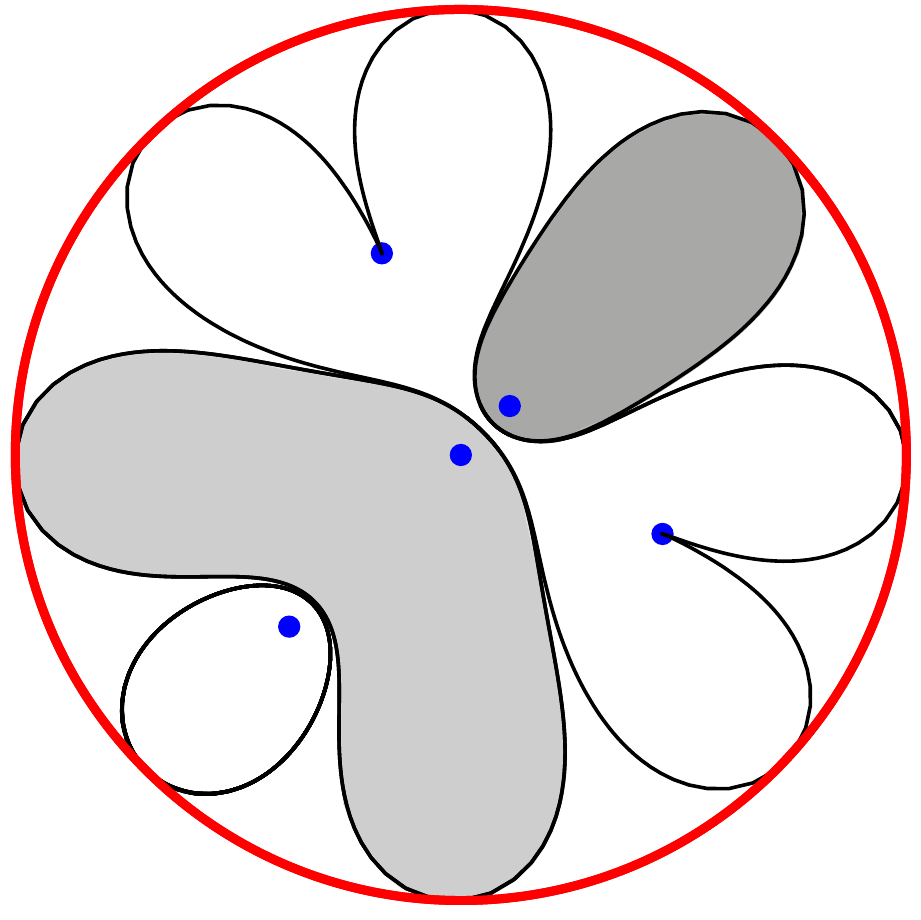}}\qquad\subfigure[]{\includegraphics[width=3.6cm]{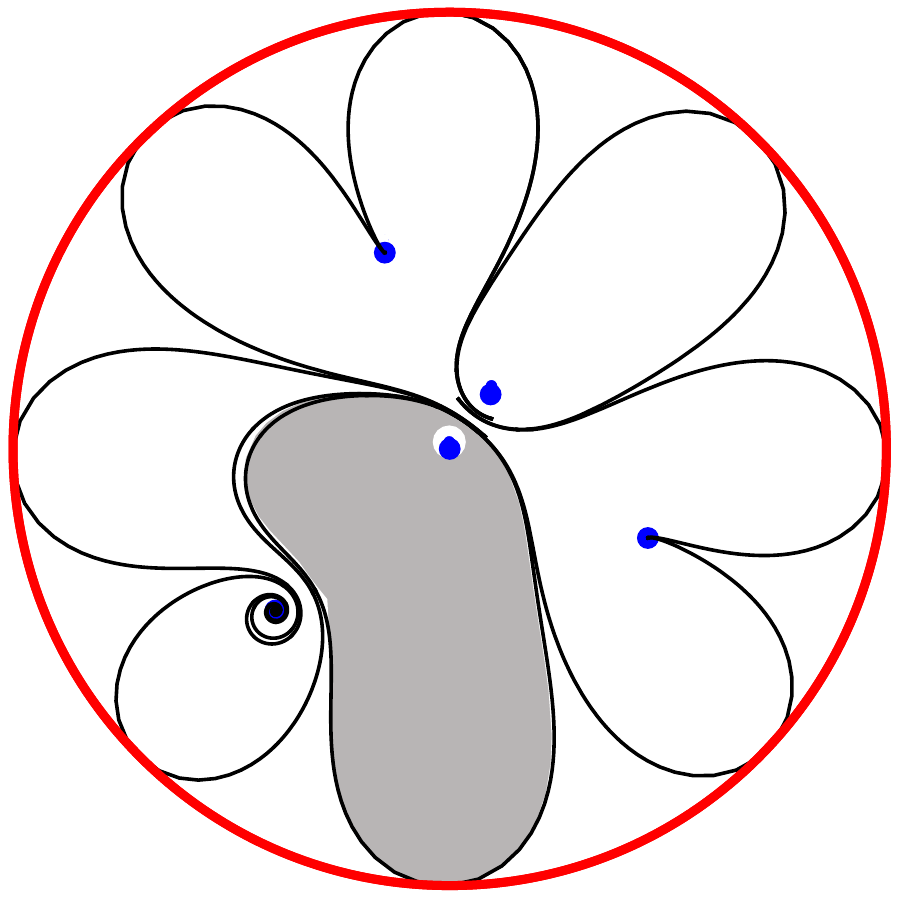}}\caption{Some periodic domains in $\D_r$. The case (b) shows a bifurcation of the generalized periodgon between (a) and (c). }\label{Gen_periodic_domains} \end{center}\end{figure}

\begin{figure}\begin{center}\subfigure[]{\includegraphics[width=5cm]{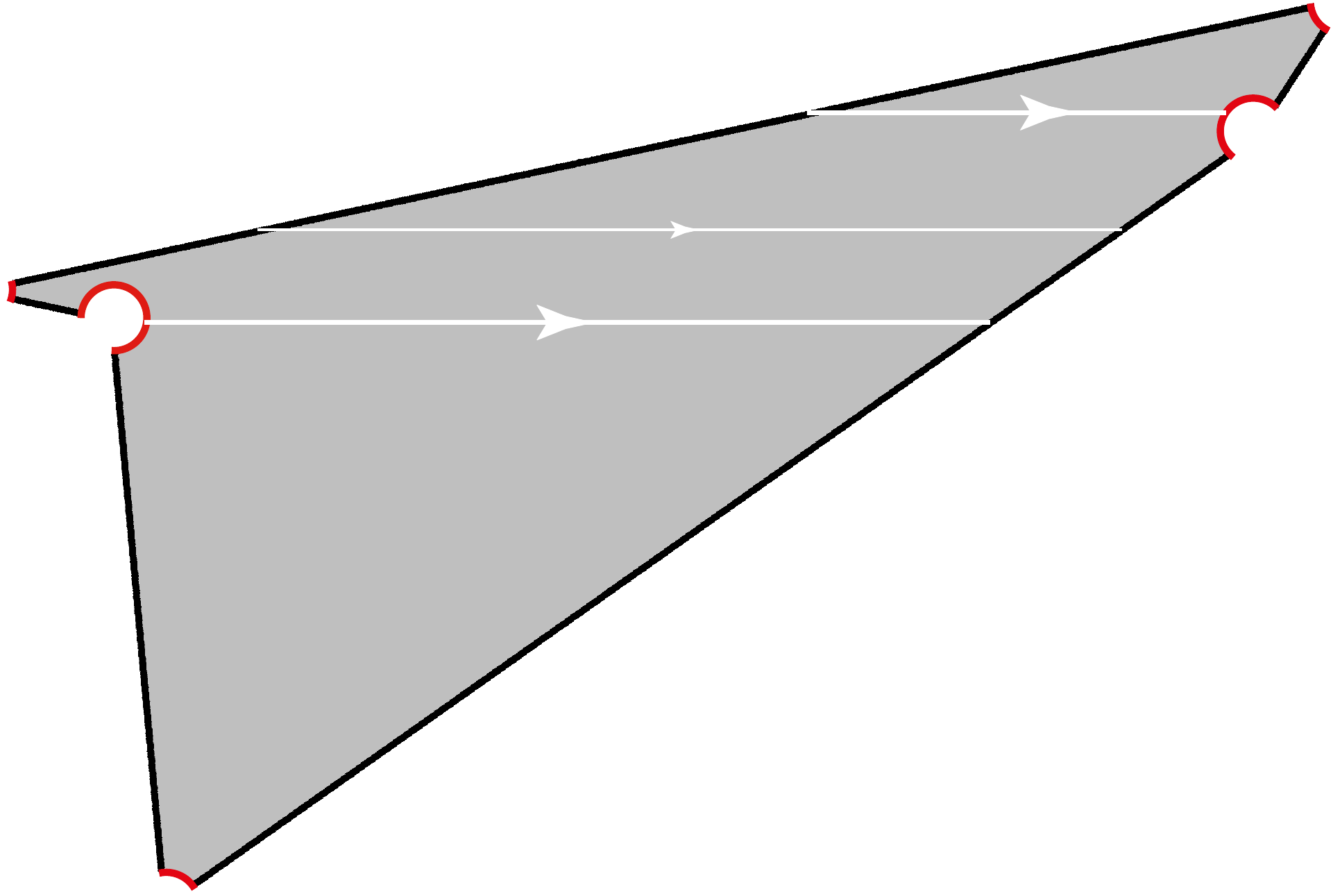}}\qquad\quad\subfigure[]{\includegraphics[width=5.5cm]{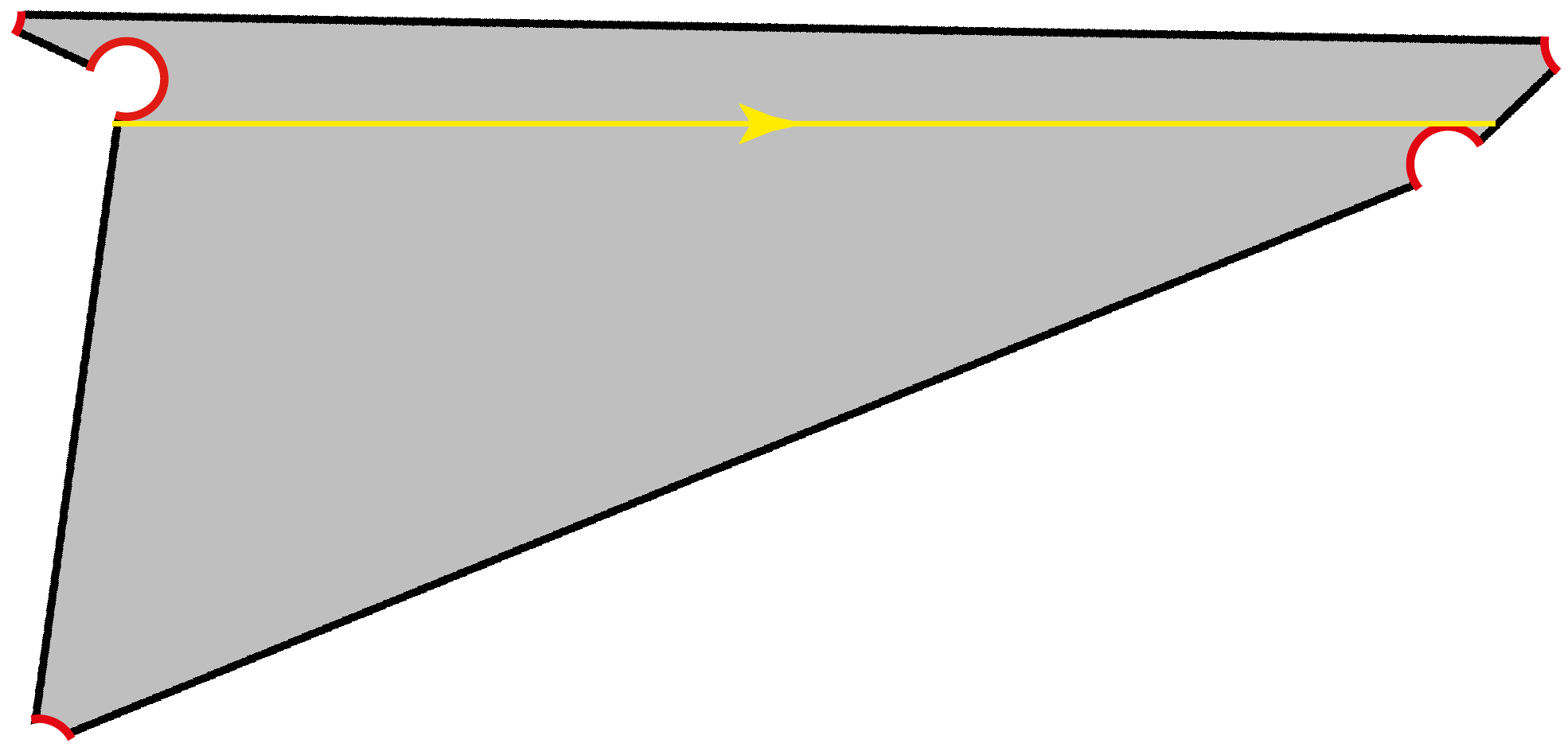}}\\
\subfigure[]{\includegraphics[width=5.5cm]{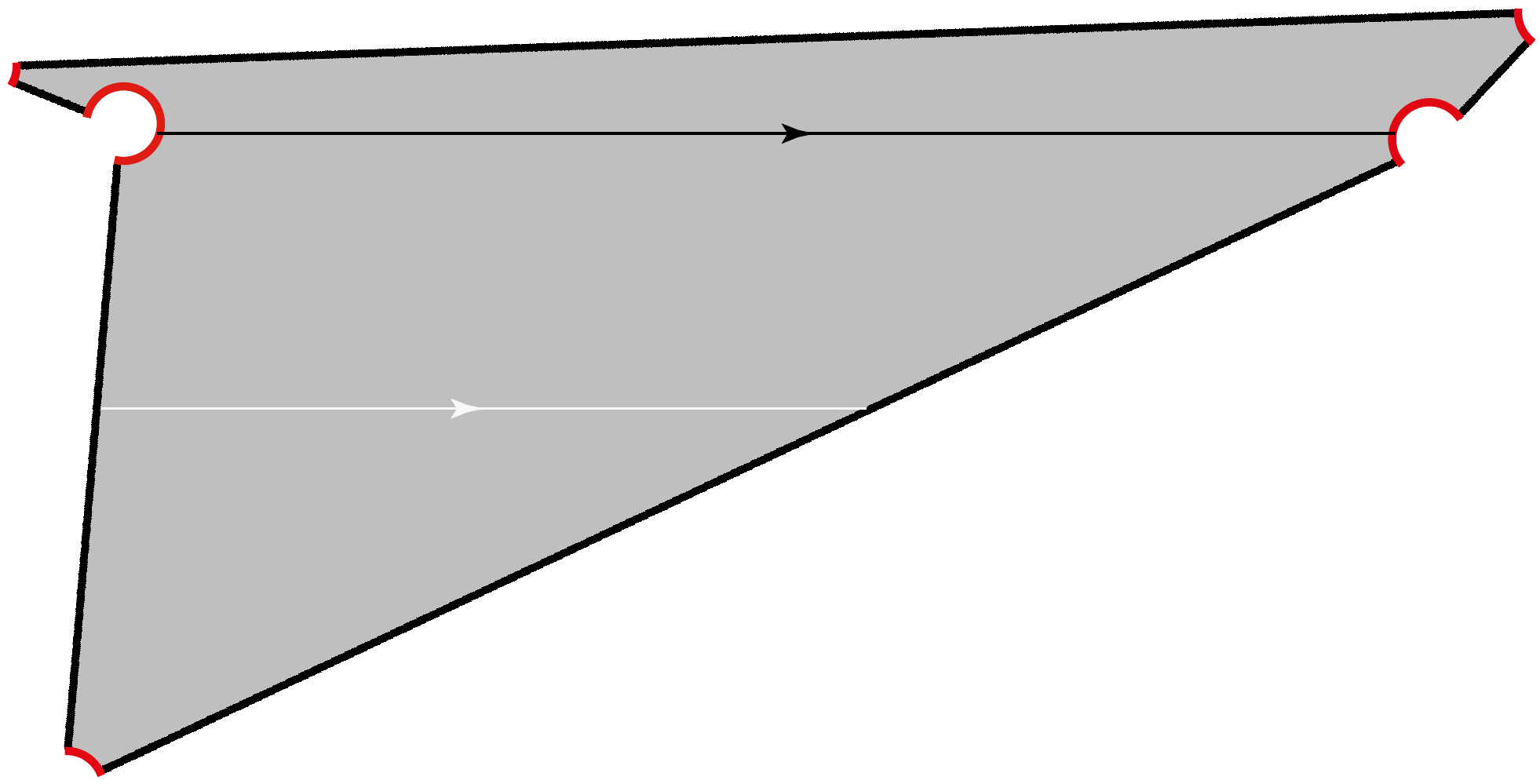}}\qquad\quad\subfigure[]{\includegraphics[width=5.5cm]{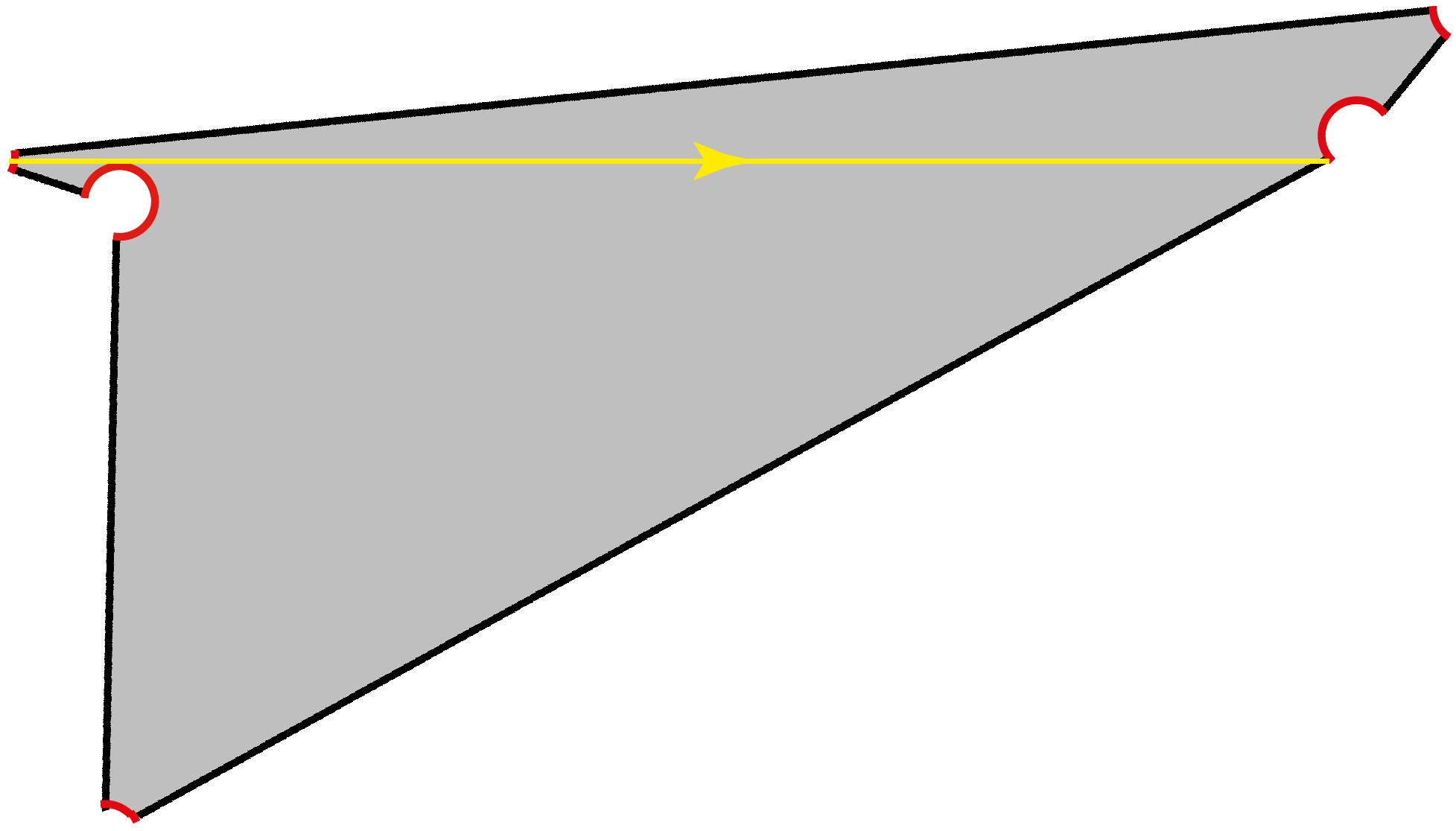}}\caption{A generalized periodgon for $k=4$ and different values of the rotational parameter $\alpha$. The red parts of the boundary are images in $t$-space of arcs of $\partial \D_r$. In (c) a separating trajectory in black. The limit positions for separating trajectories are curves of double tangency with the boundary as in (b) and (d). }\label{gen_periodgon} \end{center}\end{figure}

For generic values of the parameters there are, apart from the singular points, three kinds of generic trajectories inside $\D_r$:
\begin{enumerate}
\item Trajectories crossing $\partial \D_r$ with $\alpha$- or $\omega$-limit at a singular point (thick white trajectories in Figure~\ref{gen_periodgon}(a)); 
\item Trajectories with $\alpha$- and $\omega$-limit at a singular point (thin white trajectories in Figure~\ref{gen_periodgon}(a) and (c));
\item \emph{Separating trajectories} entering and exiting $\D_r$ (black trajectory in Figure~\ref{gen_periodgon}(c)). 
\end{enumerate} 
Non generic trajectories will have multiple tangency with the boundary $\partial \D_r$ (yellow trajectories in Figure~\ref{gen_periodgon}(b) and (d)). Several of these trajectories replace the homoclinic loops in the bifurcation diagram. 
 
 \begin{theorem} The bifurcation diagram of \eqref{prepared} consists of:
 \begin{enumerate} 
 \item Bifurcations of parabolic points as described in Proposition~\ref{prop:parabolic}. 
 \item Bifurcations of multiple tangencies of trajectories with the boundary $\partial \D_r$. Generically these are double tangencies. The corresponding bifurcation surfaces either  limit regions in the parameter space in which there exist separating trajectories (see Figure~\ref{gen_periodgon}) or regions where the generalized periodgon changes shape (see Figure~\ref{Gen_periodic_domains}(b)). \end{enumerate}\end{theorem}
 
 \noindent{\bf Question:} It would be interesting to identify a unique normal form in which the parameters are uniquely defined (canonical). This was done in \cite{KR} when we drop the constraint that the origin is fixed in the unfolding. So far, we have not been able to find such a normal form  in which the origin would be fixed. Such a normal form would lead to a classification theorem of germs of unfoldings \eqref{prepared} under conjugacy.

\section*{Acknowledgements} The author is grateful to Martin Klime\v{s} for stimulating discussions.

\end{document}